\documentclass[reqno, 11pt]{amsart}

\usepackage{amsfonts, amsthm, amsmath, amssymb}
\usepackage{hyperref}
\hypersetup{colorlinks=false}
\usepackage[all]{xy}

\usepackage[margin=3.5cm]{geometry}

\usepackage{helvet}

\RequirePackage{mathrsfs} \let\mathcal\mathscr

\numberwithin{equation}{section}

\newtheorem{theorem}{Theorem}[section]
\newtheorem{lemma}[theorem]{Lemma}

\newtheorem{corollary}[theorem]{Corollary}

\theoremstyle{definition}
\newtheorem*{ack}{Acknowledgements}

\newtheorem{definition}[theorem]{Definition}

\theoremstyle{remark}
\newtheorem{assumption}{Assumption}

\renewcommand{\d}{\mathrm{d}}
\renewcommand{\phi}{\varphi}

\newcommand{\0}{\mathbf{0}}

\newcommand{\PP}{\mathbb{P}}

\newcommand{\ZZ}{\mathbb{Z}}

\newcommand{\QQ}{\mathbb{Q}}
\newcommand{\RR}{\mathbb{R}}
\newcommand{\CC}{\mathbb{C}}

\newcommand{\cQ}{\mathcal{Q}}

\newcommand{\cW}{\mathcal{W}}

\newcommand{\XX}{\boldsymbol{X}}
\newcommand{\YY}{\boldsymbol{Y}}

\renewcommand{\leq}{\leqslant}

\renewcommand{\geq}{\geqslant}

\renewcommand{\bar}{\overline}

\newcommand{\ma}{\mathbf}

\newcommand{\m}{\mathbf{m}}

\newcommand{\x}{\mathbf{x}}
\newcommand{\y}{\mathbf{y}}
\renewcommand{\c}{\mathbf{c}}

\renewcommand{\v}{\mathbf{v}}
\renewcommand{\u}{\mathbf{u}}
\newcommand{\U}{\mathbf{U}}

\renewcommand{\b}{\mathbf{b}}
\renewcommand{\a}{\mathbf{a}}
\renewcommand{\k}{\mathbf{k}}
\newcommand{\h}{\mathbf{h}}

\renewcommand{\r}{\mathbf{r}}
\renewcommand{\t}{\mathbf{t}}
\newcommand{\w}{\mathbf{w}}

\newcommand{\fo}{\mathfrak{o}}
\newcommand{\fa}{\mathfrak{a}}
\newcommand{\fb}{\mathfrak{b}}
\newcommand{\fc}{\mathfrak{c}}
\newcommand{\fd}{\mathfrak{d}}
\newcommand{\fe}{\mathfrak{e}}

\newcommand{\fh}{\mathfrak{h}}

\newcommand{\fp}{\mathfrak{p}}
\newcommand{\fq}{\mathfrak{q}}

\newcommand{\ve}{\varepsilon}
\newcommand{\e}{\psi}

\newcommand{\bla}{\boldsymbol{\lambda}}

\newcommand{\bxi}{\boldsymbol{\xi}}

\DeclareMathOperator{\rank}{rank}

\DeclareMathOperator{\Gal}{Gal}

\DeclareMathOperator{\tr}{Tr}
\DeclareMathOperator{\nm}{Nm}
\DeclareMathOperator{\n}{N}

\newcommand{\cH}{\mathcal{H}}

\DeclareMathOperator{\diag}{Diag}

\newcommand{\gfb}{{{}^{{G}}\mathfrak{b}}}
\newcommand{\gfp}{{}^{{G}}\mathfrak{p}}

\newcommand{\gfq}{{}^{{G}}\mathfrak{q}}

\DeclareMathOperator{\Mod}{mod} 
\renewcommand{\bmod}[1]{\,(\Mod{#1})}
\newcommand{\No}{N_{W}(F,N;P)}

\newcommand{\sigstar}{\sideset{}{^*}\sum_{\sigma\bmod{\mathfrak{b}}}}
\newcommand{\ord}{\mathrm{ord}}
\newcommand{\starsum}{\sideset{}{^*}\sum}

\def\bfa{{\mathbf a}}

 \def\bfe{{\mathbf e}}

\def\bfh{{\mathbf h}}

\def\bfm{{\mathbf m}}

\def\bfu{{\mathbf u}}
\def\bfv{{\mathbf v}}

\def\bfx{{\mathbf x}}

\def\calH{{\mathcal H}}
\def\cH{{\mathcal H}}

\def\calO{{\mathfrak o}}

\def\calS{{\mathcal S}}

\newcommand{\modu}{\mbox{ mod }}

\def\calb{{\mathfrak{b}}}
\def\cald{{\mathfrak{d}}}
\def\Gb{{{}^G\mathfrak{b} }}
\def\cF{{\mathcal{F}}}

\def\alp{{\alpha}} 
  
\def\gam{{\gamma}} 

\def\Gam{{\Gamma}}
 \def\Del{{\Delta}}

\def\kap{{\kappa}}

\def\btau{{\boldsymbol \tau}}

\def\grb{{\mathfrak b}}

\def\grp{{\mathfrak p}}

\def\grp{{\mathfrak p}}

\def\R{{\mathbb R}}
\def\Z{{\mathbb Z}}

\def\Q{{\mathbb Q}}
\def\ome{{\omega}}
\DeclareMathOperator{\Tr}{Tr}

\usepackage[usenames,dvipsnames]{color}

\title[Generalised quadratic forms]{Generalised quadratic forms over totally real number fields}

\author{Tim Browning}
\address{IST Austria\\
Am Campus 1\\
3400 Klosterneuburg\\
Austria}
\email{tdb@ist.ac.at}

\author{Lillian  B. Pierce}
\address{Department of Mathematics\\ 
Duke University\\ Durham NC 27708\\ USA}
\email{pierce@math.duke.edu}

\author{Damaris Schindler}
\address{G\"ottingen University\\
Bunsenstraße 3--5\\
37073 G\"ottingen\\
Germany}
\email{damaris.schindler@mathematik.uni-goettingen.de}

\date{\today}

\begin{document}

\begin{abstract}
We introduce a new class of generalised quadratic forms over totally real number fields,  
which is rich enough to  capture the arithmetic of arbitrary systems of quadrics over the rational numbers.  We explore this connection through a version of the  Hardy--Littlewood circle method over number fields. 
\end{abstract}

\maketitle

\setcounter{tocdepth}{1}
\tableofcontents

\section{Introduction}

The study of quadratic forms over number fields  is a rich and highly developed area of mathematics.
Let  $K$ be a number field of degree $d\geq 2$ over $\QQ$ 
and let 
$$
Q(X_1,\dots,X_n)=
\sum_{1\leq i,j\leq n} c_{i,j} X_i X_j
$$
be a non-singular quadratic form, with symmetric coefficients $c_{i,j}\in \fo_K$. 
For given $N\in \fo_K$, it is very natural to ask about the solubility of 
$$
Q(x_1,\dots,x_n)=N,
$$
with $x_1,\dots,x_n\in \fo_K$. 
If $n\geq 4$, a number field version of the Hardy--Littlewood circle method is capable of establishing the Hasse principle for  these equations. When $n\geq 5$ this follows from work of Skinner \cite{skinner2}, and for $n=4$ it is carried out by 
Helfrich in a 2015 PhD thesis   \cite{helfrich}. 

In this paper we shall introduce the  notion of a 
{\em generalised quadratic form} over  $K$ and ask about the Hasse principle in this new setting. 
We shall always assume that $K/\QQ$ is a Galois  extension 
of degree $d$ that is totally real. 
(Our methods can handle arbitrary number fields, but doing so causes extra notational complexity and gives no new insight into the arithmetic of generalised quadratic forms.)
We may now make the following definition.

\begin{definition}\label{def:gq}
Let $n\geq 2$. A  {\em generalised quadratic form} is given by 
$$
F(X_1,\dots,X_n)=
\sum_{1\leq i,j\leq n} \sum_{\tau,\tau'\in \Gal(K/\QQ)}c_{i,j,\tau,\tau'} X_i^{\tau} X_j^{\tau'},
$$
for  symmetric  coefficients $c_{i,j,\tau,\tau'}=c_{j,i,\tau',\tau}\in \fo_K$. 
\end{definition}

We will be interested in the set of $(x_1,\dots,x_n)\in \fo_K^n$ for which 
$$
F(x_1,\dots,x_n)=N,
$$ 
for given $N\in \fo_K$, in which case  $x_i^\tau$ should be interpreted as the conjugate of $x_i$ under $\tau\in \Gal(K/\QQ)$. 
 Definition \ref{def:gq} encompasses 
standard  integral quadratic forms over $\fo_K$ and forms defined using  norms and traces. 
For example,  let $\tr_{K/\QQ,H}:K\to K$ be the {\em partial trace}, defined
via
$\tr_{K/\QQ,H}(u)=\sum_{\tau\in H} u^\tau$, for any subset $H\subset \Gal(K/\QQ)$.
Then, a natural generalisation of the question about representing  elements of $\fo_K$ as a sum of squares is to ask about the existence of  $\x\in \fo_K^n$ such that
\begin{equation}\label{eq:partial_trace}
\Tr_{K/\QQ,H} (x_1^2)+\dots+\Tr_{K/\QQ,H} (x_n^2)=N,
\end{equation}
for given $N\in \fo_K$ and a given subset  $H\subset \Gal(K/\QQ)$.


The coefficients of a generalised quadratic form $F(X_1,\dots,X_n)$ form a 
$dn\times dn$ matrix  $\mathbf{M}=(c_{i,j,\tau,\tau'})_{(i,\tau)\times (j,\tau')}$. 
In the generic setting we might expect this matrix to have full rank, but there are many cases of interest where the rank is  much  smaller. For example,   standard 
 quadratic forms produce a  coefficient matrix $\mathbf{M}$, which after reordering rows and columns, contains  a  $n\times n$ block  matrix in the upper left corner  and has  zeros everywhere else. 
Our methods break down  in the completely generic situation, and so our  interest in this paper lies at the opposite end of the spectrum, in which the rank of $\mathbf{M}$ is not much bigger than $n$.

Let  $W:(K\otimes_\QQ \RR)^n\to \RR_{\geq 0}$
be a  smooth weight function, whose precise construction 
is deferred until \S \ref{s:circle}.
Our main results will comprise of asymptotic formulae for  sums of the shape
$$
 N_{W}(F,N;P) = 
 \sum_{\substack{\x\in \fo_K^n\\ F(\x)=N}} W(\x/P) ,
$$
as $P\rightarrow \infty$, for given $N\in \fo_K$ and suitable generalised quadratic forms $F$. 
When  $N=0$, we  shall simply 
write $N_{W}(F;P)=N_{W}(F,0;P)$.

\subsection{Homogeneous setting}

Of particular interest is the case $N=0$, which we now assume.  For standard quadratic forms $Q\in \fo_K[X_1,\dots,X_n]$, 
studying  non-trivial zeros of $Q$ over   $\fo_K$ is equivalent to studying 
 $K$-rational points on the smooth quadric 
 $X\subset \PP_K^{n-1}$ cut out by $Q=0$. This,  in turn, can be accessed via the 
 Weil restriction (or restriction of scalars). 
The Weil restriction $R_{K/\QQ} X$ is  an algebraic variety whose set of $\QQ$-points is canonically in bijection with the $K$-rational points of $X$. In the setting where $Q\in \fo_K[X_1,\dots,X_n]$
is a non-singular quadratic form, 
 the Weil restriction $R_{K/\QQ}X$ is a smooth complete intersection of $d$ quadrics in $\PP_\QQ^{dn-1}$, all of which are defined over $\QQ$.
However,  the set of complete intersections  that arise in this way is a 
 very limited subset of the family of all smooth codimension $d$ complete intersections of quadrics over $\QQ$ 
 in $\PP_\QQ^{dn-1}$.  Our first result shows that, 
 after Weil restriction, 
 the space of generalised quadratic forms  is rich enough to capture the arithmetic over $\QQ$ of  
arbitrary codimension $d$ complete  intersections of quadrics in $\PP_\QQ^{M-1}$, provided that $d\mid M$.  

Let $F(X_1,\dots,X_n)$ be a generalised quadratic form and let $\omega_1,\dots,\omega_d $ be a $\ZZ$-basis for $\fo_K$. Any element  $\x\in \fo_K^n$ can be written $\x=\omega_1\u_1+\cdots+\omega_d\u_d$ for $(\u_1,\dots,\u_d)\in \ZZ^{dn}$.  Taking the Weil restriction corresponds to writing down a set of 
quadratic forms $Q_1,\dots,Q_d\in \ZZ[\U_1,\dots,\U_d]$, in $dn$ variables, such that 
\begin{equation}\label{eq:F-to-Q}
F(X_1,\dots,X_n)=\sum_{1\leq i\leq d}{\omega_i}Q_i(\U_1,\dots,\U_d).
\end{equation}
We henceforth call $\{Q_1,\dots,Q_d\}$ the {\em descended system}.
We shall prove the following result in \S \ref{sec:descended}.

\begin{theorem}\label{t:systems}
Let $K/\QQ$ be  a Galois extension of degree $d$. Then there is a  bijection between the space  of generalised quadratic forms in $n$ variables over $K$ and systems of $d$ rational quadratic forms in $dn$ variables.
\end{theorem}

The   question of $\fo_K$-solubility for a generalised quadratic form therefore   becomes a question of $\ZZ$-solubility for the descended system and vice versa. It presents an intriguing challenge to 
gain insight into smooth 
codimension $d$ complete intersections of quadrics in 
$\PP_\QQ^{M-1}$ over $\QQ$ by  working with generalised quadratic forms.
It follows from work of Birch \cite{birch} that the usual Hardy--Littlewood asymptotic formula holds for 
systems of quadrics over $\QQ$, 
provided that $M>B+2d(d+1)$,
where $B$ is the affine dimension of the ``Birch singular locus'' of the descended system. (Note that one can take 
$B\leq d-1$ when the descended system is a smooth complete intersection.)
Breakthrough work of 
Rydin Myerson \cite{myerson} handles smooth codimension $d$ complete intersections of quadrics in $\PP_\QQ^{M-1}$ when $M\geq 9d$. 
The latter result is particularly significant, since it allows one to handle 
arbitrary generalised quadratic forms over $K$ in $n\geq 9$ variables, provided that the descended system defines a  smooth complete intersection of codimension $d$.

Our main results will concern a special class of generalised quadratic forms, 
in which only one non-trivial automorphism appears, and in which  the conjugated  variables separate completely  from the unconjugated variables.   These examples are chosen  to 
 represent a first step on the way to a fuller understanding of generalised quadratic forms, 
and yet exhibit enough features that   make them untreatable by other methods. 
In the light of  Theorem \ref{t:systems}, 
a complete understanding of  generalised quadratic forms must lie rather  deep. 

Let $Q\in \fo_K[X_1,\dots,X_n]$ and $R\in \fo_K[X_1,\dots,X_m]$ be  quadratic forms in $n$ and $m$ variables, respectively, for $1\leq m\leq n$.
The generalised quadratic forms we shall treat take the shape
\begin{equation}\label{eq:special-0}
F(X_1,\dots,X_n)=Q(X_1, \dots, X_n) +R(X_1^{\tau},\dots, X_m^{\tau}),
\end{equation}
for a fixed non-trivial automorphism $\tau\in \Gal(K/\QQ)$.
Let $\rho_1,\dots,\rho_d$ be the $d$ distinct embeddings of $K$ into $\RR$, where we recall that $K$ is totally real.  For each $1\leq l\leq d$, we define $l_\tau$ through the relation
\begin{equation}\label{eq:l_tau}
\rho_{l_\tau}\tau=\rho_l.
\end{equation}
Suppose that $\mathbf{A}$ is the $n\times n$ symmetric matrix defining $Q$ and that  $\mathbf{B}$ is the $n\times n$ symmetric matrix given by the condition that its upper left $m\times m$ submatrix defines $R$, with  all other entries equal to $0$. 
For any  $1\leq l\leq d$, we shall  write $\mathbf{A}^{(l)}$ and $\mathbf{B}^{(l)}$ for the $l$-th embeddings of $\mathbf{A}$ and $\mathbf{B}$, respectively. 
We make the following key  hypotheses about $\mathbf{A}$ and $\mathbf{B}$.

\begin{assumption}\label{ass1}
Assume that the descended system 
$$Q_1(\U_1,\dots,\U_d)=\dots= Q_1(\U_1,\dots,\U_d)=0$$ 
has codimension $d$ in $\PP^{dn-1}$.
Furthermore, 
assume that  $\det \mathbf{A}\neq 0$ and that  the upper left  $m\times m$  submatrix of $\mathbf{B}$ is non-singular. 
\end{assumption}

Our first result deals with the special case $m=1$.

\begin{theorem}\label{t:non-diag-1}
Let $K/\QQ$ be a totally real Galois extension of degree $d\geq 2$.
Suppose that  $m=1$ and that 
Assumption \ref{ass1} holds. Assume that 
$\det (\mathbf{A}^{(l)}+t\mathbf{B}^{(l_\tau)})$ is a constant polynomial in $t$, for each $1\leq l\leq d$, where $l_\tau$ is defined via
\eqref{eq:l_tau}.
Let  $n\geq 6$ and assume that 
 the descended system has non-singular points everywhere locally. 
Then there exist constants $c>0$ and  $\Delta>0$ such that 
$$
 N_{W}(F;P) = cP^{(n-2)d} +O(P^{(n-2)d-\Delta}).
$$
\end{theorem}

The implied constants in our work are always allowed to depend on $K$ and $F$.
The generalised quadratic form
$
2X_1X_2 + a (X_1^{\tau})^2 + \tilde{Q}(X_3,\ldots, X_n)
$
meets the hypotheses of the theorem, for example, where   $\tilde Q\in \fo_K[X_3,\ldots, X_n]$ 
is a non-singular quadratic form
and $a\in \fo_K$ is non-zero.

We are also able to prove an asymptotic formula for  $N_{W}(F;P)$ for arbitrary $m\geq 1$, provided we make additional assumptions about the matrices $\mathbf{A}$ and $\mathbf{B}$.

\begin{assumption}\label{ass2}
For all $1\leq l\leq d$, 
assume that 
$
\rank (\mathbf{A}^{(l)}+t\mathbf{B}^{(l_\tau)})\geq n-1$, for all $t\in \mathbb{R}$,
where $l_\tau$ is defined via \eqref{eq:l_tau}.
\end{assumption}

\begin{assumption}\label{ass3}
For all $1\leq l\leq d$, 
assume that  $\det (\mathbf{A}^{(l)}+t\mathbf{B}^{(l_\tau)})$ has degree at least $m-1$,
viewed as a  polynomial in $t$. 
\end{assumption}

When $m=1$ and 
$\det (\mathbf{A}^{(l)}+t\mathbf{B}^{(l_\tau)})$ has degree exactly  $0$ in Assumption \ref{ass3}, we see that 
Assumption \ref{ass2} is implied by Assumption \ref{ass1}, since then 
$\rank (\mathbf{A}^{(l)}+t\mathbf{B}^{(l_\tau)})=
\rank (\mathbf{A}^{(l)})=n$.  For general $m\geq 1$, 
Assumption \ref{ass2} is similar to one that is commonly made in the study of pairs of quadratic forms. Indeed, suppose one is given  two matrices $A,B\in M_{n\times n}(L)$ over an algebraically closed field $L$ of characteristic not equal to $2$, with associated quadratic forms $Q_A$ and $Q_B$. It follows from 
Reid's thesis \cite[Prop.~2.1]{reid} that 
the rank of any element in the pencil $\lambda A+\mu B$, with $(\lambda,\mu)\neq (0,0)$, is never smaller than $n-1$, provided the intersection  $Q_A=Q_B=0$
is non-singular as a projective variety, and of the expected dimension. In our situation, by contrast,  we only look at the pencil $\mathbf{A}^{(l)}+t\mathbf{B}^{(l_\tau)}$, since  the matrix $\mathbf{B}^{(l_\tau)}$  has rank $m$
by construction. (We shall relate this situation to the properties of an appropriate singular locus in Lemma \ref{lem:reid} below.)

We are now ready to reveal our main  result in the homogeneous setting.

\begin{theorem}\label{t:non-diag}
Let $K/\QQ$ be a totally real Galois extension of degree $d\geq 2$.
Suppose that Assumptions \ref{ass1}--\ref{ass3} hold, and that
$
n> 
3m+4-4m/d.
$
Assume  that 
 the descended system has non-singular points everywhere locally. 
Then there exist constants $c>0$ and  $\Delta>0$ such that 
$$
 N_{W}(F;P) = cP^{(n-2)d} +O(P^{(n-2)d-\Delta}).
$$
\end{theorem}

On taking $m=1$, we note that this result subsumes Theorem 
\ref{t:non-diag-1} when $n\geq 7$.
If one makes further assumptions on $Q$ one can do even better. Suppose, for example, that 
the last $n-m$ variables split off from $Q$, so that 
$$
Q(X_1,\dots,X_n)=Q_1(X_1,\dots,X_m)+Q_2(X_{m+1},\dots,X_n),
$$ 
for quadratic forms $Q_1$ and $Q_2$ over $\fo_K$.
Then it seems likely that a classical version of the circle method can be employed.
On summing trivially over the first $m$-variables of the associated  exponential sums, one would be left with handling an exponential sum in $n-m$ variables involving  $Q_2.$ If $Q_2$ has rank at least $5$, then Skinner's treatment over number fields \cite{skinner2} 
would yield the necessary saving. 
This  ought to allow $n\geq m+5$ in the statement of  Theorem \ref{t:non-diag} if 
$Q(0,\dots,0,X_{m+1},\dots,X_n)$ has rank at least $5$.

\subsection{Inhomogeneous setting}

We now assume that $N\in \fo_K$ is non-zero. Then we may write $N=\omega_1 N_1+\cdots+\omega_dN_d$, where $N_1,\dots,N_d\in \ZZ$ are not all zero. 
We shall 
 henceforth call $\{Q_1-N_1,\dots,Q_d-N_d\}$ the {\em shifted descended system},
where  $Q_1,\dots,Q_d$ are obtained from $F$ via \eqref{eq:F-to-Q}, continuing to call 
$\{Q_1,\dots,Q_d\}$ the associated  {\em descended system}.

Our next result demonstrates that sharper results are available if $N\neq 0$ and  $Q, R$ are both diagonal.
Suppose that 
\begin{equation}\label{eq:special}
F(X_1,\dots,X_n)=a_1X_1^2+\cdots+a_nX_n^2+\sum_{i=1}^m b_i (X_i^{\tau})^2,
\end{equation}
for $1\leq m \leq n$ and non-zero $a_1,\dots,a_n,b_1,\dots,b_m\in \fo_K$, and 
where $\tau\in \Gal(K/\QQ)$ is a fixed non-trivial automorphism.
Taking $m=n$ and $a_i=b_i=1$ for $1\leq i\leq n$,  we are led to an instance of the partial trace problem in  \eqref{eq:partial_trace} with $H=\{\mathrm{id}, \tau\}$.
We will prove the following result.

\begin{theorem}\label{t:diag}
Let $K/\QQ$ be a totally real Galois extension of degree $d\geq 2$.
Assume that $N\in \fo_K$ is non-zero and that 
$
n\geq m+4.
$
Suppose that 
 the  descended system
 has codimension $d$ and a non-singular real point 
and that  the shifted descended system has non-singular points 
 over $\ZZ_p$ for every prime $p$.
Then there exist constants $c>0$ and  $\Delta>0$ such that 
$$
 N_{W}(F,N;P) = cP^{(n-2)d} +O(P^{(n-2)d-\Delta}).
$$
\end{theorem}

The implied constant in this result is allowed to depend on $N$, in addition to  $K$ and $F$.
In order to illustrate our result, take the quadratic number field $K=\QQ(\sqrt{2})$ in 
\eqref{eq:special} and assume that $a_1,\dots,a_n,b_1,\dots,b_m\in \ZZ$ are all non-zero.  Then 
it follows from Theorem \ref{t:diag} that our work treats the shifted descended system
\begin{align*}
\sum_{i=1}^m (a_i+b_i)u_i^2+2\sum_{i=1}^m (a_i+b_i) v_i^2 +\sum_{i=m+1}^na_i (u_i^2+2v_i^2)
&=N_1,\\
2\sum_{i=1}^m (a_i-b_i)u_iv_i+ 2\sum_{i=m+1}^na_i u_iv_i
&=N_2,
\end{align*}
when $n\geq m+4$ 
and 
 $N_1,N_2\in \ZZ$ are not both zero.

\subsection{Some words on the proof}

Let $F(X_1,\dots,X_n)$ be a generalised quadratic form defined over 
  $\fo_K$ and let $N\in \fo_K$. 
Our analysis of $\No$  relies on a   Fourier-analytic 
interpretation  of 
the indicator function
\begin{equation}\label{eq:deltaK}
\delta_K(\alpha) =
\begin{cases}
1,  & \mbox{if $\alpha = 0$}, \\
0,  & \mbox{if $\alpha\in \fo_K\setminus \{0\}$}.
\end{cases}
\end{equation}
Browning and Vishe \cite[Thm~1.2]{BV} have extended
to arbitrary number fields the smooth $\delta$-function technology of Duke--Friedlander--Iwaniec \cite{DFI}, as later refined by  Heath-Brown \cite{HB}. This will underpin the work in this paper, 
affording us the opportunity to extract non-trivial savings, in the spirit of Kloosterman's method, in the proof of Theorem \ref{t:diag}.  We will be led to  an expression for $N_W(F,N;P)$ in \eqref{eq:No1}, involving an infinite sum over non-zero integral ideals $\fb$. The next stage is to apply Poisson summation, but an obstacle arises from the fact that 
it is no longer possible to break into residue classes modulo $\fb$ for generalised quadratic forms $F$. 
 Instead we shall break into residue classes modulo a larger ideal $\gfb$, which is the least common multiple of the ideals $\fb^{\tau^{-1}}$, as $\tau$ ranges over the automorphisms that actually  occur in $F$.
Poisson summation then leads to the analysis of certain {\em exponential sums} 
$S_\fb(N;\m)$ and {\em oscillatory integrals} $I_\fb(N;\m)$, which are indexed by $\fb\subset \fo_K$ and suitable 
 vectors $\m\in K^n$.
While the treatment of 
$S_\fb(N;\m)$ is relatively standard, the main challenge is to understand 
$I_\fb(N;\m)$. When $F$ is a standard quadratic form, these integrals factorise into 
a product of $d$ oscillatory integrals,
one for each of the $d$ real embeddings of $K$. This reduces the problem to looking at oscillatory integrals 
over $\RR^n$.
For generic generalised quadratic forms, it seems very difficult to obtain the kind of  cancellation one needs for the method to go through, for the relevant  oscillatory integrals over $\RR^{dn}$.

We now summarise the contents of the paper. In \S \ref{sec:descended} we shall 
prove  Theorem~\ref{t:systems} by 
spelling out the connection between generalised quadratic forms over $K$ and descended systems over $\QQ$. 
In \S \ref{s:ant} we collect together some useful facts from algebraic number theory. 
The rest of the paper will be concerned with estimating the size of the counting function $\No$, 
as $P\to \infty$. In order to facilitate future investigation we shall present most of the arguments for arbitrary generalised quadratic forms in  \S \ref{s:circle}. Next, in \S \ref{s:hom} we shall specialise to the case \eqref{eq:special-0} and $N=0$, in order to 
deduce Theorems \ref{t:non-diag-1} and~\ref{t:non-diag}.  Finally, \S \ref{s:inhom} will deal with Theorem \ref{t:diag}, which pertains to the diagonal generalised quadratic form \eqref{eq:special} and $N\neq 0$.

\begin{ack}
The authors are grateful to Jayce Getz for asking questions that set this project in motion. 
T.B.  was
supported by FWF grant P~32428-N35 and by a grant from the Institute
for Advanced Study School of Mathematics.  L.B.P. was partially supported by NSF DMS-2200470
and DMS-1652173, and thanks the Hausdorff Centre for Mathematics for hosting research visits. 
\end{ack}

\section{Generalised quadratic forms and the descended system}\label{sec:descended}

In this section we shall prove Theorem \ref{t:systems}, by making explicit the correspondence between generalised quadratic forms $F$ and 
the descended system of $d$ quadratic forms over $\Q$ in $dn$ variables. Let $K/\QQ$ be a degree $d$ Galois number field, which (in this section only) need not be totally real.  
Assume that we are given a set of coefficients $(c_{i,j,\tau,\tau'})$ of a generalised quadratic form, 
with $c_{i,j,\tau,\tau'}=c_{j,i,\tau',\tau}$ for all $1\leq i,j\leq n$ and $\tau,\tau'\in \Gal(K/\QQ)$. We can write each coefficient $c_{i,j,\tau,\tau'}\in K$ with respect to the basis $\{\ome_1,\ldots, \ome_d\}$ as 
$
c_{i,j,\tau,\tau'}=\sum_{k=1}^d c_{i,j,\tau,\tau'}^{(k)}\ome_k.
$
We proceed  to compute the descended system explicitly, by writing 
$
X_i=\sum_{k=1}^d U_{k,i}\ome_k,
$
for $1\leq i\leq n$.
Then 
$$
F(X_1,\dots, X_n)= \sum_{1\leq i,j\leq n} \sum_{\tau,\tau'\in \Gal(K/\QQ)} \sum_{1\leq l,m,k\leq d} 
c_{i,j,\tau,\tau'}^{(k)} \ome_k U_{l,i}\ome_l^\tau U_{m,j}\ome_m^{\tau'}.
$$
Let $\{\rho_1,\ldots, \rho_d\}$ be a dual basis of $\{\ome_1,\ldots, \ome_d\}$ with respect to the trace,
so that 
$(
\tr_{K/\QQ}(\rho_i\omega_j))_{i,j}$ is the identity matrix
and any $\alpha\in K$ can be written in the form
$\alpha=\sum_{p=1}^d \Tr_{K/\QQ}(\alpha \rho_p) \omega_p$.
Thus  $F(X_1,\ldots, X_n)$ is equal to
$$
 \sum_{p=1}^d \ome_p \sum_{1\leq i,j\leq n}\sum_{1\leq l, m\leq d} U_{l,i}U_{m,j} \Tr_{K/\QQ}\left(\rho_p \sum_{1\leq k\leq d}\sum_{\tau,\tau'\in \Gal(K/\QQ)} c_{i,j,\tau,\tau'}^{(k)} \ome_k \ome_l^\tau\ome_m^{\tau'}\right)
$$
and we arrive at our descended system (\ref{eq:F-to-Q}),
with 
$$
Q_p(\underline\U)= \sum_{1\leq i,j\leq n}\sum_{1\leq l, m\leq d} \beta_{p,l,i,m,j}U_{l,i}U_{m,j} ,
$$
for  rational coefficients
\begin{align*}
\beta_{p,l,i,m,j}
&= \Tr_{K/\QQ}\left(\rho_p \sum_{1\leq k\leq d}\sum_{\tau,\tau'\in \Gal(K/\QQ)} c_{i,j,\tau,\tau'}^{(k)} 
\ome_k \ome_l^\tau\ome_m^{\tau'}\right)\\
&=
\sum_{1\leq k\leq d}\sum_{\tau,\tau'\in \Gal(K/\QQ)} c_{i,j,\tau,\tau'}^{(k)} \Tr_{K/\QQ}(\rho_p\ome_k\ome_l^\tau\ome_m^{\tau'}).
\end{align*}
By construction, the coefficients $\beta_{p,l,i,m,j}$  satisfy  $\beta_{p,l,i,m,j}=\beta_{p,m,j,l,i}$,
 for all $1\leq p,l,m\leq d$ and $1\leq i,j\leq n$. Moreover, they depend linearly on the given set of coefficients $(c_{i,j,\tau,\tau'}^{(k)})$. Now the space of all tuples $(c_{i,j,\tau,\tau'}^{(k)})$ of rational numbers satisfying the symmetry relation $c_{i,j,\tau,\tau'}^{(k)}=c_{j,i,\tau',\tau}^{(k)}$ can be parametrised by $\Q^{\frac{1}{2}dn(dn+1)d}$. Similarly, the space of all symmetric rational tuples $(\beta_{p,l,i,m,j})$ is naturally parametrised by $\Q^{\frac{1}{2}dn(dn+1)d}$. 
We  define the map
 $$
 \Phi:\Q^{\frac{1}{2}dn(dn+1)d}\rightarrow\Q^{\frac{1}{2}dn(dn+1)d}, \quad 
(c_{i,j,\tau,\tau'}^{(k)}) \mapsto (\beta_{p,l,i,m,j}).
$$
We claim that this  map is an injective linear map. This implies that there is a bijection between generalised quadratic forms in $n$ variables and systems of $d$ rational quadratic forms in $nd$ variables, as claimed in Theorem \ref{t:systems}.

To check the claim we 
assume that $\beta_{p,l,i,m,j}=0$ for all $1\leq p,l,m\leq d$ and $1\leq i,j\leq n$. By the non-degeneracy of the trace as a bilinear form, we deduce that
$$\sum_{\tau,\tau'\in \Gal(K/\QQ)} c_{i,j,\tau,\tau'} \ome_l^\tau\ome_m^{\tau'}=0, \quad 1\leq i,j\leq n,\ 1\leq l,m\leq d.$$
Note that the matrix $(\ome_l^\tau)_{\substack{1\leq l\leq d\\ \tau\in \Gal(K/\QQ)}}$ is of maximal rank, and hence we obtain
$$ \sum_{\tau'\in \Gal(K/\QQ)} c_{i,j,\tau,\tau'}\ome_m^{\tau'}= 0,\quad 1\leq i,j\leq n,\ \tau \in \Gal(K/\QQ),\ 1\leq m\leq d.$$
Applying the same argument again, we finally obtain
$$c_{i,j,\tau,\tau'}=0,\quad 1\leq i,j\leq n,\ \tau,\tau'\in \Gal(K/\QQ),$$
and hence $c_{i,j,\tau,\tau'}^{(k)}=0$ for all $1\leq k\leq d$.

\section{Recap from algebraic number theory } \label{s:ant}

In this section we collect together some of the facts about algebraic number fields that are important in our work. As usual $K/\QQ$ is a totally real Galois extension of degree $d$.  We shall henceforth write $\fo=\fo_K$ for its ring of integers.   In \S \ref{s:ideals} and 
\S \ref{s:prim} we recall some facts about ideals and discuss 
the construction of primitive characters modulo ideals, respectively. 
The need to deal with  generalised quadratic forms 
naturally leads to two  basic  objects that can be associated to a given integral ideal $\fb$ in $K$,
both of which depend on the particular generalised quadratic form we are working with and will be introduced in  \S \ref{s:dual}. 

\subsection{Properties of ideals}\label{s:ideals}

For any fractional 
ideal $\mathfrak{a}$ in $K$ one defines the dual ideal
$$
\hat \fa = \{\alpha\in K: \mbox{$\tr_{K/\QQ}(\alpha x)\in  \ZZ$ for all $x\in \mathfrak{a}$}\}.
$$
In particular 
$\hat \fa= \fa^{-1}\fd^{-1}$,  where 
$
\fd = 
\{\alpha\in K: \alpha\hat \fo\subseteq \fo\}
$ 
denotes the different ideal of $K$ and is itself an integral ideal. One notes that $\hat \fo=\fd^{-1}$. Furthermore,
we have $\hat \fa\subseteq \hat \fb$ if and only if $\fb\subseteq \fa$.
An additional integral ideal featuring in our work is the denominator ideal
$$
\fa_\gamma=\{ \alpha\in \fo: \alpha \gamma\in \fo\},
$$
associated to   any $\gamma\in K$.
Recall that    $\n \fa=|\fo/\fa|$ is the   ideal norm of any integral ideal
$\fa$.   
One important property of the ideal norm is that 
$\n\fa^{\tau}=\n \fa$ for any $\tau\in \Gal(K/\QQ).$ (This follows from the isomorphism $\fo/\fa\to\fo/\fa^\tau$ given by $\alpha\mapsto \alpha^\tau$.)  Furthermore, we have $\n\fa\in \fa$ for any integral ideal $\fa$.

We will write $(\fa,\fb)=\fa+\fb$ for the greatest common divisor of two integral ideals $\fa,\fb\subset \fo$.
When these ideals are coprime, meaning  that $\fa+\fb=\fo$,  we shall adopt the abuse of notation $(\fa,\fb)=1$. 
We close this section by recording the following basic  result.

\begin{lemma}\label{lem:alg1}
Let $\ve>0$ and let $\fb,\fc$ be integral ideals.
Then 
\begin{enumerate}
\item[(i)]
there exists $\alpha\in \fb$ such that $\ord_{\fp}(\alpha)=\ord_\fp(\fb)$ for every prime ideal
$\fp\mid \fc$;
\item[(ii)]
there exists $\alpha\in \fb$ and an unramified prime ideal $\fp$ coprime to
$\fb^\tau$ for all $\tau\in \Gal(K/\QQ)$, 
with $\n\fp \ll (\n \fb)^{\ve}$, 
 such that $(\alpha)=\fb\fp$.
\end{enumerate}
\end{lemma}

\begin{proof}
Part (i) is \cite[Lemma~2.2(i)]{BV} and part (ii) follows from an obvious modification to the proof of 
\cite[Lemma~2.2(ii)]{BV}.
\end{proof}

We shall also require a version of the Chinese remainder theorem,  as in \cite[Lemma~3]{skinner}.

\begin{lemma}\label{lem:skinner-3}
Suppose that $\fa,\fa_1,\fa_2$ are integral ideals such that $\fa=\fa_1\fa_2$,
with $\fa_1$ and $\fa_2$ coprime.
Let  $\alpha_1, \alpha_2\in \fo$ satisfy $\ord_\fp(\alpha_1)=\ord_\fp(\fa_1)$ 
and 
$\ord_\fp(\alpha_2)=\ord_\fp(\fa_2)$,
for all $\fp\mid \fa$. Then 
$$
\fo/\fa=\{\alpha_1 \mu +\alpha_2 \beta:  \beta\in \fo/\fa_1, ~\mu\in \fo/\fa_2\}.
$$
\end{lemma}

\subsection{Construction of primitive characters} \label{s:prim}

Let  $\psi(\cdot)=\exp(2\pi i \tr_{K/\QQ}(\cdot))$ 
be a character
on $K$.
The  following result gives a way to construct primitive characters 
$\fo/\fb\to \CC$.

\begin{lemma}\label{lem:denominator}
Let 
$\sigma_0(\cdot)=\psi(\gamma \cdot) :K\to \CC$,  for any 
$\gamma\in K$, 
and let $\fb\subsetneq \fo$ be an integral ideal. Then 
$\sigma_0$ is a non-trivial primitive   additive character modulo $\fb$ if and only if $\fa_\gamma=\fb\fe$ for some $\fe\mid \fd$ such that $(\fd/\fe,\fb)=1$.
\end{lemma}

\begin{proof}
We begin by showing that  
$\sigma_0$ is an additive character modulo $\fb$ if and only if $\fb\fd\subset \fa_\gamma$. 
But $\sigma_0$ is an additive character modulo $\fb$ if and only if 
$\sigma_0(x+z)=\sigma_0(x)$  for all $x\in \fo$ and $z\in \fb$.
But this happens if and only if $\gamma z \in \hat \fo$ for all $z\in \fb$, which is if and only if $\fb\fd\subset \fa_\gamma$. This establishes the claim. 

Now suppose that $\fb\fd\subset \fa_\gamma$, which means that $\fa_\gamma\mid \fb\fd$. Thus there is an integral ideal $\fh$ such that $\fb\fd=\fa_\gamma \fh	$. We wish to show that $\sigma_0$ is primitive if and only if $\fh\mid \fd$ with $(\fh,\fb)=1$. To do so we note that $\sigma_0$ is primitive if and only if $\fa_\gamma \nmid \fb_1 \fd$ for all $\fb_1\mid \fb$ with $\fb_1\neq \fb$. Indeed,  if  $\fa_\gamma \mid \fb_1 \fd$ for some proper divisor $\fb_1\mid \fb$, then   
$\gamma z \in \hat \fo$ for every $z\in \fb_1$, which would mean that $\sigma_0$ is a character modulo $\fb_1$.
Suppose that $\sigma_0$ is primitive and suppose that there is a prime ideal $\fp\mid \fh$ such that $\fp\mid \fb$. Writing 
$\fh'=\fh\fp^{-1}$ and $\fb'=\fb\fp^{-1}$ it follows that 
$\fb'\fd=\fa_\gamma \fh'	$, whence $\fa_\gamma\mid \fb'\fd$, which is a contradiction. 
Thus $\fh$ is coprime to $\fb$ and we must have $\fh\mid \fd$.
Suppose now that   $\fb\fd=\fa_\gamma \fh$ for some  $\fh\mid \fd$ such that  $(\fh,\fb)=1$. 
We wish to deduce that $\sigma_0$ is primitive, for which we 
 suppose for a contradiction that there exists a proper divisor $\fb_1\mid \fb$ such that 
 $\fa_\gamma \mid \fb_1 \fd$. Writing $\fb=\fb_1\fb_2$ and recalling that $\fb\fd=\fa_\gamma \fh$, we deduce that $\fh = (\fa_\gamma^{-1}\fb_1 \fd)\fb_2$, whence $\fb_2\mid \fh$, which is impossible since $(\fh,\fb)=1$.

Finally we  note that 
$\sigma_0$ is a trivial character if and only if $\gamma\in \hat \fo$, which is equivalent to 
 $\fa_\gamma\supseteq  \fd$. This is clearly impossible for any primitive character $\sigma_0$ modulo a proper ideal
 $\fb\subsetneq \fo$, since $(\fd/\fe,\fb)=1$  if $\fa_\gamma=\fb\fe$.
 \end{proof}

We proceed to define a  particularly convenient   additive character modulo $\fb$. 
Associated to any non-zero integral ideal  $\fb$ is the subset $\mathfrak{F}(\fb)\subset K$ given by 
$$
\mathfrak{F}(\fb)=
\hspace{-0.05cm}
\left\{\frac{g}{\alpha}\in K: 
\begin{array}{l}
\text{
$\exists$ ~prime ideal $\fp_1$  with 
$\n \fp_1 \ll \n \fb$ s.t.}\\
\qquad\text{(i)}. \quad 
(\alpha)=\fb\fd\fp_1\\
\qquad\text{(ii)}. \quad 
\text{$g\in \fp_1\cap\ZZ$  with $((g),\fb^\tau\fd)=1$ $\forall~ \tau \in \Gal(K/\QQ)$}\\
\qquad\text{(iii)}. \quad 
\text{$\exists$  $\fe\mid \fd$ s.t. $(\fd/\fe,\fb)=1$ and 
$\fa_{g/\alpha}=\fb\fe$}
\end{array}
\hspace{-0.2cm}
\right\}.
$$
Note that  condition (i) implies that $\alpha\in \fb\fd$ and condition 
(ii) implies that $\fp_1\nmid \fb^\tau\fd$ for any $\tau\in \Gal(K/\QQ)$.
We may now record a variant of \cite[Lemma~2.3]{BV}, which shows that 
$\mathfrak{F}(\fb)\neq \emptyset$ for any choice of $\fb$.

\begin{lemma}\label{lem:orthogonal}
Let $\fb\subsetneq \fo$ be a non-zero ideal. Then there exists $\gamma\in \mathfrak{F}(\fb)$
such that $\e(\gamma\cdot)$ defines a non-trivial primitive additive character modulo $\fb$. 
\end{lemma}

\begin{proof}
We consider the integral ideal
$\fc=\fb\fd$. Observe that $\fd^\tau=\fd$ for all $\tau\in \Gal(K/\QQ)$,
since $\fd=\hat\fo^{-1}$ and the trace is invariant under the action of the Galois group.
Taking $\ve=1$ in 
Lemma \ref{lem:alg1}(ii), we
can   find  $\alpha\in \fc$ and a prime
ideal $\fp_1$ coprime to $\fc^\tau=\fb^\tau\fd$,  
for every $\tau\in \Gal(K/\QQ)$, 
with $\n\fp_1 \ll \n \fb$, and such that
$(\alpha)=\fc\fp_1$.
It follows from Lemma~\ref{lem:alg1}(i)  that  there exists $\nu\in \fp_1$ such that  $((\nu), \fc^\tau)=1$
for any $\tau\in \Gal(K/\QQ)$.
But this implies that 
$g=N_{K/\QQ}(\nu)$ is coprime to $\fc^\tau$, 
for any $\tau\in \Gal(K/\QQ)$, 
with $g\in  \fp_1$.
We will show that   $\fc=\fa_\gamma$ with $\gamma=g/\alpha$, after which an application of 
Lemma \ref{lem:denominator} with $\fe=\fd$ will complete the proof.
To check the claim we note that 
$$
\beta\in \fa_\gamma \Leftrightarrow \gamma\beta\in \fo \Leftrightarrow (g\beta)\subset (\alpha)\Leftrightarrow (\alpha)=\fb\fd\fp_1\mid (g\beta)
\Leftrightarrow \beta\in \fb\fd,
$$
since $\fp_1\mid (g)$ and 
$\fb\fd$ is coprime with $(g)$.
\end{proof}

\subsection{The $G$-invariant ideal and an important $\ZZ$-module}\label{s:dual}

Let $F$ be a  generalised quadratic form, as in Definition \ref{def:gq}.
Let $G=G_F\subset \Gal(K/\QQ)$ be the subset of automorphisms $\tau\in \Gal(K/\QQ)$ that actually appear in $F$.  Note that  $G=\{\mathrm{id}\}$ if and only if $F$ is a standard quadratic form. 
For any integral ideal $\fb$ we define the {\em 
$G$-invariant ideal} to be
\begin{equation}\label{eq:Gb}
 \gfb=\bigcap_{\tau\in G} \fb^{\tau^{-1}}.
\end{equation}
This is the least common multiple of the ideals
$\fb^{\tau^{-1}}$ for $\tau\in G$.

Next, associated to our generalised quadratic form $F$ is a  {\em generalised bilinear form}
\begin{equation}\label{eq:bilinear}
B(X_1,\dots,X_n;Y_1,\dots,Y_n)=
\sum_{1\leq i,j\leq n} \sum_{\tau,\tau'\in \Gal(K/\QQ)}c_{i,j,\tau,\tau'} X_i^{\tau} Y_j^{\tau'}.
\end{equation}
This defines a map $K^n\times K^n\to K$, with 
$$
B(\x;\u+\v)=B(\x;\u)+B(\x;\v) \quad \text{and}\quad 
B(\u+\v;\y)=B(\u;\y)+B(\v;\y),
$$
for any vectors $\x,\y,\u,\v\in K^n$. (However,  this fails to be a bilinear form on $K^n$ since 
$B(\lambda \x;\y)$, $B(\x;\lambda\y)$ and $\lambda B( \x;\y)$ needn't be equal for  
 $\lambda\in K$.)

For any  ideal $\fb\subset \fo$, let 
$$
\calH_\calb=\left\{
\bfh \in \calO^n: 
\text{$F(\bfa + \bfh) \equiv F(\bfa) \modu \calb$ for all  $\bfa \in \calO^n$}\right\}.
$$
This is an  additive group and it is clear that  $\gfb^n\subset \cH_\fb \subset \calO^n$,  where $\gfb$ is the $G$-invariant ideal defined in \eqref{eq:Gb}.
By testing the hypothesis with $\bfa \equiv \0 \modu \gfb$, 
we have $F(\bfh) \equiv 0 \modu \calb$ for any 
 $\bfh\in \cH_\fb$. Hence
\begin{equation}\label{eq:define-H'}
\calH_\calb=\left\{
\bfh \in \calO^n: 
\text{$2B(\bfa;\h)\in \fb$ for all  $\bfa \in \calO^n$}\right\}.
\end{equation}
We claim that $\cH_\fb$ has the structure of a 
finitely generated $\Z$-module. To see this, let  $\bfe_i$ be the $i$th unit vector,
for $1\leq i\leq n$. Observe that $(\n \grb) \ome_j\bfe_i\in \calH_\grb$ for all $1\leq i\leq n$ and $1\leq j\leq d$. Hence the image of $\calH_\grb$ under the isomorphism $\calO^n \cong \Z^{nd}$ is a lattice of full rank and, thus, finitely generated as a $\Z$-module.

The set $\calH_\calb$ will emerge naturally in our analysis of certain key exponential sums and it will be important to have an estimate for its index in $\fo^n$. In the special case \eqref{eq:special-0} it will be easier to calculate 
 $\calH_\grb$  directly, but for now we content ourselves with proving a general bound. 
In the following lemma
we consider the coefficient matrix $(c_{i,j,\tau,\tau'})_{(i,\tau)\times (j,\tau')}$ 
of a generalised quadratic form 
as a  $nd\times nd$ matrix.

\begin{lemma}
There is a constant $C_1>0$,  depending only on $F$, such that for all $\fb$ we have
$$|\calO^n/\calH_\grb|\leq C_1 (\n \grb)^{\rank(c_{i,j,\tau,\tau'})}.$$
Moreover, there is an integral ideal $\fd_1$, such that one can take $C_1=1$ for all ideals $\fb$ with 
$(\grb,\fd_1)=1$.
\end{lemma}

\begin{proof}
Let $\Del= \rank (c_{i,j,\tau,\tau'})_{(i,\tau)\times (j,\tau')}$. Let $\calS\subset \{1,\ldots, n\}\times \Gal(K/\QQ)$ be a subset of indices such that the vectors $(c_{i,j,\tau,\tau'})_{(j,\tau')}$, $(i,\tau)\in \calS$, are linearly independent and $|\calS|$ is maximal. Then, for any $(k,\sigma)\in \{1,\ldots, n\}\times \Gal(K/\QQ)$ there are numbers $a_{i,\tau}^{(k,\sigma)}$ (for $(i,\tau)\in \calS$), such that
$$ c_{k,j,\sigma,\tau'}= \sum_{(i,\tau)\in \calS} a_{i,\tau}^{(k,\sigma)}c_{i,j,\tau,\tau'},$$
for all $1\leq j\leq n$ and $\tau'\in \Gal(K/\QQ)$. Let $\alp\in \calO$ such that $\alp a_{i,\tau}^{(k,\sigma)}\in \calO$ for all $(i,\tau)\in \calS$ and $(k,\sigma)\in \{1,\ldots, n\}\times \Gal(K/\QQ)$. Now, set
$$\calH_\grb'
= \left\{\bfh\in \calO^n: \sum_{1\leq j\leq n} \sum_{\tau'\in \Gal(K/\QQ)} c_{i,j,\tau,\tau'}h_j^{\tau'}\in (\alp)\grb,\, \forall (i,\tau)\in \calS\right\}.$$
Observe that
$\calH_\grb'\subset \calH_\grb$.
Moreover, if $\grb$ and $(\alp)$ are coprime, then the ideal $(\alp)$ may be omitted in the definition of $\calH_\grb'$.\par
Finally, we observe that there is an injection
\begin{equation*}
\begin{split}
\psi: \calO^n/\calH_\grb' &\rightarrow (\calO/\alp\grb)^{\Del},\quad
[\bfh]\mapsto \left( \sum_{1\leq j\leq n}\sum_{\tau'\in \Gal(K/\QQ)}c_{i,j,\tau,\tau'}h_j^{\tau'}\right)_{(i,\tau)\in \calS}.
\end{split}
\end{equation*}
Hence 
$|\calO^n/\calH_\grb'|\leq (\n (\alp\grb))^{\Del},$
which suffices, since   $\calH_\grb'\subset \calH_\grb$. 
\end{proof}

We now wish to provide  an alternative upper bound involving $\calH_\grb$, under a suitable assumption on the generalised quadratic form.

\begin{definition}\label{cond2}
We say that $F(X_1,\dots,X_n)$ is {\em admissible} if there exist vectors 
$$
\bfv_1,\ldots, \bfv_n\in K^n
$$ 
such that 
$B(\bfv_i;\bfh)=0$ for all $1\leq i\leq n$ if and only if  $\bfh=\0$.
\end{definition}

In this language  a standard  quadratic form is admissible if and only if it is non-singular. 
We may now prove the following result.

\begin{lemma}\label{lem:discriminant}
Assume that $F$ is admissible. Then there exists a constant $C_2>0$,  depending only on $F$, such that 
$$|\calH_\grb/\gfb^n|\leq C_2\frac{(\n\gfb)^n}{(\n\grb)^n} .$$
Moreover, there exists an integral ideal $\fd_2$ such that one can take $C_2=1$ for all ideals $\grb$ with $(\grb,\fd_2)=1$.
\end{lemma}

We can use this result to get information about the index of $\calH_\grb$ in $\fo^n$ via the identity
\begin{equation}\label{eq:dinky}
|\fo^n/\calH_\grb||\calH_\grb/\gfb^n| =|\fo^n/\gfb^n|=(\n \gfb)^n.
\end{equation}
Lemma \ref{lem:discriminant} is of the expected  magnitude, which we can see by considering the case of the standard diagonal quadratic form  $F(\XX)=\sum_{i=1}^n c_iX_i^2$, for example, 
with non-zero $c_1,\dots,c_n\in \fo$. In this case $G=\{\mathrm{id}\}$ and $\gfb=\fb$. 
It therefore follows that 
$|\cH_\fb/\gfb^n|\ll 1$, since 
$\cH_\fb=(2c_1)^{-1}\fb\times\cdots \times (2c_n)^{-1}\fb$.

\begin{proof}[Proof of Lemma \ref{lem:discriminant}]
Let $\bfv_1,\dots, \bfv_n$ be a set of vectors as in Definition  \ref{cond2}. By scaling these vectors with a rational  integer, we may assume that $\bfv_i\in \calO^n$ for all $1\leq i\leq n$. We define the auxiliary set
$$ \widetilde{\calH}_\grb=\{\bfh\in \calO^n: 2B(\bfv_i;\bfh)\in \grb,\ \forall 1\leq i\leq n\},$$
and observe that $\calH_\grb\subset \widetilde{\calH}_\grb$. Next, consider the map
\begin{equation*}
\begin{split}
\phi: \calO^n &\rightarrow \calO^n , \quad 
\bfh \mapsto (2B(\bfv_i;\bfh))_{1\leq i\leq n},
\end{split}
\end{equation*}
which is injective by the definition of admissibility in Definition  \ref{cond2}. Let $\Gam$ be the image of $\calO^n$ under the map $\phi$. Then $\phi$ induces an isomorphism
$$ \calO^n/\widetilde{\calH}_\grb \cong \Gam/(\grb^n\cap \Gam).$$
Note that $\Gam$ only depends on $B(\XX;\YY)$ and vectors $\bfv_1,\dots,\bfv_n$, and 
hence can be taken to be independent of the ideal $\grb$. As in \eqref{eq:dinky}, we therefore obtain
$$|\calH_\grb/\gfb^n|\leq |\widetilde{\calH_\grb}/\gfb^n| = \frac{|\calO^n/\gfb^n|}{|\calO^n/\widetilde{\calH}_\grb|} = \frac{(\n\gfb)^n}{|\Gam /(\grb^n\cap \Gam)|} \leq C_\Gam \frac{(\n\gfb)^n}{(\n\grb)^n},$$
where $C_\Gam$ is a constant only depending on $\Gam$. Moreover, there is an ideal $\mathfrak{d}_2$ such that $|\Gam/(\grb^n\cap \Gam)|=(\n\grb)^n$ whenever $(\fd_2,\grb)=1$.
This completes the proof of the lemma.
\end{proof}

\section{Enter the circle method}\label{s:circle}

Our primary tool in this paper is a number field version of the Hardy--Littlewood circle method
to interpret the function $\delta_K$ in \eqref{eq:deltaK}.
Let $K$ be a totally real Galois extension of $\QQ$  of degree $d$.
Let $Q\geq 1$ and let $\alpha\in \fo$.
Then we shall use the version worked out by Browning and Vishe  \cite[Thm.~1.2]{BV}. This states  that there exists a positive constant $c_Q=1+O_A(Q^{-A})$, for any $A>0$,  and an infinitely differentiable function $h: 
(0,\infty)\times \mathbb R\rightarrow \RR$ such that
\begin{equation}\label{eq:BV}
 \delta_K(\alpha)=\frac{c_Q}{Q^{2d}}
 \sum_{(0)\neq \mathfrak{b}\subseteq \fo  }
~\sigstar \sigma(\alpha)h\left(\frac{\n \fb}{Q^{d}}  , \frac{|\n_{K/\QQ} (\alpha)|}{Q^{2d}}\right),
\end{equation}
where $\n\fb=|\fo/\fb|$ denotes the norm of the ideal $\fb$ and 
the notation 
$\sum^{*}_{\sigma \bmod{\fb}}$ means that the sum is taken over
primitive additive characters modulo $\fb$.
Furthermore,  we have $h(x,y)\ll x^{-1}$ 
and $h(x,y)\neq 0$ only if $x\leq \max\{1,2|y|\} $.

We fix some notation, before proceeding further.
Let 
$D_K$ be  
the discriminant of $K$ and  note that $D_K>0$, since $K$ is totally real.
Let $\rho_{1},\dots,\rho_{d}:K \hookrightarrow \RR$ be the distinct real embeddings of $K$,  and let 
 $V=
	K\otimes_\QQ \RR 
\cong \RR^d
$.
There is a canonical embedding $K\hookrightarrow V$ 
given by $\alpha
\mapsto  (\rho_{1}(\alpha), \dots ,  
\rho_{d}(\alpha))$.
We  identify $K$ with its image in $V$. 
If  $v=(v_1,\dots,v_d)\in V$ then 
we  extend the norm and trace on $K$ to get functions 
$\nm(v):V\to\RR$ and $\tr(v):V\to \RR$, with
$$ \nm(v)=\prod_{l=1}^{d}v_l,\quad \tr(v)= \sum_{l=1}^{d}v_l
.$$
We extend the absolute value on $\RR$ to give a norm on $V$ via $| v | = \max_{1\leq l\leq d}|v_l|$, which we extend to $V^n$ in the obvious way.

Let $N\in \fo$ and let $F(X_1,\dots,X_n)$ be a generalised quadratic form defined over $\fo$.
Our central concern is with the 
asymptotic behaviour of the sum
$$
 N_{W}(F,N;P) = 
 \sum_{\substack{\x\in \fo^n\\ F(\x)=N}} W(\x/P) ,
$$
as $P\rightarrow \infty$, for   $W\in \cW_n^+(V)$, where $\cW_n^+(V)$ is the class of smooth weight functions described in \cite[\S 2.2]{BV}.
Our goal in this section is to lay some  groundwork that will be useful 
for Theorems \ref{t:non-diag-1}--\ref{t:diag}, but which applies to arbitrary generalised quadratic forms.

First, in \S \ref{s:link} we shall discuss the link between the descended system associated to $F$ and the ``embedded system'' that arises from looking at all of the different real embeddings of $F$. In \S \ref{s:W} we shall construct the weight function $W$ that features in our counting function $ N_{W}(F,N;P)$. In \S \ref{s:poisson} we shall combine \eqref{eq:BV} with Poisson summation, in order to arrive at a preliminary expression for $ N_{W}(F,N;P) $ in Lemma \ref{lem:poisson}. In \S \ref{s:sum} we  make some preliminary investigations into exponential sums, and similarly for exponential  integrals
in \S \ref{s:integrals}.  In \S \ref{s:trivial} we shall discuss the main term that comes from the trivial character after Poisson summation is applied. 
Finally, in \S \ref{s:non-trivial} we shall make some initial observations concerning the contribution from the non-trivial characters.

\subsection{The  embedded system} \label{s:link}
Let $F(X_1,\dots,X_n)$ be a generalised quadratic form and let $\{\omega_1,\dots,\omega_d\}$ be a $\ZZ$-basis for $\fo$. We have seen in 
\eqref{eq:F-to-Q} how there is a descended system $\{Q_1,\dots,Q_d\}$ 
of quadratic forms, that is 
associated to $F$ via
$$
F(X_1,\dots,X_n)=\sum_{1\leq i\leq d}{\omega_i}Q_i(\U_1,\dots,\U_d),
$$
with variables 
$\U_l=(U_{l1},\dots,U_{ln})$ for $1\leq l\leq d$.

We will need to be able to relate the descended system to the {\em embedded system}, which amounts to how $F(\x)$ embeds in $V$ for given $\x\in K^n$.  
We extend  $F:K^n\to K$ to get a  map $V^n\to V$, through the identification of $K$ with $V$. 
Associated to $\x$ is the vector 
 $
(\x^{(1)}, \dots, \x^{(d)}),
$
with $\x^{(l)}\in \RR^n$ for $1\leq l\leq d$. 
Let $l\in \{1,\dots,d\}$. To any  $\tau\in \Gal(K/\QQ)$ may be associated a unique 
integer $l_\tau\in \{1,\dots,d\}$ such that 
\eqref{eq:l_tau} holds.
Then we define
\begin{equation}\label{eq:Ql}
F^{(l)}(\x^{(1)},\dots,\x^{(d)})=
\sum_{1\leq i,j\leq n} \sum_{\tau,\tau'\in \Gal(K/\QQ)}c_{i,j,\tau,\tau'}^{(l)}x_i^{(l_{\tau^{-1}})} x_j^{(l_{\tau'^{-1}})},
\end{equation}
where $c_{i,j,\tau,\tau'}^{(l)}=\rho_l(c_{i,j,\tau,\tau'})\in \RR$. 
With this notation we have 
$$
\rho_l(F(\x))=F^{(l)}(\x^{(1)},\dots,\x^{(d)}).
$$
Thus 
$\rho_l(F(\x))$ is a real quadratic form  in the $dn$ variables 
$\x^{(1)},\dots,\x^{(d)}$. We call $\{F^{(1)},\dots,F^{(d)}\}$ the {\em embedded system}.
In particular, it is clear that $N_{K/\QQ}(F(\x))=\nm(F(\x))$ and 
\begin{equation}\label{eq:sol}
\tr(vF(\x))=
\sum_{1\leq l \leq d} v_l \rho_l\left(F(\x)\right),
\end{equation}
for any $v=(v_1,\dots,v_d)\in V$ and $\x\in V^n$,
identities that we shall often make use of
in our analysis of the exponential integrals in \S \ref{s:integrals}.

Note that 
if  $F$ is a standard quadratic form, then $\rho_l(F(\x))=F^{(l)}(\x^{(l)})$, for $1\leq l\leq d$. 
One positive effect of this is that the relevant  oscillatory  integrals 
factorise into a product of $d$ integrals, one for each embedding. 
The situation is much  more complicated for  generalised quadratic forms since
there is usually no such factorisation.

Let $\mathbf{A}=(\omega_j^{(i)})_{1\leq i,j\leq d}$, where $\omega_j^{(i)}=\rho_i(\omega_j)$. 
Then $(\det \mathbf{A})^2=D_K$. 
Moreover, on recalling
that $\x=\omega_1\u_1+\cdots+\omega_d\u_d$, 
 we have 
\begin{equation}\label{eq:W-transform}
\begin{pmatrix}
\x^{(1)}\\
 \vdots \\
 \x^{(d)}
 \end{pmatrix}
 =
\mathbf{W} 
\begin{pmatrix}
\u_1\\
 \vdots \\
 \u_d
 \end{pmatrix},
\end{equation}
where $\mathbf{W}$ is the $dn\times dn$ block matrix
\begin{equation}\label{eq:W-matrix}
\mathbf{W}=\begin{pmatrix}
    \omega_1^{(1)}\mathbf{I}_n      & \omega_2^{(1)}\mathbf{I}_n &  \dots & 
    \omega_d^{(1)} \mathbf{I}_n\\
    \omega_1^{(2)} \mathbf{I}_n     & \omega_2^{(2)} \mathbf{I}_n&  \dots & 
    \omega_d^{(2)}\mathbf{I}_n\\
    \hdotsfor{4} \\
    \omega_1^{(d)} \mathbf{I}_n      & \omega_2^{(d)} \mathbf{I}_n &  \dots & 
    \omega_d^{(d)}\mathbf{I}_n
\end{pmatrix}.
\end{equation}
Switching appropriate   rows and columns  takes $\mathbf{W}$ to 
$\mathrm{Diag}(\mathbf{A},\dots,\mathbf{A})$, whence 
$\det \mathbf{W}=(\det \mathbf{A})^{n}=D_K^{n/2}$.
In particular, it follows that 
$$
F^{(l)}(\x^{(1)},\dots,\x^{(d)})
=\sum_{1\leq i\leq d} \omega_i^{(l)} Q_i(\u_1\dots,\u_d),
$$
for any $1\leq l\leq d$, 
under the transformation \eqref{eq:W-transform}.

\subsection{Construction of the weight $W$}\label{s:W}

We  assume that the descended system is of codimension $d$ and has a non-singular real point. This means that there exists $\underline\bxi=(\bxi_1,\dots,\bxi_d)\in \RR^{dn}$ such that  $J_{Q_1,\dots, Q_d}(\underline\bxi)$ has rank $d$, where  
  $$
J_{Q_1,\dots,Q_d}=
\left(\frac{\partial }{\partial X_j^{(k)}}Q_l\right)_{\substack{1\leq l\leq d\\ 
1\leq k\leq d, 1\leq j\leq n}
}
$$
is the associated
 $d\times dn$
Jacobian 
 matrix. 
 Define the smooth weight function
$$
w(x)=\begin{cases}
e^{-1/(1-x^2)} &\text{ if $|x|<1$,}\\
0&\text{ if $|x|\geq 1$,}
\end{cases}
$$
and let $\delta>0$ be a  small parameter.
 In this paper we shall work with the weight function
$W:V^n\to \RR_{\geq 0}$, which is given by
$$
W(\x)=w(\delta^{-1}|\mathbf{W}^{-1}\x-\underline\bxi|),
$$
where $\x$ is identified with  $(\x^{(1)},\dots,\x^{(d)})$, and where
$\mathbf{W}$ is the matrix in \eqref{eq:W-matrix}.
It is clear that $W$ is infinitely
differentiable and that it is supported on the region
$|\mathbf{W}^{-1}\x-\underline\bxi| \leq \delta$. 
Ultimately we will want to work with a value of $\delta$ that is
sufficiently small, but which  still satisfies $1\ll \delta \leq 1$ for an absolute implied constant.

\subsection{Poisson summation}\label{s:poisson}

It follows from \eqref{eq:BV} that 
\begin{equation}\label{eq:No1}
\begin{split}
&\No \\
&\quad =\frac{c_{Q}}{Q^{2d}}\sum_\mathfrak{b} \sigstar
\sum_{\x\in \fo^{n}}
   \sigma(F(\x)-N)W(\x/P)h\left(\frac{\n\mathfrak{b}}{Q^d}  ,\frac{|N_{K/\QQ}(F(\x)-N)|}{Q^{2d}}\right),
\end{split}
\end{equation}
for any $Q\geq 1$.
Here 
the constant $c_Q$ satisfies 
$
c_Q=1+O_A(Q^{-A}),
$ 
for any  $A>0$.
 Furthermore,  we have $h(x,y)\ll x^{-1}$ for all $y$
and $h(x,y)\neq 0$ only if $x\leq \max\{1,2|y|\} $.

In our work we will take $Q=P$ and we  henceforth follow the convention that 
the implied constant in any estimate involving $W$  
is allowed to depend implicitly on the parameters  that enter into the  definition
of $\mathcal{W}_n(V)$
in \cite[\S 2.2]{BV}.
Likewise, the integer $N$ and the number field $K$ are considered fixed once and for all,  so 
that all implied constants are allowed to
depend implicitly on $N$ and $K$. 
In view of the  fact that $h(x,y)\neq 0$ only if $x\leq \max(1,2|y|) $, 
it is clear that the sum over $\mathfrak{b}  $ is restricted to $\n \fb\ll Q^{d}=P^d $.

If  $F$ were a standard quadratic form over $\fo$, we would 
 proceed by breaking the sum over $\x$
into residue classes modulo $\fb$, before executing an application of Poisson summation. 
This would ultimately lead to an expression of the form \cite[Thm.~5.1]{BV}.
For generalised quadratic forms $F$ this 
route is not directly accessible,  since for given $\a,\h\in \fo^n$ 
and any primitive character $\sigma$ modulo $\fb$, 
  one may have 
$\sigma(F(\a+\h))\neq \sigma(F(\a))$
even when  $\h\in \fb^n$. 
In this way, we see that a special role will be played by the set  
$\cH_\fb$, that was introduced in \S \ref{s:dual}.

\begin{lemma}
 \label{lem:poisson}
We have
$$
\No=
\frac{c_{P}P^{(n-2)d}}{D_{K}^{n/2}} 
\sum_{\n\mathfrak{b}\ll P^d} 
\sum_{\substack{\m\in \widehat{\gfb  }^n}}
(\n \gfb  )^{-n}
S_\mathfrak{b}  (N;\m)
I_\mathfrak{b}  (N/P^2;P\m),
$$
where the sum over $\fb$ is over non-zero integral ideals and
\begin{align*}
S_\mathfrak{b}  (N;\m)
&=\sigstar\sum_{\ma{a}\bmod{\gfb }}\sigma(F(\ma{a})-N)\e(\m. \ma{a}),\\
I_\mathfrak{b}  (t;\k)
&=
\int_{ V^{n}}W(\x)h\left(\frac{\n \mathfrak{b}}{P^d}  ,|\nm(F(\x)-t)|\right)\e\left(-\k.\x \right)\d\x.
\end{align*}
\end{lemma}

\begin{proof}
Our approach is based on breaking the $\x$-sum in \eqref{eq:No1} into residue classes modulo $\gfb$.
Since $Q=P$ and $\gfb^n\subset \cH_\fb$, it follows that this sum equals
$$
\sum_{\ma{a}\in (\fo/\gfb  )^n}\sigma(F(\ma{a})-N)
\sum_{\x\in {\gfb}^n}W\left((\x+\ma{a})/P\right)
h\left(\frac{\n \mathfrak{b}}{P^d}  ,\frac{|N_{K/\QQ}(F(\x+\ma{a})-N)|}{P^{2d}}\right),
$$
for any primitive character $\sigma$ modulo $\fb$. 
We apply the multi-dimensional Poisson summation formula (cf. \cite[\S 5]{BV}). Recalling that $K$ is totally real, we find that 
the inner $\x$-sum  is equal to
\begin{align*}
\frac{1}{D_K^{n/2}(\n \gfb  )^{n}} 
\sum_{\m\in \widehat{\gfb  }^n}
\e(\m.\a)
\int_{V^{n}} W(\x/P)
h\left(\frac{\n \mathfrak{b}}{P^d}, \frac{|\nm(F(\x)-N)|}{P^{2d}}\right)
\e(-\m.\x)\d \x,
\end{align*}
where we recall that 
$\widehat{\gfb  }=\gfb^{-1}\fd^{-1}$ is the dual of $\gfb$. 
Putting everything together in \eqref{eq:No1}, we have therefore established
that 
$$
\No=
\frac{c_{P}}{D_{K}^{n/2}P^{2d}}
\sum_{\n\mathfrak{b}\ll P^d} \sum_{\m\in \widehat{\gfb  }^n}(\n \gfb  )^{-n}S_\mathfrak{b}  (N;\m)
\tilde I_\mathfrak{b}  (\m),
$$
with $S_\mathfrak{b}  (N;\m)$ as in the statement of the lemma and 
$$
\tilde I_\mathfrak{b}  (\m)= 
\int_{ V^{n}}W(\x/P)h\left(\frac{\n \mathfrak{b}}{P^d}  ,\frac{|\nm(F(\x)-N)|}{P^{2d}}\right)\e\left(-\m.\x \right)\d\x.
$$
A simple change of variables yields 
$\tilde I_\mathfrak{b}  (\m)=P^{dn} I_\mathfrak{b}  (N/P^2;P\m)$, as required.
\end{proof}

\subsection{The exponential sum}
\label{s:sum}

We proceed by  discussing 
$S_\mathfrak{b}  (N;\m)$ 
in Lemma \ref{lem:poisson}, 
for $\m\in \widehat\gfb^n$. Let  $\gamma=g/\alpha\in \mathfrak{F}(\fb)$ be as in Lemma \ref{lem:orthogonal}.
Then we have 
$$
\starsum_{\sigma\bmod{\fb}
}\sigma(x) = 
\sum_{a\in (\fo/\fb)^*} \e( \gamma a x),
$$
for any $x\in \fo$.
It follows that 
\begin{equation}\label{eq:SUM}
S_\mathfrak{b}  (N;\m)
=\sum_{a\in (\fo/\fb)^*} \psi(-\gamma a N)
\sum_{\x\bmod{\gfb}}\psi\left( \gamma a F(\x)+   \m. \x\right).
\end{equation}
Our work hinges upon the following upper bound for this sum.

\begin{lemma}\label{expsum1}
We have 
$$
|S_\mathfrak{b}  (N;\m)|
\leq |(\calO/\calb)^*| |\calH_\calb/\Gb^n|^{1/2}|\calO/\Gb|^{n/2},
$$
where $\cH_\fb$ is given by \eqref{eq:define-H'}.
\end{lemma}
\begin{proof}
For fixed $a\in (\calO/\calb)^*$,  we have
\begin{equation*}
\begin{split}
\Bigg| \sum_{\x \bmod \Gb}\psi(\gam a&F(\x)+\bfm.\x)\Bigg|^2 \\
&= \sum_{\bfh \bmod \Gb}\sum_{\bfu \bmod \Gb} \psi (\gam a(F(\bfu+\bfh)-F(\bfu))+\bfm.\bfh)\\
&\leq \sum_{\bfh\bmod \Gb} \Bigg| \sum_{\bfu \bmod \Gb} \psi (2\gam a B(\bfu;\bfh))\Bigg|,
\end{split}
\end{equation*}
in the notation of \eqref{eq:bilinear}.
We observe that the function $\bfu \mapsto  \psi (2\gam a B(\bfu;\bfh))$
is a character modulo $\Gb^n$, and it is the trivial character precisely when
\[ 2\gam a B(\bfu;\bfh) \in \cald^{-1}, \qquad \forall \bfu \in (\fo/\Gb)^n.\]
We rewrite $\gam$ in the form $\gam=g/\alp$ with $(\alp)=\calb\cald\grp_1$ for some prime ideal $\grp_1$ and $g\in \grp_1\cap\Z$ with the property that $((g),\cald\Gb)=1$. 
 Thus the above condition is equivalent to the condition 
 \[ 2 g a B(\bfu;\bfh) \in (\alpha) \cald^{-1} = \calb \grp_1, \qquad \forall \bfu \in (\calO/\Gb)^n.\]
 Since $a \in ( \calO/\calb)^*$, $g\in \grp_1$ and $((g),\cald\Gb)=1$, this is equivalent to saying that 
  \[ 2 B(\bfu;\bfh) \in  \calb, \qquad \forall \bfu \in (\calO/ \Gb)^n.\]
Finally, since this condition on $\u$ is invariant modulo $\Gb^n$, 
this is equivalent to the condition
$2 B(\bfu;\bfh) \in  \calb$, for all  $\bfu \in \calO^n$, 
which is 
equivalent to specifying that $\bfh \in \calH_\calb$, by 
 \eqref{eq:define-H'}. The statement of the lemma  now follows.
\end{proof}

\begin{corollary}\label{cor-expsum1}
Assume that $F$ is admissible, in the sense of Definition  \ref{cond2}.
Let  $\calb$ be an integral ideal and let 
 $\bfm\in K^n$. Then 
$
S_\fb(N;\bfm)\ll  (\n\fb)^{1-n/2} (\n \gfb)^{n}.
$
\end{corollary}

\begin{proof}
This follows from combining  
Lemmas \ref{lem:discriminant} and~\ref{expsum1}.
\end{proof}

It  is straightforward to show that 
$S_\mathfrak{b}(N;\m)$ vanishes unless $\m$ satisfies additional constraints, as demonstrated in the following result. 

\begin{lemma}\label{lem:S-zero}
We have $S_\mathfrak{b}  (N;\m)=0$ unless 
$\m.\h\in \fd^{-1}$ for all $\h\in \mathcal{H}_\fb$. 
\end{lemma}

\begin{proof}
Returning to the definition of $S_\fb(N;\m)$ in 
 Lemma \ref{lem:poisson} 
and noting that 
$\gfb^n\subset \cH_{\fb}\subset \fo^n$, we may write
$$
S_\mathfrak{b}  (N;\m)
=\sigstar\sum_{\ma{a}\in \fo^n/{\cH_{\fb}}} \sigma(-N )
\sum_{\ma{h}\in \cH_\fb/\gfb^n}
\sigma(F(\ma{a}))\e(\m. \ma{a})\e(\m. \ma{h}).
$$
However, orthogonality of characters gives
\[ \sum_{\bfh \in \calH_\calb/\Gb^n} \psi( \m.\bfh) = 
\begin{cases}
  |\calH_\calb/\Gb^n| & \text{if $ \bfm.\bfh \in \cald^{-1}~\forall \bfh \in \calH_\calb/\gfb^n$,} \\
 0 & \text{otherwise}.
\end{cases}
\]
Since we automatically have 
$\m.\h\in \fd^{-1}$ for any  
$\m \in 
 \widehat{\gfb  }^n$ and 
$\h\in \gfb^n$,   the statement of the lemma follows. 
\end{proof}

We shall also need to establish a multiplicativity property for the exponential sums.
 This is achieved in the following result. 

\begin{lemma}\label{lem:mult}
Let  $\calb$ be a non-zero integral ideal and 
suppose that 
 $\fb=\fb_1\fb_2$ for  integral ideals $\fb_1,\fb_2$, such that 
 $
\gcd (\n\fb_1,\n\fb_2)=1.
 $
 Then, for any $N\in \fo$ and 
 any $\m\in \widehat{\gfb}^n$, we have
$$S_\fb(N;\bfm)=
  S_{\fb_1}  (\bar{\n\fb_2}^2N;(\n\fb_2)\m)  S_{\fb_2}  (\bar{\n\fb_1}^2N;(\n\fb_1)\m).
$$
\end{lemma}

\begin{proof}
According to Lemma \ref{lem:orthogonal}, there exists $\gamma=g/\alpha \in \mathfrak{F}(\fb)$ such that 
$\psi(\gamma \cdot)$ is a primitive character modulo $\fb$. Then,  \eqref{eq:SUM} implies that 
$$
S_{\fb}  (N;\m)
=
\sum_{a\in (\fo/\fb)^*} \psi(-\gamma a N)
\sum_{\x\bmod{\gfb}}\psi\left( \gamma a F(\x)+   \m. \x\right).
$$
Let us write   $\n\fb_i=b_i$ for $i=1,2$. The 
assumption $\gcd (b_1,b_2)=1$ implies that 
$(\gfb_1,\gfb_2)=1$.
Moreover, we have 
$b_1\in \fb_1$, $b_2\in \fb_2$ and 
\begin{equation}\label{eq:toast}
((b_1),\fb_2)=((b_2),\fb_1)=1.
\end{equation}
According to Lemma \ref{lem:alg1}(i)
we find   elements $\lambda, \mu\in \fo$ such that $\ord_{\fp}(\lambda)=\ord_\fp(\fb_1)$ and $\ord_{\fp}(\mu)=\ord_\fp(\fb_2)$ for all $\fp \mid \gfb_1\gfb_2$. 
It follows from the Chinese remainder theorem, in the form
Lemma \ref{lem:skinner-3}, that we can write
 $a=\mu b+\lambda c$ 
 for $b \bmod{\fb_1}$ 
and $c \bmod{\fb_2}$. 
Likewise, we claim that we can write  $\x=b_2\b+b_1\c$, 
for $\b \bmod{\gfb_1}$ and 
$\c \bmod{\gfb_2}$.
To prove the claim it suffices to show that there is an isomorphism 
$\fo/\gfb_1\times \fo/\gfb_2\to \fo/\gfb$, given by $(u,v)\mapsto b_2u+b_1v$.
This map is clearly well-defined,  since $b_1\in \gfb_1$ and $b_2\in \gfb_2$.
Moreover, injectivity follows from the coprimality conditions 
$((b_2),\gfb_1)=((b_1),\gfb_2)=1$, which are a direct consequence of \eqref{eq:toast}.
The claim follows, since the cardinalities are the same, by the Chinese remainder theorem.

In summary, on observing that $b_1\in \gfb_1$ and $b_2\in \gfb_2$, it follows that 
\begin{align*}
 S_{\fb}(N;\m)
=~&\sum_{\substack{b\in (\fo/\fb_1)^*\\ c\in (\fo/\fb_2)^*}}
 \psi(-\gamma (\mu b+\lambda c) N)\\
 &\quad \times
\sum_{\substack{\b\bmod{\gfb_1}\\ \c\bmod{\gfb_2}}}
\e\left( \gamma (\mu b+\lambda c)F(b_2\b+b_1\c)+ \m.(b_2\b+b_1\c)\right)\\
=~&\sum_{b\in (\fo/\fb_1)^*} \psi(-\gamma \mu b N)\sum_{\b\bmod{\gfb_1}}\e\left( \gamma \mu b_2^2b F(\b)+b_2\m.\b\right) \\
&\quad\times \sum_{c\in (\fo/\fb_2)^*}\psi(-\gamma \lambda c N)\sum_{\c\bmod{\gfb_2}}
\e\left(  \gamma \lambda b_1^2 c F(\c)+ b_1\m.\c\right).
\end{align*}
We claim that $\psi(\gamma \mu b_2^2\cdot)$ defines a primitive character modulo $\fb_1$.
For this we note that 
$$
\beta\in \fa_{\gamma\mu b_2^2} \Leftrightarrow \gamma \mu b_2^2 \beta\in \fo \Leftrightarrow (g \mu b_2^2\beta)\subset (\alpha)\Leftrightarrow \fb_1\fd\mid (\mu b_2^2\fb_2^{-1})(g\fp_1^{-1})(\beta),
$$
since $b_2\in \fb_2$ and $g\in \fp_1$. Now $(g)$ is coprime to $\fb_1\fd$ and 
$(\mu b_2)$ is coprime to $\fb_1$. Thus it follows that 
$$
\beta\in \fa_{\gamma\mu b_2^2} \Leftrightarrow  \beta\in \fb_1\fe,
$$
where $\fe=\fd/(\fd,\mu b_2^2\fb_2^{-1})$. Clearly $\fe\mid \fd$.
We claim that $(\fd/\fe,\fb_1)=1$. To see this, note that 
$\fd/\fe$ is equal to the common divisor  $(\fd, \mu b_2^2 \fb_2^{-1})$. Now $\mu$ is coprime to $\fb_1$ and so is $\fb_2$. Hence the common divisor of these ideals most be coprime to the ideal $\fb_1$, as claimed.
Thus  
Lemma \ref{lem:denominator} 
establishes the claim that 
$\psi(\gamma \mu b_2^2\cdot)$ is a primitive character modulo $\fb_1$.
It follows that 
$$
\sum_{b\in (\fo/\fb_1)^*}
\psi(-\gamma \mu b N)
\sum_{\b\bmod{\gfb_1}}\e\left( \gamma \mu b_2^2b F(\b)+b_2\m.\b\right) 
=  S_{\fb_1}  (\bar{\n\fb_2}^2N;(\n\fb_2)\m),
$$
where $\bar{\n\fb_2}$ is the multiplicative inverse of 
$\n\fb_2$ modulo $\fb_1$.
Similarly,
$$ \sum_{c\in (\fo/\fb_2)^*}
\psi(-\gamma \lambda c N)
\sum_{\c\bmod{\gfb_2}}
\e\left(  \gamma \lambda b_1^2 c F(\c)+ b_1\m.\c\right)=  
S_{\fb_2}  (\bar{\n\fb_1}^2N;(\n\fb_1)\m),
$$
from which the lemma follows.
\end{proof}

\begin{corollary}\label{cor:mult}
Let  $\calb$ be a non-zero integral ideal and 
suppose that 
 $\fb=\fb_1\fb_2$ for  integral ideals $\fb_1,\fb_2$, such that 
 $
\gcd (\n\fb_1,\n\fb_2)=1.
 $
 Then 
$S_\fb(N;\0)=
  S_{\fb_1}  (N;\0)  S_{\fb_2}  (N;\0).
$
\end{corollary}

\begin{proof}
On making an obvious change of variables to the $a$-sum and the $\x$-sum in \eqref{eq:SUM}, we note that   
$S_\fb(c^2N;\0)=S_\fb(N;\0)$ for any $c\in \ZZ$ which is coprime to $\fb$. 
The statement now follows from an application of Lemma \ref{lem:mult}.
\end{proof}

\subsection{The exponential integral}\label{s:integrals}

In this section we discuss the exponential integral 
$I_\fb(t;\k)$  that appears in Lemma \ref{lem:poisson}, for given $t\in V$ and $\k\in V^n$.
It will be convenient to set 
$$
0<\rho=\frac{\n \mathfrak{b}}{P^d}\ll 1,
$$
with which notation we have
$$
I_\fb(t;\k)=\int_{ V^{n}}W(\x)h\left(\rho  ,|\nm(F(\x)-t)|\right)\e\left(-\k.\x \right)\d\x.
$$
We now bring into play the work in \cite[\S 6]{BV}. It follows from an application of Fourier inversion, as in \cite[Eq.~(6.3)]{BV}, that there exists a function $p_\rho(v):V\to \CC$ such that 
\begin{equation}\label{eq:integral-step1}
I_\fb(t;\k)=\int_{ V} p_\rho(v) \psi(-vt) K(v,\k)
 \d v,
\end{equation}
where
\begin{equation}\label{eq:integral-K}
K(v,\k)=\int_{ V^{n}}W(\x)\e\left(vF(\x)-\k.\x \right)\d\x.
\end{equation}

In our analysis it will be useful to have the notion of a height function on  $V$. 
Accordingly, we define $\mathfrak{H}:V\to \RR_{\geq 1}$ via
$$
\mathfrak{H}(v)=\prod_{l=1}^d \max\{1,|v_l|\},
$$
for $v=(v_1,\dots,v_d)\in V$.
In the closing stages of our argument we will need to estimate integrals involving 
powers of $\mathfrak{H}(v)$ over 
various regions in $V$. 
First, it  follows from  \cite[Lemma 5.3]{BV} that 
\begin{equation}
\label{eq:BoundC}
\int_V \mathfrak{H}(v)^{\alpha}\d v \ll 1 \quad \text{ if $\alpha<-1$}.
\end{equation}
We can use this to deduce two further bounds that will play  important roles.

For any $A\geq 1$ and $\ve>0$, we claim that
\begin{equation}
\label{eq:BoundA}
\int_{\{v\in V: \mathfrak{H}(v)\geq A\}} \mathfrak{H}(v)^{\alpha}\d v \ll A^{\alpha+1+\ve} \quad \text{ if $\alpha<-1$}.
\end{equation}
If $\alpha<-1$ then we can clearly  assume  that $\ve<-\alpha-1$. But then the conditions of integration imply that  $(\mathfrak{H}(v)/A)^{-\alpha-1-\ve}\geq 1$, whence
$$
\int_{\{v\in V: \mathfrak{H}(v)\geq A\}} \mathfrak{H}(v)^{\alpha}\d v \leq A^{\alpha+1+\ve} 
\int_{V} \mathfrak{H}(v)^{-1-\ve}\d v \ll A^{\alpha+1+\ve} ,
$$
by \eqref{eq:BoundC}.

Next, for any $B\geq 1$ and $\ve>0$, we claim that  
\begin{equation}
\label{eq:BoundB}
\int_{\{v\in V: \mathfrak{H}(v)\leq B\}} \mathfrak{H}(v)^{\alpha} \d v \ll B^{\alpha+1+\ve} \quad \text{ if $
\alpha\geq -1$}.
\end{equation}
To see this we note that 
$(B/\mathfrak{H}(v))^{\alpha+1+\ve}\geq 1$, under the conditions of the integral,
if $\alpha\geq -1$. But then 
$$
\int_{\{v\in V: \mathfrak{H}(v)\leq B\}} \mathfrak{H}(v)^{\alpha} \d v \leq B^{\alpha+1+\ve} 
\int_{V} \mathfrak{H}(v)^{-1-\ve}\d v  \ll B^{\alpha+1+\ve} ,
$$
by \eqref{eq:BoundC}.

Returning to the function 
$p_\rho(v)$ in \eqref{eq:integral-step1}, 
the following result summarises 
its key properties
 and is extracted from \cite[Lemmas 6.3 and 6.4]{BV}.

\begin{lemma}
For any $\ve>0$, 
we have 
$
p_\rho(v)\ll  P^\ve,
$
for any $v\in V$. Moreover, for any $\ve>0$ and 
 $A\geq 1$, we have 
$$
p_\rho(v)\ll_A \rho^{-1} \left(\rho^{-1}P^\ve \mathfrak{H}(v)^{-1}\right)^A.
$$
\end{lemma}

Recall here that $\rho>0$. 
The next result is a straightforward consequence of the previous result, once combined with 
\eqref{eq:integral-step1} and the bound 
$$
|K(v,\k)|\leq \int_{ V^{n}}W(\x)\d\x\ll 1,
$$
which follows from the fact that $W$ is compactly supported.

\begin{corollary}\label{cor:integral-bound}
Let $\ve>0$.
Let  $t\in V$ and $\k\in V^n$. Then 
$$
I_\fb(t;\k)\ll_A P^\ve \int_{ \mathcal{U}} |K(v,\k)|
 \d v  +P^{-A},
$$
for any $A\geq 1$, where
$$
\mathcal{U}=\mathcal{U}_\ve=\left\{v\in V: 
\mathfrak{H}(v)\leq \frac{P^{d+\ve}}{\n\fb}\right\}.
$$
\end{corollary}

It is interesting to pause and reflect on the corresponding situation for cubic forms $G$ over a number field $K$ that was considered in \cite{BV}, recalling that we are assuming $K$ to be totally real in our setting. In \cite{BV}, crucial  use was made of the fact that 
the   integral over $\x$ factors as 
$$
\prod_{1\leq l\leq d} \int_{\RR^{n}}W^{(l)}(\x^{(l)})e\left(v^{(l)}G^{(l)}(\x^{(l)})-\k^{(l)}.\x^{(l)} \right)\d\x^{(l)},
$$
since 
$\tr (vG(\x) )=\sum_{l=1}^{d}v^{(l)}G^{(l)} (\x^{(l)})$, 
where
$G^{(l)}=\rho_l(G)$ is a cubic form over $\RR$. 
We have chosen our main example \eqref{eq:special-0} in order that a similar property holds. 
Such a factorisation is not necessarily
enjoyed for arbitrary 
generalised quadratic forms $F$, however,  and it seems very difficult  to analyse the integrals $K(v,\k)$ in generic situations.

Define
\begin{equation}\label{eq:real-Q}
\mathcal{Q}(\x^{(1)},\dots,\x^{(d)})=
\sum_{1\leq l \leq d} v_l F^{(l)}(\x^{(1)},\dots,\x^{(d)}),
\end{equation}
for fixed $v\in V$, 
where $F^{(l)}$ is the  quadratic form \eqref{eq:Ql}.
Thus $\mathcal{Q}$ 
is a  quadratic form over $\RR$ in $dn$ variables.
Let us write, temporarily, $\underline\x=(\x^{(1)},\dots,\x^{(d)})$ and 
$
\underline\k=(\k^{(1)},\dots,\k^{(d)}).
$
Then, in the light of \eqref{eq:sol},  we may write
\begin{equation}\label{eq:K(um)}
K(v,\k)
=\int_{ \RR^{dn}}W(\underline\x)
e\left(\mathcal{Q}(\underline\x)-\underline\k.\underline\x\right)\d\underline\x.
\end{equation}
A general study of these exponential integrals has been carried out by 
Heath-Brown and Pierce \cite[Lemma 3.1]{HBP}. Assuming that the support of $W$ is contained in $[-1,1]^{dn}$, we may  appeal to their work, which we record here for the convenience of the reader.

\begin{lemma}\label{lem:HBP}
Let $\mathcal{Q}\in \RR[X_1,\dots,X_m]$ be a quadratic form with  
 coefficients of maximum modulus $\|\cQ\|$ and 
 eigenvalues $\rho_1,\dots,\rho_m$.
Let $\bla\in \RR^m$ and suppose that $w:\RR^m\to \RR$ is any smooth weight function supported on $[-1,1]^m$. Then 
$$
\int_{\RR^m} w(\u) e\left(\cQ(\u)-\bla.\u\right) \d \u \ll _w \prod_{i=1}^m \min\left\{1,|\rho_i|^{-1/2}\right\}.
$$
Furthermore, if $|\bla|\geq 4\|\cQ\|$ then the integral is 
$O_{w,A} (|\bla|^{-A})$ 
for any $A\geq 1$.
\end{lemma}

We will apply this result with 
$\bla=\underline\k$ and with
the real quadratic form  in \eqref{eq:real-Q}.
Note that $\|\cQ\|\ll | v|$.
Next, define 
$$
\cF(v)=\det \left(
\sum_{1\leq l \leq d} v_l \mathbf{M}^{(l)}\right),
$$
where $\mathbf{M}^{(l)}$ is the $dn\times dn$ matrix associated to  $F^{(l)}$.
The function $\cF(v)$ is a real form of degree $dn$ 
in the variables $v_1,\dots,v_d$. The  following estimate is a direct consequence of  Lemma \ref{lem:HBP}.

\begin{corollary}
Assume  $| \k| \gg | v|$. Then 
$
K(v,\k)\ll_A  |\k|^{-A},
$
for any  $A\geq 1$. Moreover, 
$
K(v,\k)\ll \min\{1,|\cF(v)|^{-1/2}\}
$, for any $\k\in V^n$.
\end{corollary}

Unfortunately, 
it appears difficult to  extract anything useful from the second  bound, unless the generalised quadratic form is assumed to have extra structure.

\subsection{Contribution from the trivial character}\label{s:trivial}

In this  section we study the overall contribution 
from the vector $\m=\0$ in  the expression for 
$\No$ in Lemma \ref{lem:poisson}. This contribution is 
$$
M(P)=
\frac{P^{(n-2)d}}{D_{K}^{n/2}} 
\sum_{
\substack{0\neq \fb\subset \fo\\
\n\mathfrak{b}\ll P^d} }
(\n \gfb  )^{-n}
S_\mathfrak{b}  (N;\0)
I_\mathfrak{b}  (N/P^2;\0) 
+O_A(P^{-A}),
$$
in the notation of that result. 

It will ease notation if we put $t=N/P^2\in \RR$. 
Assuming that the descended system has codimension $d$, 
we begin by analysing the exponential integral $I_\fb(t;\0)$, writing 
\begin{equation}\label{eq:zug}
I_\mathfrak{b}  (t;\0)
=
\int_{ V^{n}}W(\x)h\left(\rho  ,|\nm(F(\x)-t)|\right)\d\x
=\int_{ V} f(v)h\left(\rho  ,|\nm(v)|\right)\d\v,
\end{equation}
where $\rho=\n\fb/P^d$ and 
$$
f(v)=\int_{\substack{\x\in V^n\\ F(\x)-t=v}} W(\x)\d \x,
$$
where, by an abuse of notation,  $\d \x$ is the surface measure
 obtained by eliminating $d$ variables  from the equation
$F(\x)-t=v$. 
We shall think of $f(y)$ as a function of $\mathbf{y}=(y_1,\dots,y_d)$ on $\RR^d$, in which $t$ is fixed and bounded absolutely.
The following result summarises its main properties. 

\begin{lemma}\label{lem:I0}
Assume that the descended system has codimension $d$.
There exist positive  constants $C,C_0,C_1,\dots$ such that 
the function $f: \RR^d\to \RR$ is  a smooth weight function 
that is supported on $[-C,C]^d$ and satisfies 
$$
\left|\frac{\partial^{i_1+\cdots+i_d}}{\partial y_1^{i_1}\cdots \partial y_d^{i_d} }f (\y)\right| \leq C_{i_1+\cdots+i_d}
$$
for any $\y\in [-C,C]^d$ and any $i_1,\dots,i_d\geq 0$.
The constants $C,C_0,C_1,\dots$ depend only on the coefficients of $F$ and the parameter $\delta$ in the definition of $W$.
\end{lemma}

\begin{proof}
In the course of the proof it will be convenient to write 
$\underline{\mathbf{s}}=(\mathbf{s}_1,\dots,\mathbf{s}_d)$, 
$\underline{\u}=(\u_1,\dots,\u_d)$,  
$\underline{\bxi}=(\bxi_1,\dots,\bxi_d)$ and  $\mathbf{t}=(t_1,\dots,t_d)$.
Recall the definition of the weight function $W$ in \S\ref{s:W} for a suitable  fixed $\underline\bxi\in \RR^{dn}$.
Making the change of variables in \eqref{eq:W-transform}, we see that 
\begin{equation}\label{eq:hunch}
f(\y)
=D_K^{-n/2}  
\int_{\substack{\underline\u\in \RR^{dn}\\ Q_l(\underline\u)-\tau_l=w_l}} 
w(\delta^{-1}|\underline\u-\underline\bxi|)
\d \underline\u,
\end{equation}
where 
$\btau=\mathbf{A}^{-1} \t$ and 
$\mathbf{w}=\mathbf{A}^{-1}\y$.
It will clearly suffice to prove the properties 
recorded in the lemma for the 
integral on the right hand side, 
$\widetilde{f}({\mathbf{w}})$ say, 
regarded as a function of $\mathbf{w}$.
Making the change of variables $\underline{\mathbf{s}}=\underline\u-\underline\bxi$, we 
have 
$$
\widetilde f(\mathbf{w})
=
\int_{\substack{\underline{\mathbf{s}}\in \RR^{dn}\\ Q_l(\underline{\mathbf{s}}+\underline\bxi)-\tau_l=w_l}} 
w(\delta^{-1}|\underline{\mathbf{s}}|)
\d \underline{\mathbf{s}}.
$$
It is now clear that $\widetilde{f}(\mathbf{w})=0$ unless $\w\in [-C,C]^d$ for suitable $C>0$.

Next, 
we recall that $J_{Q_1,\dots,Q_d}(\underline\bxi)$ has rank $d$. 
We may assume without loss of generality that 
$$
\det \left(\frac{\partial }{\partial U_{j1}}Q_i(\underline\bxi)\right)_{\substack{1\leq i,j\leq d}}\neq 0.
$$
Let $\phi:\RR^{dn}\to \RR^{dn}$ be given by 
$$
\underline{\mathbf{s}} \mapsto \left(Q_1(\underline{\mathbf{s}}+\underline\bxi)-\tau^{(1)},s_{1,2},\dots,s_{1,n},\dots ,
Q_d(\underline{\mathbf{s}}+\underline\bxi)-\tau^{(d)},s_{d,2},\dots,s_{d,n}
\right).
$$
The implicit function theorem implies that  there exist  open subsets
$W',W\subset \RR^{dn}$ with $\underline\0\in W'$ and 
$\phi(\underline\0)\in W$, such that 
$\phi:W'\to W$ is a bijection and has differentiable inverse $\phi^{-1}$ on $W$.
It is now clear that we wish to choose $\delta>0$ small enough to ensure  that $\underline{\mathbf{s}}\in W'$ whenever 
$|\underline{\mathbf{s}}|\leq \delta$.

We may now conclude that
$$
\widetilde{f}(\mathbf{w})
=
\int_{\substack{\underline{\mathbf{s}}'\in \RR^{d(n-1)}}}
 \partial_1 \phi^{-1}
w(\delta^{-1}
|\left(s_{1,1},\dots,s_{d,n}
\right)|)
\d \underline{\mathbf{s}}',
$$
where
$\underline{\mathbf{s}}'=
(s_{1,2},\dots,s_{1,n},\dots ,s_{d,2},\dots,s_{d,n})
$, 
and  $s_{1,1},s_{2,1},\dots, s_{d,1}$ are implicitly given by $\underline{\mathbf{s}}'$ and $\w$, and 
$$\partial_1 \phi^{-1}=\left.
\det \left(\frac{\partial (\phi^{-1})_{in+1-n}}{\partial w_j}   \right)_{1\leq i,j\leq d}\right|_{(s_{1,1},\dots ,
s_{d,n})}
$$
is the associated Jacobian. Since $\phi^{-1}$ is smooth this implies that 
$\widetilde{f}(\mathbf{w})$ is infinitely differentiable and that its partial derivatives satisfy the bound claimed in the lemma. 
\end{proof}

Now it follows from 
Corollary
\ref{cor:integral-bound} that  for $t=N/P^2\in \RR$, and $\ve$ fixed as in the corollary, we have 
$$
I_\fb(t;\0)\ll  \frac{P^{d+2\ve}}{\n\fb}.
$$
Furthermore, in view of 
Lemma \ref{lem:I0}, it follows from \eqref{eq:zug} and  \cite[Lemma~4.1]{BV}
that 
$$
I_\mathfrak{b} (t;\0)
=
\sqrt{D_K} f(0) +O_A\left(\left(\frac{\n\fb}{P^d}\right)^A\right),
$$
for any $A\geq 0$, 
where 
$$
f(0)
=
\int_{\substack{\underline\x\in \RR^{dn}\\ F^{(l)}(\underline\x)=t_l}} W(\underline\x)\d \underline\x,
$$
if $\underline\x=(\x_1,\dots,\x_d)$.
According to   \eqref{eq:hunch}, we have 
$f(0)=
D_K^{-n/2}  \sigma_\infty(t),$
where
\begin{equation}\label{eq:sig-infinity}
\sigma_\infty(t)=
\int_{\substack{\underline\u\in \RR^{dn}\\ Q_l(\underline\u)=\tau_l}} 
w(\delta^{-1}|\underline\u-\underline\bxi|)
\d \underline\u
\end{equation}
is the usual  singular integral for the descended system. 
In particular, arguing as in Davenport and Lewis \cite[\S 6]{dl}, 
a standard argument ensures that  $\sigma_\infty(t)>0$, since $\underline\bxi$ 
is a non-singular real point on the descended system.

We summarise our preliminary treatment of the main term in the following result.

\begin{lemma}\label{lem:main_term}
Let $\ve>0$. Then, for any $A\geq 1$, we have 
\begin{align*}
M(P)=~&
\frac{P^{(n-2)d} }{D_{K}^{n-1/2}}  
\sum_{\substack{0\neq \fb\subset \fo\\ \n\fb\ll P^d}}
(\n \gfb  )^{-n}
S_\mathfrak{b}  (N;\0)
\left(\sigma_\infty(N/P^2)
+O_A\left(\left(\frac{\n\fb}{P^d}\right)^A\right)\right) \\
&\qquad
+O_A(P^{-A}),
\end{align*}
where
 $\sigma_\infty(N/P^2)>0$ is  given by \eqref{eq:sig-infinity}.
\end{lemma}

In order to proceed further, it is clear that one requires a good enough upper bound for $S_\fb(N;\0)$, in order to show that the error term is satisfactory and the sum over $\fb$ can be extended to infinity.
Such a bound is available for admissible $F$, thanks to Corollary~\ref{cor-expsum1}. Although we omit details, 
one can use it to prove that 
$$
M(P)=
\frac{\sigma_\infty(N/P^2) }{D_{K}^{n-1/2}}  P^{(n-2)d}
\sum_{\substack{0\neq \fb\subset \fo}}
(\n \gfb  )^{-n}
S_\mathfrak{b}  (N;\0)
+O(P^{dn/2+\ve}),
$$
for any admissible $F$ such that $n\geq 5$.
In the setting of Theorems \ref{t:non-diag} and \ref{t:diag} we shall produce better bounds for 
$S_\mathfrak{b}  (N;\0)$ which allow such a deduction under milder hypotheses. 

We close this section with a formal analysis of the 
singular series
$$
\mathfrak{S}(N)=\sum_{(0)\neq \fb\subset \fo} (\n\gfb)^{-n}S_\mathfrak{b}  (N;\0),
$$
ignoring issues of convergence. This is summarised in the following result. 

\begin{lemma}
Assume that $\mathfrak{S}(N)$ is absolutely convergent. Then
$$
\mathfrak{S}(N)=\prod_p \lim_{\ell\to \infty} 
p^{-d\ell(n-1)}\#\left\{\x\in (\fo/p^\ell\fo)^n: F(\x)\equiv N\bmod{p^\ell}\right\}.
$$
We have $\mathfrak{S}(N)>0$ if the shifted descended system has a non-singular $p$-adic solution for every prime $p$.
\end{lemma}

\begin{proof}
We may write
$$
\mathfrak{S}(N)
=\sum_{k=1}^\infty
\sum_{\substack{ \fb\subset \fo\\ 
\n\fb=k
}} (\n\gfb)^{-n}S_\mathfrak{b}  (N;\0)
=\sum_{k=1}^\infty S(k),
$$
say. It follows from Corollary \ref{cor:mult}
that $S(k_1k_2)=S(k_1)S(k_2)$ if $k_1,k_2$ are coprime integers. 
Hence 
$$
\mathfrak{S}(N)=\prod_p \sum_{j\geq 0}S(p^j).
$$
Since $K$ is Galois we may assume that $p$ admits a factorisation $(p)=(\fp_1\cdots \fp_r)^e$, with $\n \fp_1=\cdots =\n\fp_r=p^f$. 
Let $\ell\geq 0$ and let $I_\ell$ denote the set of integral ideals $\fb=\fp_1^{k_1}\cdots \fp_r^{k_r}$,  with 
$0\leq k_i\leq \ell e$ for $1\leq i\leq r$. Then the union of $I_\ell$ over $\ell\geq 0$ exactly matches the set of integral ideals whose norm is a power of $p$. 
Hence
$$
\mathfrak{S}(N)=\prod_p \lim_{\ell\to \infty} 
\sum_{\substack{ \fb\in I_\ell}} (\n\gfb)^{-n}S_\mathfrak{b}  (N;\0).
$$

It follows from \eqref{eq:SUM} that 
\begin{align*}
S_\mathfrak{b}  (N;\0)
&=\sum_{a\in (\fo/\fb)^*} \psi(-\gamma a N)
\sum_{\x\bmod{\gfb}}\psi\left( \gamma a F(\x)\right)\\
&=\frac{(\n\gfb)^{n}}{p^{\ell dn}}
\sum_{a\in (\fo/\fb)^*} \psi(-\gamma a N)
\sum_{\x\bmod{p^\ell\fo}}\psi\left( \gamma a F(\x)\right),
\end{align*}
on extending the inner sum 
to a sum over elements of 
$(\fo/p^\ell\fo)^n$.
Hence, on rearranging, we obtain
$$
\sum_{\substack{ \fb\in I_\ell}} (\n\gfb)^{-n}S_\mathfrak{b}  (N;\0)=
\frac{1}{p^{\ell dn}} 
\sum_{\x\bmod{p^\ell\fo}}
E(p),
$$
where
$$
E(p)=
\sum_{\substack{ \fb\in I_\ell}} 
\sum_{a\in (\fo/\fb)^*} \psi\left(\gamma a (F(\x)-N)\right).
$$
We claim that $\psi(\gamma a\cdot )$ runs over all characters modulo $p^\ell$, as $a$ runs over 
$(\fo/\fb)^*$ and $\fb$ runs over $I_\ell$. 
On one hand, since
$a\in (\fo/\fb)^*$, 
 each character 
$\psi(\gamma a\cdot )$ is a primitive character modulo $\fb$. Since 
$\fb=\fp^{k_1}\cdots \fp_r^{k_r}$, for $k_1,\dots,k_r\leq \ell e$ and $p^\ell=(\fp_1\cdots\fp_r)^{\ell e}$, we 
conclude that each such character  induces a character modulo $p^\ell$.  
In order to complete the proof of the claim it remains to show that we get all $p^{\ell d}$ characters modulo $p^\ell$ this way. 
But the number of characters is precisely
\begin{align*}
\sum_{\fb\in I_\ell} \n\fb \prod_{\fp\mid \fb} \left(1-\frac{1}{\n\fp}\right) &=
\sum_{\fb\in I_\ell} p^{f(k_1+\cdots+k_r)}\prod_{\fp\mid \fb} \left(1-\frac{1}{p^f}\right) \\
&=\prod_{1\leq i\leq r} \left(1+\sum_{1\leq k\leq \ell e} p^{fk} \left(1-\frac{1}{p^f}\right)\right) \\
&=p^{\ell d},
\end{align*}
as required.

We may now conclude from orthogonality of characters that 
$$
E(p)=\begin{cases}
p^{\ell d} & \text{ if $F(\x)\equiv N\bmod{p^\ell}$,}\\
0 & \text{ otherwise,}
\end{cases}
$$
from which the first part of the lemma follows.
The second part is standard. Using \eqref{eq:F-to-Q}, the 
solubility of $F(\x)- N$ in $\fo/p^\ell\fo$ can be reduced to the solubility of a shifted descended system 
$Q_i(\u_1,\dots,\u_d) -N_i$ modulo primes powers, for $1\leq i\leq d$, 
where we have written  $N=\omega_1N_1+\cdots+\omega_d N_d$. 
Arguing as in work of Birch \cite[Lemma 7.1]{birch}, for example, 
the existence of non-singular $p$-adic
zeros of this system is enough to deduce that  $\mathfrak{S}(N)>0$. The 
details of this will not be repeated here.  
\end{proof}

\subsection{Contribution from the non-trivial characters}\label{s:non-trivial}

In this section we make some initial steps in the treatment of the contribution from the non-zero vectors $\mathbf{m}$ in the asymptotic formula for $N_W(F,N;P)$ in Lemma~\ref{lem:poisson}.
This contribution is 
$$
\ll 
P^{(n-2)d} E(N;P),
$$ 
where
\begin{equation}\label{eq:draft}
E(N;P)=
\sum_{\substack{0\neq \fb\subset \fo\\ \n\fb\ll P^d}} 
\sum_{\substack{\0\neq \m\in \widehat{\gfb  }^n}}
(\n \gfb  )^{-n}
|S_\mathfrak{b}  (N;\m)|
|I_\mathfrak{b} (N/P^2;P\m)|.
\end{equation}
The primary  task is to establish conditions under which there is an absolute constant $\Delta>0$ such that $E(N;P)=O(P^{-\Delta})$.

We now place ourselves in the context of the generalised quadratic forms \eqref{eq:special-0} and make some initial steps that will be common to Theorems \ref{t:non-diag-1}--\ref{t:diag}. 
It will be convenient to consider the overall contribution from $\fb$ such that $\n\fb$ and $\n\gfb$ are constrained to lie in dyadic intervals. 
Note that $\n\fb\ll P^d$ and
$\n\gfb\leq (\n\fb)^2$, 
since $\#G=2$.
Accordingly, we let  $X,Y$ be parameters such that 
\begin{equation}\label{eq:XY}
1\leq X\leq Y\leq X^2, \quad X\ll P^d.
\end{equation}
We then write $E(N;P;X,Y)$ for the overall contribution to $E(N;P)$ from non-zero ideals $\fb\subset \fo$ for which 
$$
X\leq \n\fb <2X \quad \text{ and } \quad Y\leq \n\gfb <2Y.
$$
We denote by $\mathcal{B}(X,Y)$ the set of all such ideals.
On summing over dyadic intervals for $X,Y$ satisfying \eqref{eq:XY}, it will suffice to establish the existence of $\Delta>0$ such that 
\begin{equation}\label{eq:end}
E(N;P;X,Y)=O(P^{-\Delta}),
\end{equation}
for any $X,Y$ satisfying \eqref{eq:XY}.

It follows from  Corollary \ref{cor:integral-bound} that 
$$
I_\mathfrak{b} (N/P^2;P\m) \ll_A
P^\ve \int_{\mathcal{U}} 
|K(u,P\m)|\d u +P^{-A},
$$
for any $A\geq 1$, 
where $K(u,P\m)$ is given by \eqref{eq:integral-K} and 
\begin{equation}\label{eq:breeze}
\mathcal{U}=
\left\{u\in V: \mathfrak{H}(u)\leq \frac{P^{d+\ve}}{\n\fb}\right\}.
\end{equation}
Hence  \eqref{eq:draft} yields
\begin{equation}\label{eq:ENPXY}
\begin{split}
E(N;P;X,Y)\ll_A  P^{-A}+
P^\ve Y^{-n}
\sum_{\substack{\fb\in \mathcal{B}(X,Y)}}
\sum_{\substack{0\neq \m\in \widehat{\gfb  }^n}}
|S_\mathfrak{b}  (N;\m)|
\int_{\mathcal{U}}
|K(u,P\m)|\d u ,
\end{split}
\end{equation}
for any $A\geq 1$, 
where $\mathcal{B}(X,Y)$ is the set of non-zero ideals $\fb\subset \fo$ for which 
$
X\leq \n\fb <2X$ and $Y\leq \n\gfb <2Y$.

\section{Homogeneous case: proof of Theorems 
\ref{t:non-diag-1} and 
\ref{t:non-diag}}\label{s:hom}

We begin by proving a general result about rank drop in pencils of quadratic forms in situations where one of the matrices has much smaller rank. It parallels the basic fact in
Reid's thesis \cite[Prop.~2.1]{reid} about 
rank drop in pencils $\nu_1A+\nu_2 B$, for suitable $n\times n$ matrices $A,B$, and shows how 
Assumption  \ref{ass2} can be deduced from an appropriate hypothesis about the shape of the associated singular locus.

\begin{lemma}\label{lem:reid}
Let $L$ be an algebraically closed field of characteristic not equal to $2$, and let  $m<n$. Consider two matrices $A,B\in M_{n\times n}(L)$ such that $B$ has only non-zero entries in the upper left $m\times m$ submatrix, which we also assume to be non-singular. Let $\det(A)\neq 0$. Assume that all singular points of the intersection of the two  quadratic forms associated to $A$ and $B$ have the shape $(0,\mathbf{x}'')$ with $\mathbf{x}''=(x_{m+1},\ldots, x_n)$, and that the intersection has codimension $2$.
Then we have
$$\rank (A+\lambda B) \geq n-1,\quad \forall \lambda \in L.$$
\end{lemma}

\begin{proof}
Assume that there is some $\lambda\in L$ with 
$$\rank(A+\lambda B) \leq n-2.$$
Let $V_0\subset \mathbb{A}^n$ be the affine subspace given by the kernel of $A+\lambda B$. Then $\dim V_0\geq 2$. Let $\mathbb{P}(V_0)=V\subset \mathbb{P}^{n-1}$ and let $Q_B\subset \mathbb{P}^{n-1}$ be the quadric given by the matrix $B$. Then $\dim V\geq 1$ and $\dim Q_B=n-2$ as projective varieties. We deduce that the intersection $V\cap Q_B$ is non-empty. Consider a point $\mathbf{x}=(\mathbf{x}',\mathbf{x}'')\in L^n\setminus\{0\}$ in the affine cone of $V\cap Q_B$, where $\mathbf{x}'\in L^m$ and $\bfx'' \in L^{n-m}$. Then we deduce that
$$0=\bfx^t (A+\lambda B) \bfx =\bfx^tA\bfx + \lambda \bfx^t B\bfx = \bfx^tA \bfx.$$
We deduce that $\bfx$ lies on the quadric given by $A$ and as it is in the kernel of $A+\lambda B$, it is a singular point of the intersection $Q_A \cap Q_B$. We claim that $\bfx'\neq 0$, i.e. $\bfx$ is not of the shape $(0,\bfx'')$. Assume for a moment that $\bfx=(0,\bfx'')$. Note that
$$0=(A+\lambda B) (0,\bfx'')= A (0,\bfx'').$$
This is a contradiction to $A$ being non-singular. Hence we found a singular point of the intersection $Q_A\cap Q_B$ which is not of the form $(0,\bfx'')$. 
\end{proof}

The main aim of  this section is to  carry out the proof of  Theorems \ref{t:non-diag-1} and 
\ref{t:non-diag}, which corresponds to taking $N=0$ and 
$$
F(X_1,\dots,X_n)=Q(X_1, \dots, X_n) +R(X_1^{\tau},\dots, X_m^{\tau}),
$$
as in \eqref{eq:special-0}.
Suppose that $\mathbf{A}$ is the $n\times n$ symmetric matrix defining $Q$ and that  $\mathbf{B}$ is the $n\times n$ symmetric matrix given by the condition that its upper left $m\times m$ submatrix defines $R$, with  all other entries are equal to $0$. We may proceed under the assumption that Assumptions \ref{ass1}--\ref{ass3} hold.

We have two tasks remaining. The first is to show that the sum over 
$\fb$ in Lemma~\ref{lem:main_term} can be extended to infinity, with acceptable error, and the second is to 
prove that \eqref{eq:end} holds.
We'll  need some more preparations for estimating the relevant  exponential sum in 
Lemma \ref{lem:main_term} and 
\eqref{eq:ENPXY}.
Recalling the
definition  \eqref{eq:define-H'} of 
 $\calH_\calb$, we lower 
 bound its index 
  in $\fo^n$. 
  
    \begin{lemma}\label{lem:1}
There exist non-zero constants $\kappa_1,\dots,\kappa_n, \tilde\kappa_1,\dots,\tilde\kappa_m\in K$, depending only on $F$ and $K$, such that 
 $$\calH_\calb\subseteq   (\kappa_1\fb\cap \tilde \kappa_1 \fb^{\tau^{-1}}) \times \cdots \times
(\kappa_m\fb\cap \tilde \kappa_m \fb^{\tau^{-1}} )
\times \kappa_{m+1}\fb\times \cdots\times \kappa_{n} \fb.
 $$
 Moreover, we have 
 $\kappa_1^{-1},\dots,\kappa_n^{-1}, \tilde\kappa_1^{-1},\dots,\tilde\kappa_m^{-1}\in \fo$.
\end{lemma}

\begin{proof}
Assume that $\mathbf{A}$ has symmetric entries $a_{i,j}\in \fo$, for $1\leq i,j\leq n$, and that 
$\mathbf{B}$ has 
symmetric entries $b_{i,j}\in \fo$, for $1\leq i,j\leq m$.
Then the associated bilinear form takes the shape
$$
B(X_1,\dots,X_n;Y_1,\dots,Y_n)=\sum_{i,j\leq n} a_{i,j} X_i Y_j +
\sum_{i,j\leq m} b_{i,j} X_i^\tau Y_j^\tau.
$$
Now 
 $\h\in \mathcal{H}_\fb$ if and only if
 $2B(\h,\k)\in \fb$ for all $\k\in \fo^n$. Let $\omega_1,\ldots, \omega_d$ be an integral basis of $\fo$ with $\ome_1=1$. Let $l\in \{1,\ldots, d\}$ and $j\in \{1,\dots,n\}$ and consider a vector $\k$ such that the $j$-th entry is equal to $\ome_l$ and all other entries are equal to zero. Then the condition 
 $B(\h,\k)\in \fb$ implies that 
 $$
2 \omega_l \sum_{i=1}^n a_{i,j} h_i + 2\ome_l^\tau \sum_{i=1}^m b_{i,j} h_i^\tau  \in \fb,\quad 1\leq l\leq d,\ 1\leq j\leq n.
 $$
 As the matrix $(\ome_l^\tau)_{1\leq l\leq d, \tau\in \Gal(K/\QQ)}$ is invertible, this implies that there exists $\beta\in K$ with $\beta^{-1}\in \fo$ such that
 $$ 
 \sum_{i=1}^n a_{i,j} h_i\in \beta \fb,\quad \sum_{i=1}^m b_{i,j} h_i^\tau \in \beta \fb,\quad 1\leq j\leq n.
 $$
 Thus we find that $\mathbf{A}\bfh\in (\beta\fb)^n$ and $\mathbf{B}(\h')^\tau \in (\beta\fb)^m$, where 
$\h'=(h_1,\dots,h_m).$
 As both matrices $\mathbf{A}$ and $\mathbf{B}$ are non-singular, this implies that 
 $$
\h\in \frac{1}{(\det \mathbf{A})} (\beta \fb)^n,\quad  \h' \in \frac{1}{(\det\mathbf{B})^{\tau^{-1}}} (\beta^{\tau^{-1}}\fb^{\tau^{-1}})^m.
$$
Putting these together, the statement of the lemma easily follows. 
\end{proof}

\begin{corollary}\label{cor:2}
Let $N\in \fo$ and let $F$ be given by \eqref{eq:special-0}. Suppose that Assumption \ref{ass1}
holds. 
Then 
$$
S_\mathfrak{b}  (N;\m)\ll 
  (\n\fb)^{1-(n-m)/2} 
(\n \gfb)^{n-m/2}.
$$
Moreover, $S_\mathfrak{b}  (N;\m)=0$ unless 
$
m_i\in \fd^{-1}\fb^{-1}$, for  $m<i\leq n$.
\end{corollary}
\begin{proof}
It follows from 
Lemma \ref{lem:1} that $|\fo^n/\calH_\fb|\gg (\n\gfb)^{m}(\n\fb)^{n-m}$. Thus \eqref{eq:dinky} implies that 
$$
|\calH_\grb/\gfb^n|= \frac{|\calO^n/\gfb^n|}{|\calO^n/{\calH}_\grb|} \ll \frac{(\n\gfb)^n}{
(\n\gfb)^{m}(\n\fb)^{n-m}}
=\left(\frac{\n\gfb}{\n\grb}\right)^{n-m}.
$$
Inserting this into  Lemma \ref{expsum1} yields  the desired upper bound. 
We have already observed in Lemma  \ref{lem:S-zero} that 
$S_\mathfrak{b}  (N;\m)=0$ unless $\m.\h\in \fd^{-1}$ for all $ \h\in \mathcal{H}_\fb$.
Noting that $\gfb^m\times \fb^{n-m}\subset \mathcal{H}_\fb$, the second part easily follows. 
\end{proof}

Returning to Lemma \ref{lem:main_term}, 
it immediately follows from this that 
the overall contribution from the tail  $\n\fb\gg P^d$ is 
\begin{align*}
&\ll 
P^{(n-2)d} 
\sum_{\substack{ \fb\subset \fo\\ \n\fb\gg P^d}}
(\n \gfb  )^{-n}
|S_\mathfrak{b}  (N;\0)|\\
&\ll 
P^{(n-2)d} 
\sum_{\substack{ \fb\subset \fo\\ \n\fb\gg P^d}}
  (\n\fb)^{1-n/2+m/2} 
(\n \gfb)^{-m/2}.
\end{align*}
Since $\n\gfb\geq \n\fb$, this is acceptable provided  that $n>4$, which is certainly implied by the hypotheses in Theorems \ref{t:non-diag-1} and \ref{t:non-diag}.
Thus we can focus our remaining efforts on establishing \eqref{eq:end}.

Our next goal is to analyse the integrals $K(u,P\m)$ 
in \eqref{eq:ENPXY}
for the case that $F$ has the shape 
$F(\x)=Q(\x)+R(x_1^\tau, x_2^\tau,\ldots, x_m^\tau),$
for $\tau\in \Gal(K/\QQ)$ some fixed automorphism. Taking the $l$th embedding into the real numbers gives
$$F^{(l)}(\x^{(1)},\ldots, \x^{(d)}) =Q^{(l)}(\x^{(l)})+R^{(l)}(\rho_l(\x^{(\tau)})),\quad 1\leq l\leq d$$
where we write $\x^{(l)}=\rho_l(\x)$. For each $1\leq l\leq d$, we define $l_\tau$ through the relation
\eqref{eq:l_tau}.
With this notation we obtain
\begin{align*}
\mathcal{Q}(\underline\x)&= \sum_{1\leq l \leq d} u_l F^{(l)}(\underline\x) \\
&= \sum_{1\leq l\leq d} u_lQ^{(l)}(\x^{(l)}) + \sum_{1\leq l\leq d} u_l R^{(l)}(\rho_l(\x^{\tau}))\\
&=  \sum_{1\leq l\leq d} u_lQ^{(l)}(\x^{(l)}) + \sum_{1\leq l\leq d} u_{l_\tau} R^{(l_\tau)}(\x^{(l)}).
\end{align*}
Hence
$$
K(u,P\m)=\prod_{l=1}^d \int_{\R^n} W(\x^{(l)}) e(G^{(l)}(\x^{(l)})-P\m^{(l)}.\x^{(l)}) \d\x^{(l)},
$$
with 
$$G^{(l)}(\x^{(l)})= u_lQ^{(l)}(\x^{(l)})+u_{l_\tau}R^{(l_\tau)}(\x^{(l)}).$$

Note that $G^{(l)}(\x^{(l)})$ is a quadratic form in $\x^{(l)}$ and hence can be represented by a symmetric matrix, which can be diagonalised using an orthogonal base change. Thus, for every tuple $u=(u_1,\ldots, u_d)$, there exists a diagonal matrix $\diag(\eth_{l,i}(u))_{1\leq i\leq n}$ and an orthogonal matrix $M_l(u)\in O(n)$ such that 
$$
G^{(l)}(\x^{(l)})= (\x^{(l)})^t M_l(u)^t \diag(\eth_{l,i}(u)) M_l(u)\x^{(l)}. 
$$
Set
$$
K^{(l)}(u,P\m)= \int_{\R^n} W(\x^{(l)}) e(G^{(l)}(\x^{(l)})-P\m^{(l)}.\x^{(l)})\d\x^{(l)}, \quad \text{ for $1\leq l\leq d$}.
$$
With the change of coordinates $M_l(u)\x^{(l)}=\y^{(l)}$, we get
\begin{align*}
K^{(l)}&(u,P\m)\\
&=\pm \int_{\R^n} W(M_l(u)^t\y^{(l)}) e((\y^{(l)})^t \diag(\eth_{l,i}(u))\y^{(l)}-P\m^{(l)}.(M_l(u))^t \y^{(l)})\d \y^{(l)}\\
&= \pm \int_{\R^n} W(M_l(u)^t\y^{(l)}) e((\y^{(l)})^t \diag(\eth_{l,i}(u))\y^{(l)}-PM_l(u)\m^{(l)}. \y^{(l)})\d \y^{(l)}.
\end{align*}
We are now ready to prove the following result. 

\begin{lemma}
For any $\ve>0$, 
the integral $K^{(l)}(u,P\m)$ is essentially supported on the set of $u$ and $\m$ for which 
$$|(M_l(u)\m^{(l)})_i|\ll P^{-1+\ve}| \eth_{l,i}(u)|,\quad 1\leq i\leq n,$$
and
$$|m_i^{(l)}|\ll P^{-1+\ve} |u_l|,\quad m< i\leq n.$$
Moreover, we have 
$$
K^{(l)}(u,P\m)\ll   \prod_{i=1}^n \min\left( 1, \frac{1}{|\eth_{l,i}(u)|^{1/2}}\right).$$
\end{lemma}

\begin{proof}
Recall that $M_l(u)\in O(n)$. In particular all entries of $M_l(u)$ are bounded independently  of $u$ and  we obtain
$$
\frac{\partial^k}{\partial (y_i^{(l)})^k} W(M_l(u)^t \y^{(l)}) \ll_k 1,
$$
uniformly in $u$, for all $k\in \mathbb{N}$.
The result now follows from Lemma \ref{lem:HBP}.
\end{proof}

Henceforth we take $N=0$ and write
$E(P;X,Y)=E(0;P;X,Y)$ in 
\eqref{eq:ENPXY}. We shall adhere to  common convention and allow the value of $\ve>0$ to change at each appearance, so that $P^\ve\log P\ll P^\ve$, for example.  Moreover, all implied constants are allowed to depend on $\ve$.

Applying Corollary 
\ref{cor:2}, we deduce that 
$$
E(P;X,Y)\ll P^\ve 
X^{1-(n-m)/2}
Y^{-m/2}
\sum_{\substack{\fb\in \mathcal{B}(X,Y)}}
\sum_{\substack{\0\neq \m\in \widehat{\gfb  }^n\\
i>m \Rightarrow m_i\in \mathfrak{d}^{-1}\mathfrak{b}^{-1}
}}
\int_{\mathcal{U}}
|K(u,P\m)|\d u.
$$
Let $\delta\in \gfb\fd$ and let  $\fp_1$ 
be a prime
ideal coprime to $\gfb\fd$,  with $\n\fp_1 \ll (\n \fb)^{\ve/d}$, such that
$(\delta)=\gfb\fd\fp_1$. On multiplying $\delta$ by an appropriate unit,  
there is no loss of generality in assuming that 
\begin{equation}\label{eq:peake}
Y^{1/d}
\ll |\delta^{(l)} |\ll
 Y^{1/d+\ve},
\end{equation}
for $1\leq l\leq d$, 
since $Y\leq \n \gfb <2Y$.
We are led to make the change of variables 
\begin{equation}\label{eq:change-var}
c_i=
\delta m_i ,
\end{equation}
for $1\leq i\leq n$, so that $\c=(c_1,\dots,c_n)\in \fo^n$.
Then 
\begin{equation}\label{eq:squall}
c_i\in \delta
\fd^{-1}\fb^{-1}= \fp_1\fb^{-1}\gfb\subset \fb^{-1}\gfb, 
\quad \text{for  $m<i\leq n$}.
\end{equation}
We may now write
$$
E(P;X,Y)\ll P^\ve 
X^{1-(n-m)/2}
Y^{-m/2}
\sum_{\substack{\fb\in \mathcal{B}(X,Y)}}
\sum_{\substack{\0\neq \c\in \fo^n\\
\text{\eqref{eq:squall} holds}}}
\int_{\mathcal{U}}
|K(u,P\delta^{-1}\c)|\d u.
$$

Define the function
\begin{equation}\label{eq:f(u)}
f(u)=\prod_{1\leq l\leq d} \prod_{1\leq i\leq n} \min\left( 1, \frac{1}{|\eth_{l,i}(u)|^{1/2}}\right).
\end{equation}
Let $\mathcal{R}(\m)$ be the set of $u\in \mathcal{U}$ such that
$$|(M_l(u)\m^{(l)})_i|\ll P^{-1+\ve} |\eth_{l,i}(u)|,\quad 1\leq i\leq n, 1\leq l\leq d,$$
and
$$|m_i^{(l)}|\ll P^{-1+\ve} |u_l|,\quad m<i\leq n,\, 1\leq l\leq d.$$
We now have 
$$
E(P;X,Y)\ll P^\ve 
X^{1-(n-m)/2}
Y^{-m/2}
\sum_{\substack{\fb\in \mathcal{B}(X,Y)}}
\sum_{\substack{\0\neq \c\in \fo^n\\
\text{\eqref{eq:squall} holds}}}
\int_{\mathcal{R}(\delta^{-1}\c)}
f(u)\d u.
$$
Let 
$$
L(u)=
\sum_{\fb\in \mathcal{B}(X,Y)} 
\sum_{
\c\in \mathcal{C}(u,\fb)}
1,
$$
where 
$\mathcal{C}(u,\fb)$ is the set of non-zero vectors $\c\in \fo^n$ for which \eqref{eq:squall} holds,
$$
|(M_l(u)\c^{(l)})_i| \ll P^{-1+\ve}Y^{1/d} |\eth_{l,i}(u)|,\quad 1\leq i\leq n, \,1\leq l\leq d,$$
and
$$|c_i^{(l)}|\ll P^{-1+\ve} Y^{1/d} |u_l|,\quad m<i\leq n, \, 1\leq l\leq d.$$
Then  we have
$$
E(P;X,Y)\ll P^\ve 
X^{1-(n-m)/2}
Y^{-m/2}\int_{\mathcal{U}} f(u)L(u) \d u.$$

Our next goal is to estimate $L(u)$. For each  $1\leq l\leq d$ we sort the eigenvalues $\eth_{l,i}(u)$ in a way such that
$$
|\eth_{l,1}(u)|\geq |\eth_{l,2}(u)|\geq \cdots \geq |\eth_{l,n}(u)|.
$$
Note that we can always achieve this by adjusting the orthogonal matrix $M_l(u)$ with suitable permutations.
Moreover, for all $1\leq i\leq n$ and $1\leq l\leq d$, we have
\begin{equation}\label{upperboundEV1}
|\eth_{l,i}(u)|\ll |u_l|+|u_{l_\tau}|.
\end{equation}
It will now be useful to  make the observation 
\begin{equation}\label{eq:carpet}
 \prod_{l=1}^d (1+|u_l| +|u_{l_\tau}|) \ll \prod_{l=1}^d ((1+|u_l|)(1+|u_{l_\tau}|)) \ll \mathfrak{H}(u)^2.
\end{equation}
We proceed by proving the following result.

\begin{lemma}\label{lem:wish}
Let $u\in V$ such that $\mathfrak{H}(u)\leq P^{d+\ve}/X$.
If $L(u)\neq 0$ then 
\begin{equation}\label{eq:thisa}
P^{-d+\ve} Y 
\mathfrak{H}(u)^2\gg 1.
\end{equation}
Moreover, we have 
$
L(u)\ll P^\ve X J(u),
$
where
$$
 J(u)=
\prod_{1\leq l\leq d}  \prod_{1\leq i\leq m} (1+ P^{-1+\ve}Y^{1/d} |\eth_{l,i}(u)|).
$$
\end{lemma}

\begin{proof}
Let us write $\c=(\c',\c'')$, where $\c'=(c_1,\dots,c_m)$ and $\c''=(c_{m+1},\dots,c_n)$. 
Keeping in mind  \eqref{eq:squall}, we first fix  a choice of $\c''\in \left(\fb^{-1} \gfb \right)^{n-m}$ satisfying
$$
|c_i^{(l)}| \ll 
P^{-1+\ve}  Y^{1/d} |u_l| ,
$$
for $m+1\leq i\leq n$ and $1\leq l\leq d$.
Choose $\lambda\in K$ such that 
$(\lambda)=\fb^{-1} \gfb  \fp_2^{-1}$, 
for a suitable prime ideal $\fp_2$ of norm $O(P^\ve).$ 
We may assume  that $\lambda$ is well-shaped, in the sense of \eqref{eq:peake}, 
 on multiplying by a suitable unit. 
 Thus 
 $
 X^{1/d}Y^{1/d}\ll |\lambda^{(l)}| \ll X^{1/d}Y^{1/d+\ve},
 $
 for $1\leq l\leq d$.
Making the change of variables $\c''=\lambda \mathbf{d}''$
and 
recalling that  $\n\fb \asymp X$ and
 $\n\gfb \asymp Y$, we must have  
$$
|d_i^{(l)}| \ll 
P^{-1+\ve}  X^{1/d} |u_l| ,
$$
for $m+1\leq i\leq n$ and $1\leq l\leq d$.

We begin by showing that \eqref{eq:thisa} holds if $L(u)\neq 0$. Thus there exists 
$\c\neq \0$ counted by $L(u)$.
Suppose first that  $\c''\neq \0$. Then there exists $i\in \{m+1,\dots,n\}$ such that 
$$
1\leq |N_{K/\QQ}(d_i)| \ll 
P^{-d+\ve}  X |\nm (u)| ,
$$
whence
$
1\ll P^{-d+\ve} X \mathfrak{H}(u)\ll P^{-d+\ve} Y \mathfrak{H}(u)^2, 
$
since $X\leq Y$. This is satisfactory for \eqref{eq:thisa}.
Suppose next that   $\c'\neq \0$. 
  In particular we have
 \begin{equation}\label{eq:deer}
|(M_l(u)\c^{(l)})_i| \ll P^{-1+\ve}Y^{1/d} |\eth_{l,i}(u)|,\quad 1\leq i\leq n, \,1\leq l\leq d.
\end{equation}
As $M_l(u)$ is an orthogonal matrix, this implies that 
$$|c_j^{(l)}|\ll P^{-1+\ve}Y^{1/d} \max_{1\leq i\leq n} |\eth_{l,i}(u)|,\quad 1\leq j\leq n,\, 1\leq l\leq d,
$$
whence \eqref{eq:carpet} yields
$$
1\ll P^{-d+\ve} Y \prod_{l=1}^d \max_{1\leq i\leq n} |\eth_{l,i}(u)|\ll 
P^{-d+\ve} Y \prod_{l=1}^d (|u_l|+|u_{l_\tau}|) \ll 
P^{-d+\ve} Y \mathfrak{H}(u)^2.
$$
This completes the proof of \eqref{eq:thisa} under the assumption that $L(u)\neq 0$.

Turning now to the estimation of $L(u)$, 
it  readily  follows from a result in Lang
 \cite[Thm.~0 in \S V.1]{lang} that  
 the overall  number of  vectors $\mathbf{d}''$ is 
 \begin{align*}
&\ll  \left(1+  \prod_{l=1}^d
P^{-1+\ve} X^{1/d} |u_l|\right)^{n-m} 
\ll \left(1+ P^{-d+\ve} 
X  \nm(u)\right)^{n-m}\ll P^\ve.
 \end{align*}
 It remains to count the number of vectors $\c'$ associated to a particular choice of $\c''$.
Let $L(u,\fb,\c'')$ be the number of $\c'\in \fo^m$ such that \eqref{eq:deer} holds.
Assume that the matrix $M_l(u)$ is given by $M_l(u)=(m_{l\alpha\beta})_{1\leq \alpha,\beta\leq n}$. Write
$$
M_l(u)=(M_l'(u) M_l''(u)),
$$
with $M_l'(u)=(m_{l\alpha\beta})_{\substack{1\leq \alpha\leq n\\ 1\leq \beta\leq m}}$ and $M_l''(u)=(m_{l\alpha\beta})_{\substack{1\leq \alpha\leq n\\ m< \beta\leq n}}$.
Then we consider the system of inequalities
$$
|M_l'(u)\c'^{(l)}+r_{li}|
\ll P^{-1+\ve}Y^{1/d} |\eth_{l,i}(u)|,\quad 1\leq i\leq n, \,1\leq l\leq d,
$$
where $\r_l=(r_{li})_{1\leq i\leq n}=M_l''(u) \c''^{(l)}$.

Write 
$$c_i=\sum_{l=1}^d c_{il}\ome_l,\quad c_{il}\in \mathbb{Z},\, 1\leq i\leq m.$$
Then, for $1\leq i\leq n$, we can write
$$
(M'_l(u)\c'^{(l)})_i = \sum_{\beta=1}^m m_{li\beta} c_\beta^{(l)} = \sum_{\beta=1}^m m_{li\beta}\sum_{k=1}^d c_{\beta k} \ome_k^{(l)} = \sum_{\beta=1}^m \sum_{k=1}^d m_{li\beta} \ome_k^{(l)} c_{\beta k}.
$$
Let $H$ be the $dn\times dm$ matrix given by
$$H=(m_{li\beta}\ome_k^{(l)})_{(l,i)\times (k,\beta)},$$
with $1\leq l\leq d$, $1\leq i\leq n$, $1\leq k\leq d$, $1\leq \beta\leq m$,  and consider the lattice
$$\Lambda = H \mathbb{Z}^{md}\subset \mathbb{R}^{nd}.$$
Then $L(u,\fb,\c'')$ counts lattice points in $\Lambda$ which lie in a box of side length 
$$\ll P^{-1+\ve}Y^{1/d} |\eth_{l,i}(u)|,\quad 1\leq i\leq n, \,1\leq l\leq d.$$
We claim that the successive minima of the lattice $\Lambda$ are bounded above and below by constants depending only on $K$ and $n$. Taking this on faith, it will then follow that 
$$
L(u,\fb,\c'')\ll \prod_{1\leq l\leq d}\prod_{1\leq i\leq m}   (1+ P^{-1+\ve}Y^{1/d} |\eth_{l,i}(u)|),
$$
which will settle the lemma, on summing over $O(X)$ choices for $\fb\in \mathcal{B}(X,Y)$.

To check the claim, we order the index tuples $(l,i)$ and $(k,\beta)$ in the matrix $H$ lexicographically. Write 
$$A_{lk}=(m_{li\beta}\ome_k^{(l)})_{1\leq i\leq n,1\leq \beta\leq m}= \ome_k^{(l)}(m_{li\beta})_{(i,\beta)}= \ome_k^{(l)}B_l,$$
with the $n\times m$ matrix $B_l=(m_{li\beta})_{1\leq i\leq n,1\leq \beta\leq m}$. Note that $B_l$ has orthogonal and norm one columns for $1\leq l\leq d$. We can then write $H$ as a block matrix
$$H=(A_{lk})_{1\leq l,k\leq d}.$$
Let $B=B(u)$ be the $nd\times md$ matrix which is a diagonal block matrix,  
with the matrices $B_1,\ldots, B_d$ on the diagonal. Let  $W$ be the $md\times md$ block matrix, 
with blocks
$\ome_k^{(l)} E_{m}$ at each  place $1\leq l,k\leq d$, where $E_m$ is the $m$-dimensional identity matrix.
Then 
$$H=BW.$$
Consider the lattice $\Gamma = W\mathbb{Z}^{md}\subset \mathbb{R}^{md}$ and note that this only depends on the basis $\ome_1,\ldots, \ome_d$. Moreover, if $\w\in \Gamma$, then
$$\langle B\w, B\w\rangle = \w^tB^tB\w = \w^t E_{md} \w = \langle \w,\w\rangle.$$
Hence the successive minima of the lattice $\Lambda$ coincide with those of $\Gamma$, which thereby establishes the claim.
\end{proof}

It follows from the previous result that 
\begin{equation}\label{eq:E-so-far}
\begin{split}
E(P;X,Y)\ll &P^\ve 
X^{2-(n-m)/2}
Y^{-m/2}\\
&\times \int_{\mathcal{U}^*}
f(u)\prod_{1\leq i\leq m}\prod_{1\leq l\leq d}   (1+ P^{-1+\ve}Y^{1/d} |\eth_{l,i}(u)|) \d u,
\end{split}
\end{equation}
where $f(u)$ is given by \eqref{eq:f(u)}
and $\mathcal{U}^*$ is the set of $u\in \mathcal{U}$ such that \eqref{eq:thisa} holds.

Recall that the $\eth_{l,i}(u)$ are the eigenvalues of the matrix associated to the quadratic form
$$u_lQ^{(l)}(\x^{(l)})+u_{l_\tau}R^{(l_\tau)}(\x^{(l)}).$$
The next result   collects together a number of properties concerning the  size of the eigenvalues 
$\eth_{l,i}(u)$. 

\begin{lemma}\label{size_EVs}
Assume that Assumptions \ref{ass1}--\ref{ass3} hold and 
suppose that $\tilde m\geq m-1$ is the degree of the polynomial appearing in Assumption \ref{ass3}.
For each $1\leq l\leq d$  we order the eigenvalues $\eth_{l,i}(u)$ such that
$$|\eth_{l,1}(u)|\geq |\eth_{l,2}(u)|\geq \cdots \geq |\eth_{l,n}(u)|.
$$
Then there exist constants $C_1,\dots,C_d>0$ such that the following holds:
\begin{enumerate}
\item
If $|u_{l_\tau}|\leq C_l |u_{l}|$, then 
$$|\eth_{l,1}(u)|\ll |u_l|\quad \text{ and } \quad |\eth_{l,n-1}(u)|\gg |u_l|.$$
Moreover, if $m=1$ and $\tilde m=0$, then 
$$|\eth_{l,1}(u)|\ll |u_l|\quad \text{ and } \quad |\eth_{l,n}(u)|\gg |u_l|.$$
\item
If $|u_{l_\tau}|> C_l |u_l|$, then 
$$|\eth_{l,m+1}(u)\cdots \eth_{l,n}(u)|\gg \frac{|u_l|^{n-\tilde m}}{|u_{l_\tau}|^{m-\tilde m}}. $$
\end{enumerate}
\end{lemma}

\begin{proof}
To begin with, according to Assumption \ref{ass3}, for  each $1\leq l\leq d$ there exists a constant $C_l$ such that
$$
|\det (Q^{(l)}+t R^{(l_\tau)})|\gg |t|^{\tilde m},
$$
for $|t|\geq C_l$.

We start by examining the case $|u_{l_\tau}|\leq C_l |u_l|$. The first bound $|\eth_{l,1}(u)|\ll |u_l|$ follows directly from  \eqref{upperboundEV1}. Assume now that $u_l\neq 0$. Note that each of the eigenvalues $\eth_{l,i}(u)$ arises by multiplication with $u_l$ from the eigenvalues of the matrix corresponding to
$$Q^{(l)} + \frac{u_{l_\tau}}{u_l} R^{(l_\tau)}.$$
Write $\tilde\eth_{l,i}(u)$ for those eigenvalues in the same ordering. Assume that the lower bound $|\tilde\eth_{l,n-1}(u)|\gg 1$ is not satisfied. Thus there exists a sequence of $t_j$ in the range $|t_j|\leq C_l$ such that $\tilde\eth_{l,n-1}(t_j) \rightarrow 0 $, for $j\rightarrow \infty$, where we write $\tilde \eth_{l,n-1}(t)$ for the second smallest eigenvalues of $Q^{(l)}+tR^{(l_\tau)}$. As the set of $t$ is compact there is a convergent subsequence, convergent to $t'$ say, with $\rank (Q^{(l)}+t' R^{(l_\tau)})<n-1$. This  contradicts  Assumption \ref{ass2}.

Now we consider the case $m=1$ and $\tilde m=0$. By Assumptions \ref{ass1} and \ref{ass3} we deduce that $\det (Q^{(l)} + tR^{(l_\tau)})$ is a non-zero constant independent of $t$. In particular the rank of this matrix is always  $n$ and the argument above shows that $|\eth_{l,n}(u)|\gg |u_l|$.

Next we consider the case $|u_{l_\tau}|> C_l |u_l|$ and $u_l\neq 0$. 
Again, we write $\tilde \eth_{l,i}(u)$ for the eigenvalues of $Q^{(l)} + \frac{u_{l_\tau}}{u_l} R^{(l_\tau)}.$ Note that we have
$$|\tilde \eth_{l,i}(u)|\ll |\frac{u_{l_\tau}}{u_l}|,\quad 1\leq i\leq n.$$
Moreover, we observe that
$$|\tilde\eth_{l,1}(u)\cdots \tilde\eth_{l,n}(u)| =  | \det( Q^{(l)} + \frac{u_{l_\tau}}{u_l} R^{(l_\tau)})|\gg \left| \frac{u_{l_\tau}}{u_l}\right|^{\tilde m}.$$
We therefore find that
\begin{align*}
|\tilde\eth_{l,m+1}(u)\cdots \tilde\eth_{l,n}(u)|
 &\gg \left| \frac{u_l}{u_{l_\tau}}\right|^m\left| \frac{u_{l_\tau}}{u_l}\right|^{\tilde m}
 = \left| \frac{u_l}{u_{l_\tau}}\right|^{m-\tilde m} 
\end{align*}
From this we obtain the lower bound
$$|\eth_{l,m+1}(u)\cdots \eth_{l,n}(u)|\gg |u_l|^{n-m}\left| \frac{u_l}{u_{l_\tau}}\right|^{m-\tilde m} = \frac{|u_l|^{n-\tilde m}}{|u_{l_\tau}|^{m-\tilde m}},
$$
which completes the proof of the lemma.
\end{proof}

We now continue with our analysis of  $E(P;X,Y)$ 
in \eqref{eq:E-so-far}. 
Let $E_1(P;X,Y)$ denote the overall contribution from the case $Y\geq P^d$, and let 
$E_2(P;X,Y)$ denote the remaining contribution. The following pair of results treats these two quantities in turn.

\begin{lemma}\label{lem:E1}
Assume $\tilde m\geq m-1$. Then  $E_1(P;X,Y)\ll P^{d(2-n/2+m/2)+\ve}+P^{-dm+\ve}$. 
\end{lemma}

\begin{proof}
On recalling the definition \eqref{eq:f(u)} of
 $f(u)$, we deduce from 
 \eqref{eq:E-so-far} that 
 \begin{align*}
E_1(P;X,&Y)
\ll P^{\ve} X^{2-(n-m)/2}Y^{-m/2}P^{-md}Y^m \\ & \times\int_{\substack{\mathcal{U}^*}} \left(\prod_{1\leq l\leq d} (1+|u_l|+|u_{l_\tau}|)^{m/2}\right)\prod_{1\leq l\leq d} \prod_{m<i\leq n}  \min\left( 1, \frac{1}{|\eth_{l,i}(u)|^{1/2}}\right) \d u,
\end{align*}
since \eqref{upperboundEV1} implies that 
$$
\prod_{1\leq i\leq m} (1+ |\eth_{l,i}(u)|)^{1/2}\ll  (1+|u_l|+|u_{l_\tau}|)^{m/2}.
$$
Here, we recall that $\mathcal{U}^*$ is the set of $u\in \mathcal{U}$ such that \eqref{eq:thisa} holds.
Consider for a moment a fixed value of $l$. If $C_l|u_l| \geq |u_{l_\tau}|$, then  Lemma \ref{size_EVs} yields
$$
\prod_{m<i\leq n}  \min\left( 1, \frac{1}{|\eth_{l,i}(u)|^{1/2}}\right) \ll \min\left(1,|u_l|^{-(n-m-1)/2}\right).
$$
If $C_l|u_l| < |u_{l_\tau}|$ and $\tilde m\geq m-1$, then Lemma \ref{size_EVs} yields
\begin{align*}
\prod_{m<i\leq n}  \min\left( 1, \frac{1}{|\eth_{l,i}(u)|^{1/2}}\right) 
&\ll   \min\left( 1, \frac{1}{|\eth_{l,m+1}(u)\cdots \eth_{l,n}(u)|^{1/2}}\right)\\
 &\ll \min\left( 1, \frac{|u_{l_\tau}|^{1/2}}{|u_l|^{(n-m+1)/2}} \right).
 \end{align*}
In either of these two cases we therefore have
$$\prod_{m<i\leq n}  \min\left( 1, \frac{1}{|\eth_{l,i}(u)|^{1/2}}\right) \ll (1+|u_l|+|u_{l_\tau}|)^{1/2} \min\left(1,|u_l|^{-(n-m)/2}\right),$$
whence
\begin{align*} E_1(P;X,Y)\ll & P^{\ve} X^{2-(n-m)/2}Y^{-m/2}P^{-md}Y^m \\ &\times  \int_{\mathcal{U}} 
\left(\prod_{1\leq l\leq d} (1+|u_l|+|u_{l_\tau}|)^{(m+1)/2}\right)\prod_{1\leq l\leq d}   \min\left( 1, \frac{1}{|u_l|^{(n-m)/2}} \right) \d u.
\end{align*}

It follows from \eqref{eq:carpet} that 
\begin{align*}
E_1(P;X,Y) &\ll P^{\ve} X^{2-(n-m)/2}Y^{-m/2}  P^{-md}Y^m \int_{\mathfrak{H}(u)\leq P^{d+\ve}/X} \mathfrak{H}(u)^{m+1-(n-m)/2} \d u\\
&\ll P^{\ve} X^{2-(n-m)/2}Y^{-m/2} P^{-dm} Y^{m} (1+ (P^d/X)^{2+m-(n-m)/2})\\
&\ll P^{\ve} X^{-m} Y^{m/2} P^{d(2-n/2+m/2)} + P^{\ve} X^{2-(n-m)/2}Y^{m/2} P^{-dm}.
\end{align*}
The contribution gets maximal for $Y\asymp X^2$, in which  case we get the upper bound
\begin{align*}
E_1(P;X,Y)
&\ll P^{\ve} X^{-m}  X^{m} P^{d(2-n/2+m/2)}+ P^{\ve} X^{2-n/2 + 3m/2} P^{-dm} \\
&\ll P^{d(2-n/2+m/2)+\ve} +  X^{2-n/2 + 3m/2} P^{-dm+\ve}.
\end{align*}
The first term is satisfactory for the lemma. 
If  $2-n/2+3m/2\leq 0$, then  the second term is 
$O( P^{-dm+\ve})$, which is satisfactory.  
If $n\leq 3+3m$, on the other hand, then we take $X\ll P^d$ and get the satisfactory upper bound 
$O(P^{d(2-n/2+m/2)+\ve})$.
\end{proof}

\begin{lemma}\label{lem:E2}
Assume that $n\geq m+4$
and $\tilde m \geq m-1$. Let 
$$
\kappa =
\begin{cases}
1 & \text{ if $m=1$ and $\tilde m=0$,}\\ 
0 & \text{ otherwise.}
\end{cases}
$$ 
Then $E_2(P;X,Y)$ is 
$$
\ll 
P^{-d(n-m-4+\kappa)/4+\ve}+
 P^{-1/2+\ve }+
 P^{-2m + d(3m + 4-\kappa - n)/2+\ve}+
  P^{-2m + d(3m + 4-\kappa - n)/4+\ve}.
$$ 
\end{lemma}

\begin{proof}
For $1\leq i\leq m$, 
we clearly  have
$$ \min\left( 1, \frac{1}{|\eth_{l,i}(u)|^{1/2}}\right) (1+ P^{-1+\ve}Y^{1/d} |\eth_{l,i}(u)|) \ll 1+P^{-1+\ve}Y^{1/d} |\eth_{l,i}(u)|^{1/2}.$$
Hence we find that $E_2(P;X,Y)$ is
\begin{align*}
&\ll P^\ve 
X^{2-(n-m)/2}
Y^{-m/2}\\  &\times \int_{\mathcal{U}^*}\left(\prod_{1\leq l\leq d} \prod_{m< i\leq n} \min\left( 1, \frac{1}{|\eth_{l,i}(u)|^{1/2}}\right) \right)\prod_{1\leq l\leq d} \prod_{1\leq i\leq m}  (1+ P^{-1+\ve}Y^{1/d} |\eth_{l,i}(u)|^{1/2}) \d u
\end{align*}
If $C_l|u_l|\geq |u_{l_\tau}|$, then Lemma \ref{size_EVs} leads to the bound
$$\prod_{m<i\leq n} \min \left( 1, \frac{1}{|\eth_{l,i}(u)|^{1/2}}\right) \ll \min \left(1, |u_l|
^{-(n-m-1+\kappa)/2} 
\right).$$
For the case $C_l |u_l| < |u_{l_\tau}|$ we still have
\begin{align*}
\prod_{m<i\leq n}  \min\left( 1, \frac{1}{|\eth_{l,i}(u)|^{1/2}}\right) 
&\ll \min\left( 1, \frac{|u_{l_\tau}|^{1/2}}{|u_l|^{(n-m+1)/2}} \right),
\end{align*}
since $\tilde m\geq m-1$.
We now deduce that in either case we have
$$
\prod_{m<i\leq n}  \min\left( 1, \frac{1}{|\eth_{l,i}(u)|^{1/2}}\right) \ll (1+|u_l|+|u_{l_\tau}|)^{1/2} \min\left(1,|u_l|^{-(n-m+\kap)/2}\right).
$$
It now follows from \eqref{eq:carpet} that $E_2(P;X,Y)$ is 
\begin{align*}
\ll~& P^\ve 
X^{2-(n-m)/2}
Y^{-m/2}\\  & \times \int_{\mathcal{U}^*} \mathfrak{H}(u)^{1-(n-m+\kappa)/2} 
\prod_{1\leq l\leq d} \prod_{1\leq i\leq m}  (1+ P^{-1+\ve}Y^{1/d} (|u_l|+|u_{l_\tau}|)^{1/2}) \d u\\
\ll~& P^\ve 
X^{2-(n-m)/2}
Y^{-m/2}\\  &\times \int_{\mathcal{U}^*}\mathfrak{H}(u)^{-\frac{n-m-2+\kappa}{2}}
\prod_{1\leq l\leq d}   (1+ P^{-1+\ve}Y^{1/d} (|u_l|+|u_{l_\tau}|)^{1/2})^m \d u.
\end{align*}
Let $I_1$ denote the contribution to the integral from those $u$ for which 
there exists at least one  $u_l$ with $|u_l|\gg (P/Y^{1/d})^{2} $, and let 
$I_2$ denote the remaining contribution.

On recalling that $\mathcal{U}^*$ is the set of $u\in \mathcal{U}$ such that \eqref{eq:thisa} holds, 
it is clear that 
 $$
 I_2\ll
  \int_{\substack{P^{d/2-\ve}/Y^{1/2} \ll \mathfrak{H}(u)\ll P^{d +\ve}/X\\ |u_l|\ll P^2Y^{-2/d}, ~1\leq l\leq d}}\mathfrak{H}(u)^{-\frac{n-m-2+\kappa}{2}} \d u.
$$
Turning to $I_1$, we see that 
\begin{align*}
\prod_{1\leq l\leq d}   (1+ P^{-1+\ve}&Y^{1/d} (|u_l|+|u_{l_\tau}|)^{1/2})^m\\
&\ll P^\ve
\prod_{\substack{1\leq l\leq d\\ |u_l|+|u_{l_\tau}|\geq (P/Y^{1/d})^{2}}}  
( P^{-1}Y^{1/d} (|u_l|+|u_{l_\tau}|)^{1/2})^m\\
&\ll P^\ve \mathfrak{H}(u)^m(P^{-1}Y^{1/d})^{m\sharp\{1\leq l\leq d:  |u_l|+|u_{l_\tau}|\geq (P/Y^{1/d})^{2} \}},
\end{align*}
by \eqref{eq:carpet}.
But if there is one $u_l$ with $|u_l|\gg (P/Y^{1/d})^{2} $, then clearly 
$$\sharp\{1\leq l\leq d:  |u_l|+|u_{l_\tau}|\geq (P/Y^{1/d})^{2} \}\geq 2.$$
Hence, since $P^{-1}Y^{1/d}\ll 1$, it now follows that 
$$
I_1
\ll P^\ve
 \int_{P^{d/2-\ve}/Y^{1/2} \ll \mathfrak{H}(u)\ll P^{d +\ve}/X}\mathfrak{H}(u)^{-\frac{n-m-2+\kappa}{2}+m} (P^{-1}Y^{1/d})^{2m} \d u.
 $$
In summary we have shown that 
$$
E_2(P;X,Y)
 \ll P^\ve 
X^{2-(n-m)/2}
Y^{-m/2} 
\left(I_1+I_2\right),
$$
with $I_1,I_2$ as above.

Since $n\geq m+4$,  the exponent of $\mathfrak{H}(u)$ in $I_2$ is less than or equal to $-1$. 
If 
$\frac{n-m-2+\kappa}{2}>1$,  then it follows from  \eqref{eq:BoundA} that 
$$
I_2\ll P^\ve 
 (P^{d/2}/Y^{1/2})^{-\frac{n-m-2+\kappa}{2}+1}.
$$ 
However, if   $\frac{n-m-2+\kappa}{2}=1$, then we apply 
\eqref{eq:BoundB}  to deduce that the same bound holds. 

On the other hand, 
 \eqref{eq:BoundA} and \eqref{eq:BoundB} also yield
\begin{align*}
I_1
&\ll  
   P^{-2m+\ve}Y^{2m/d}\left((P^d/X)^{-\frac{n-m-2+\kappa}{2}+m+1}   
   +(P^{d/2}/Y^{1/2})^{-\frac{n-m-2+\kappa}{2}+m+1} \right)\\
&\ll X^{\frac{n-m-2+\kappa}{2}-m-1}
Y^{2m/d} P^{-2m-d(\frac{n-m-2+\kappa}{2})+md+d+\ve}    \\
&\quad +
Y^{2m/d+n/4-3m/4+\kappa/4-1} P^{-2m+d/2( -\frac{n-m-2+\kappa}{2}+m+1)+\ve}.
\end{align*}
We conclude that 
\begin{align*}
E_2(P;X,Y)
& \ll 
X^{\kappa/2  -m }
Y^{-m/2+2m/d} P^{-2m + 3md/2 + (4-\kappa)d/2 - dn/2+\ve}   \\
&+ 
X^{2-(n-m)/2}
Y^{ 2m/d+n/4-5m/4-(4-\kappa)/4}    P^{-2m - dn/4+3md/4+ (4-\kappa)d/4+\ve}   \\
& + 
X^{2-(n-m)/2}
Y^{-3m/4+n/4-(4-\kappa)/4}  P^{d/2(-\frac{n-m-2+\kappa}{2}+1)+\ve}.
\end{align*}

We now consider these three terms separately, starting with the third and 
recalling that $n-m\geq 4$.
If 
$-3m/4+n/4-(4-\kappa)/4\leq 0,$
then we get an upper bound
$$
\ll P^{d/2(-\frac{n-m-2+\kappa}{2}+1)+\ve}\ll 
P^{-d(n-m-4+\kappa)/4+\ve}.
$$
In the opposite case we get the upper bound
$
\ll P^{-dm/2+\ve }\ll P^{-m+\ve},
$
on using  $Y\leq P^d$ and $d\geq 2$.

We now turn to the second term. If
$2m/d+n/4-5m/4-(4-\kappa)/4\leq 0,$
then we get the upper bound 
$$
\ll P^{-2m - dn/4+3md/4+ (4-\kappa)d/4+\ve}.
$$
In the opposite case, on using  $X\geq Y^{1/2}$,  we get the  upper bound 
\begin{align*}
&\ll Y^{1-(n-m)/4}
Y^{ 2m/d+n/4-5m/4-(4-\kappa)/4}    P^{-2m - dn/4+3md/4+ (4-\kappa)d/4+\ve}\\
&\ll  Y^{\kappa/4-m+2m/d}    P^{-2m - dn/4+3md/4+ (4-\kappa)d/4+\ve}.
\end{align*}
If $d\geq 3$ or $\kappa =0$, then we reduce to the case above. If $d=2$ and $\kappa =1$, on the other hand,  we obtain the upper bound
$$\ll P^{ d/4} P^{-2m - dn/4+3md/4+ (4-\kappa)d/4}
\ll P^{-(n+m-4)/2+\ve},
$$
since $Y\leq P^d$. Clearly $(n+m-4)/2\geq m$
if $n\geq m+4$, whence this case contributes $O(P^{-m+\ve})$, which is satisfactory. 

It remains to deal with the first term. Again we use the lower bound $X\geq Y^{1/2}$, allowing us to  bound the first term by
\begin{align*}
&\ll 
Y^{\kappa/4  -m/2 }
Y^{-m/2+2m/d} P^{-2m + 3md/2 + (4-\kappa)d/2 - dn/2+\ve} \\
&\ll Y^{\kappa/4-m+2m/d} P^{-2m + 3md/2 + (4-\kappa)d/2 - dn/2+\ve}.
\end{align*}
If $d\geq 3$ or $\kappa =0$ then 
we get $O(P^{-2m + 3md/2 + (4-\kappa)d/2 - dn/2+\ve})$, which is satisfactory. 
Alternatively, if $d=2$ and $\kappa=1$, then we get 
$$\ll P^{ 1/2} P^{-2m + 3md/2 + (4-\kappa)d/2 - dn/2+\ve}
\ll P^{-(n-m-7/2)+\ve}.
$$
This is $\ll P^{-1/2+\ve}$ since $n\geq m+4$, which thereby completes the proof of the lemma. 
\end{proof}

It remains to combine Lemmas \ref{lem:E1} and \ref{lem:E2}.
We make the assumption 
$$
n\geq m+5.
$$
Under this assumption,  the bound in  Lemma 
\ref{lem:E1} is $O(P^{-d/2+\ve})$.  
Moreover, the bound in Lemma \ref{lem:E2} is 
$$
\ll P^{-d(1+\kappa)/4+\ve}+
 P^{-1/2+\ve } +
  P^{-2m + d(3m + 4-\kappa - n)/2+\ve}+
   P^{-2m + d(3m + 4-\kappa - n)/4+\ve}.
$$ 
Hence, since $d\geq 2$ and $m\geq 1$,  it finally follows that 
\eqref{eq:end} holds for a suitable $\Delta>0$, provided that  $n\geq m+5$ and 
$$
n> 3m+4-\frac{4m}{d}-\kappa,
$$
where  $\kappa$ is defined in the statement of Lemma \ref{lem:E2}.

Suppose first that $m=1$ and place ourselves under the  hypotheses of Theorem~\ref{t:non-diag-1}. Then 
Assumptions \ref{ass1}--\ref{ass3} hold with $\tilde m=0$. Hence $\kappa=1$
and the condition on $n$ reduces to  $n\geq 6$, 
as required for Theorem \ref{t:non-diag-1}.
Assume now that $m\geq 1$, but $\kappa=0$.
Since  $d\geq 2$, we have 
$3m+4-4m/d\geq m+4$,
from which the statement of Theorem \ref{t:non-diag} follows.

\section{Inhomogeneous case: proof of Theorem \ref{t:diag}}\label{s:inhom}

In this section we complete the proof of Theorem \ref{t:diag}.
We note that the quadratic form in \eqref{eq:special} is a special case of \eqref{eq:special-0}, with 
$\mathbf{A}=\diag(a_1,\dots,a_n)$ and 
$\mathbf{B}=\diag(b_1,\dots,b_m,0,\dots,0)$. Hence 
Corollary \ref{cor:2} applies to the situation considered in 
Theorem \ref{t:diag}. In particular, assuming that $n>4$, 
the argument in the previous section shows that 
the sum over $\fb$ 
in Lemma \ref{lem:main_term} 
can be extended to infinity  with acceptable error. Since the assumption  $n>4$  is implied by the hypotheses in 
 Theorem \ref{t:diag}, this leaves us free to focus our efforts on  proving \eqref{eq:end}.

In the present setting, 
it will be vital to obtain additional cancellation from the sum over 
primitive characters  in $S_\fb(N;\m)$. 
We plan to improve on Corollary \ref{cor:2} in generic situations, beginning with an examination of 
a particular exponential sum modulo degree 1 prime ideals.
The saving we shall achieve is linked to the fact that $N\neq 0$ and will also involve the special generalised  quadratic form
\begin{equation}\label{eq:def-G}
G(\x)=a_1\cdots a_n b_1\cdots b_m  \left(\frac{x_1^2}{a_1}+\cdots+\frac{x_n^2}{a_n}
+
\frac{(x_1^\tau)^2}{b_1}+\cdots+\frac{(x_m^\tau)^2}{b_m}
\right),
\end{equation}
that is the analogue of the {\em dual form} in our setting. (Note that it has coefficients in $\fo$.)
For any unramified prime ideal $\fp$ and 
any vector $\v\in \widehat{\gfp}^n$, it will be convenient to observe  that 
$\ord_\fp(G(\v))\geq -2$, 
since $\ord_\fp (v_i)\geq -1$ and $\ord_{\fp} (v_i^\tau)\geq -1$, for any $v_i\in  \widehat{\gfp}$.
With this in mind, 
we proceed by proving the following bound for $S_\fp(N;\v)$.

\begin{lemma}\label{lem:prime}
Let 
$\fp$ be a prime ideal of residue degree $1$ and let  $\v\in \widehat{\gfp}^n$. Then 
$$
S_\mathfrak{p}  (N;\v)\ll 
(\n\fp)^{\theta_{\fp}(\v)+(3n-m)/2},  
$$
where
$$
\theta_\fp(\v)=
\begin{cases}
1 & \text{ if $\fp\mid N$  and $\ord_\fp(G(\v))\geq -1$,}\\
\frac{1}{2} & \text{ otherwise.}
\end{cases}
$$
\end{lemma}

\begin{proof}
Let 
$\fp$ be a prime ideal of residue degree $1$,
so that $\n\fp=p$, for a rational prime $p$.
We may assume that $p$ is unramified in $K$ and that 
$$
\fp\nmid 2a_1\cdots a_n b_1\cdots b_m,
$$
since the desired estimate is trivial otherwise.
Since $K/\QQ$ is Galois, this means that there is a factorisation $(p)=\fp_1\cdots\fp_d$
into prime ideals, where $\fp_1,\dots,\fp_d$ are the $d$ conjugates of $\fp$ under $\Gal(K/\QQ)$, satisfying 
$\n\fp_i=p$ for $1\leq i\leq d$. 

It will be convenient to write 
$S_\fp=S_\mathfrak{p}  (N;\v)$ and 
$\tilde\fp=\fp^{\tau^{-1}}$ in the proof.
Then  $\fp$ and $\tilde\fp$ are distinct prime ideals, with 
$\gfp=\fp \tilde\fp$ and $\n \fp=\n\tilde{\fp}=p$.
Choose $\gamma=g/\alpha \in \mathfrak{F}(\fp)$ 
as in Lemma \ref{lem:orthogonal}, so that 
$\psi(\gamma \cdot)$ is a primitive character modulo $\fp$. Then we can write
\begin{align*}
S_{\fp}  
&=
\sum_{a\in (\fo/\fp)^*} 
\sum_{\x\bmod{\gfp}}\psi\left( \gamma a (F(\x)-N)+   \v. \x\right)\\
&=
\sum_{a\in (\fo/\fp)^*} \psi(-\gamma a N)
\sum_{\x\bmod{\gfp}}\psi\left( \gamma\left\{ a F(\x)+   \alpha \v. \x\right\}\right),
\end{align*}
as in \eqref{eq:SUM}.

Lemma  \ref{lem:skinner-3} yields
$$
S_{\fp}=
\sum_{a\in (\fo/\fp)^*} \psi(-\gamma a N)
\sum_{\u\in (\fo/\fp)^n} 
\sum_{\w\in (\fo/\tilde\fp)^n} 
\psi\left(\gamma\left\{ a F(\mu \u+\lambda \w ) +\alpha(\mu \u+\lambda \w ).\v\right\}\right),
$$
for suitable  $\lambda,\mu\in \fo$ such that  
$$
\ord_{\fp}(\mu)=\ord_{\tilde\fp}(\lambda)=0 \quad \text{ and }\quad
\ord_{\fp}(\lambda)=\ord_{\tilde\fp}(\mu)=1.
$$
Clearly 
$$
\psi\left(\gamma\alpha (\mu \u+\lambda \w ).\v\right)
=\psi\left(\gamma \alpha\mu \u.\v\right)\psi\left( \gamma \alpha\lambda \w.\v\right)
$$
and 
$$
\psi\left( \gamma a F(\mu	 \u+\lambda \w )\right)=
\psi\left( \gamma a \left\{
\mu^2\sum_{i=1}^n a_iu_i^2 +(\lambda^\tau)^2 \sum_{i=1}^m b_i (w_i^{\tau})^2\right\}\right),
$$
since the characters $\psi( \gamma\lambda \cdot)$ 
and 
$\psi( \gamma\mu^\tau \cdot)$ are both trivial on $\fo$.
Putting everything together, it  follows that 
\begin{equation}\label{eq:waypoint}
S_{\fp}=\Sigma_0\sum_{a\in (\fo/\fp)^*} \psi(-\gamma a N)\Sigma_1(a)\Sigma_2(a), 
\end{equation}
where
\begin{align*}
\Sigma_0
&=\prod_{i={m+1}}^n \sum_{w \in \fo/\tilde\fp}
\psi\left(\gamma \alpha \lambda w  v_i\right),\\
\Sigma_1(a)&=
\sum_{\u\in (\fo/\fp)^n} 
\psi\left( \gamma \left\{a 
\mu^2\sum_{i=1}^n a_iu_i^2 + \alpha\mu \u.\v\right\}\right),\\
\Sigma_2(a)&=
\sum_{\w\in (\fo/\tilde\fp)^m} 
\psi\left(  \gamma \left\{a 
(\lambda^\tau)^2 \sum_{i=1}^m b_i (w_i^{\tau})^2
+ \alpha\lambda \sum_{i=1}^m w_iv_i\right\}\right).
\end{align*}
We estimate the first sum trivially via $\Sigma_0\ll (\n\tilde\fp)^{n-m}=(\n\fp)^{n-m}$.

The second sum factorises as  
\begin{align*}
\Sigma_1(a)
 &=
\prod_{i=1}^n \sum_{u\in \fo/\fp} 
\psi\left( \gamma \left\{a  \mu^2
 a_iu^2 +    u(\alpha\mu v_i)\right\}\right),
\end{align*}
Recall from the definition of $\mathfrak{F}(\fp)$ that $\alpha\in \fp \fd$. Hence
$$
\alpha\mu v_i\in \fp \fd \cdot \tilde{\fp} \cdot \widehat\gfp = \fo,
$$
since $\ord_{\tilde \fp}(\mu)=1$ and 
$\v\in \widehat{\gfp}^n$, by assumption.
Making   the change of variables 
$$
u\to u-\bar{2a\mu^2 a_i} \alpha\mu v_i,
$$ 
where $\bar{2a\mu^2a_i}$ denotes the multiplicative inverse of 
$2a\mu^2 a_i$ modulo $\fp$, we are led to the expression
$$
\sum_{u\in \fo/\fp} 
\psi\left( \gamma \left\{a  \mu^2
 a_iu^2 +    u(\alpha\mu v_i)\right\}\right)
=\psi\left(- \gamma\bar{4a \mu^2a_i} (\mu \alpha v_i)^2\right)
\sum_{u\in \fo/\fp} 
\psi\left( \gamma a  \mu^2
 a_iu^2 \right),
$$
since $\bar{4}-\bar{2}\equiv -\bar{4} \bmod{\fp}$.
The inner sum is a classical Gauss sum, as found in work of Hecke \cite[Satz 155]{hecke}, for example. 
We obtain
$$
\sum_{u\in \fo/\fp} 
\psi\left( \gamma \left\{a  \mu^2
 a_iu^2 +    u(\alpha\mu v_i)\right\}\right)
 =
\left(\frac{aa_i}{\fp}\right)\tau_\fp 
\psi\left( -\gamma\bar{4a \mu^2a_i} (\mu \alpha v_i)^2\right),
$$
where 
$$
\tau_\fp
=
\sum_{u\in \fo/\fp} 
\psi\left(  \gamma
u^2 \right).
$$
This completes the proof of the identity
$$
\Sigma_1(a)= \left(\frac{a}{\fp}\right)^n \left(\frac{a_1\cdots a_n}{\fp}\right) \tau_\fp^n 
\psi\left(-\gamma  \bar{4a\mu^2}  \sum_{i=1}^n \bar{a_i} (\mu \alpha v_i)^2\right).
$$

It turns out that the remaining sum $\Sigma_2(a)$ can also be interpreted as a product of Gauss sums. 
First, we observe that we have the factorisation
\begin{align*}
\Sigma_2(a)
&=
\prod_{i=1}^m
\sum_{w\in \fo/\tilde\fp} 
\psi\left(  \gamma \left\{a 
(\lambda^\tau)^2 b_i (w^{\tau})^2
+ \alpha\lambda wv_i\right\}\right)\\
&=
\prod_{i=1}^m
\sum_{u\in \fo/\fp} 
\psi\left(  \gamma \left\{a 
(\lambda^\tau)^2 b_i u^2
+ \alpha\lambda u^{\tau^{-1}}v_i\right\}\right),
\end{align*}
on making the change of variables $u=w^\tau$.
The  trace is left invariant under conjugation. On recalling that $g\in  \ZZ$, so that $g^\tau=g$,  it therefore follows that 
$$
\psi\left(  \gamma  \alpha\lambda u^{\tau^{-1}}v_i\right)=
\psi\left(  \gamma^\tau  \alpha^\tau \lambda^\tau uv_i^\tau\right)=
\psi\left(  \gamma  \alpha \lambda^\tau uv_i^\tau\right),
$$
since $(\gamma\alpha)^\tau=g=\gamma\alpha$.
Hence 
\begin{align*}
\Sigma_2(a)
&=
\prod_{i=1}^m
\sum_{u\in \fo/\fp} 
\psi\left(  \gamma \left\{a 
(\lambda^\tau)^2 b_i u^2
+ u(\alpha \lambda^\tau v_i^\tau)\right\}\right),
\end{align*}
where
$$
\alpha \lambda^\tau v_i^\tau \in \fp\fd \cdot \fp^\tau \cdot (\widehat \gfp)^\tau  \in \fo,
$$
for $1\leq i\leq m$. The inner sum is a  Gauss sum that we can evaluate, as previously. This yields
$$
\Sigma_2(a)
=
 \left(\frac{a}{\fp}\right)^m \left(\frac{b_1\cdots b_m}{\fp}\right) \tau_\fp^m
\psi\left(-\gamma  \bar{4a(\lambda^\tau)^2}  \sum_{i=1}^m \bar{b_i} (\lambda^\tau \alpha v_i^\tau)^2\right).
$$

We now piece  everything together in \eqref{eq:waypoint}.
To begin with, it follows from   squaring and differencing that
$$
|\tau_\fp|^2= 
\sum_{u\in \fo/\fp} 
\psi\left(\gamma    
u^2 \right)
\sum_{v\in \fo/\fp} 
\psi\left( 2 \gamma  
uv \right).
$$
Since $\fp\nmid 2$, we see that the  inner sum is $\n\fp$ if $u\in \fp$ and 
and $0$ otherwise.
 Hence it follows that 
$
|\tau_\fp|=\sqrt{\n\fp},
$
from which we deduce that 
$$
 S_{\fp}
\ll ( \n \fp)^{(3n-m)/2} \left|
\sum_{a\in (\fo/\fp)^*}  \left(\frac{a}{\fp}\right)^{m+n} 
\psi\left(\gamma \left\{ -aN-\bar{4a} M\right\}
\right)
\right|,
$$
where
$$
M=
\bar{\mu^2}\sum_{i=1}^n \bar{a_i} (\mu \alpha v_i)^2
+
\bar{ (\lambda^\tau)^2}\sum_{i=1}^m \bar{b_i} (\lambda^\tau \alpha v_i^\tau)^2.
$$
Since $\fp\nmid 2a_1\cdots a_n b_1\cdots b_m \mu \lambda^\tau$, we 
may replace 
$\bar a$ by 
$4a a_1\cdots a_n b_1\cdots b_m \mu^2(\lambda^\tau)^2$ by $a$, in order to obtain  
$$S_{\fp}(\v)\ll ( \n \fp)^{(3n-m)/2}|K_\fp|,
$$ 
with
$$
K_\fp
=
\sum_{a\in (\fo/\fp)^*}  \left(\frac{a}{\fp}\right)^{m+n} 
\psi\left( \gamma \left\{-a \mu^2 (\lambda^\tau)^2 \alpha^2G( \v)
-\bar{4a a_1\cdots a_n b_1\cdots b_m \mu^2(\lambda^\tau)^2}N
\right\}\right),
$$
with  $G$ is given by \eqref{eq:def-G}. One notes that 
$\mu^2 (\lambda^\tau)^2 \alpha^2G( \v)\in \fo$ when $\v\in (\widehat \gfp)^n$.
In particular 
$$
\ord_\fp\left(\mu^2 (\lambda^\tau)^2 \alpha^2G( \v)\right)=\ord_\fp(G(\v))+2.
$$
Thus $K_\fp$
 is a Kloosterman sum, if $2\mid m+n$, and a Sali\'e sum if $2\nmid m+n$. It follows that 
 $$
 K_\fp
\ll 
\begin{cases}
\n\fp & \text{ if $\fp\mid N$ and $\ord_\fp(G(\v))+2>0$,}\\ 
\sqrt{\n\fp}  &\text{ otherwise}.
\end{cases}
 $$
The statement of the lemma is now clear.
\end{proof}

 We are now ready to reveal  our final estimate for 
the exponential sum $S_\mathfrak{b}  (N;\m)$.

\begin{lemma}\label{lem:b=p}
Let $\ve>0$. 
Let $\fb\subset \fo$ be a non-zero ideal and  let 
 $\m\in \widehat\gfb^n$. 
Then 
$$
S_\mathfrak{b}  (N;\m)\ll 
  (\n\fb)^{\frac{1}{2}-(n-m)/2+\ve}  
(\n \gfb)^{n-m/2} 
\prod_{
\substack{
\fp \mid (\fb,N)\\ 
\n\fp
 \| \n\fb\\
\ord_\fp(G(\m))\geq -1\\
}
}(\n\fp)^{\frac{1}{2}}
\hspace{0.2cm}
\prod_{
\substack{p^k\| \n \fb \\ k\geq 2}}
p^{\frac{k}{2}},
$$
where $G$ is given by 
\eqref{eq:def-G}.
\end{lemma}

\begin{proof}
There is a factorisation $\fb=\fb_1\fb_2$, where $\n\fb_1$ is square-free and 
$\n\fb_2$ is square-full, with $\gcd(\n\fb_1,\n\fb_2)=1$. 
It follows from 
Lemma \ref{lem:mult} and  Corollary~\ref{cor:2}
that 
\begin{equation}\label{eq:breath}
\begin{split}
S_\fb(N;\m)
&=
  S_{\fb_1}  (\bar{\n\fb_2}^2N;(\n\fb_2)\m)  S_{\fb_2}  (\bar{\n\fb_1}^2N;(\n\fb_1)\m)\\
&\ll 
|S_{\fb_1}  (\bar{\n\fb_2}^2N;(\n\fb_2)\m) |
  (\n\fb_2)^{1-(n-m)/2} 
(\n \gfb_2)^{n-m/2}.
\end{split}
\end{equation}
We now turn to $S_{\fb_1}  (\bar{\n\fb_2}^2N;(\n\fb_2)\m)$, in which we note that 
$$
(\n\fb_2)m_i\in (\n\fb_2) \widehat{\gfb}\in  \widehat{\gfb_1} (\n\fb_2) \gfb_2^{-1}\in \widehat{\gfb_1},
$$ 
for $1\leq i\leq n$.  Since   $\n \fb_1$ is square-free, we have a factorisation
$$
\fb_1=\fq_1\cdots\fq_r,
$$
for distinct prime ideals $\fq_1,\cdots,\fq_r$ of residue degree $1$ such that $\n\fq_1,\cdots,\n\fq_r$ are distinct rational primes.
Let 
$$
c_i=\prod_{\substack{j=1\\ j\neq i}}^r \n\fq_j,
$$
for $1\leq i\leq r$.
It now follows from a further application of 
Lemma \ref{lem:mult} that 
$$
S_{\fb_1}  (\bar{\n\fb_2}^2N;(\n\fb_2)\m) 
=S_{\fq_1}(\bar{\n\fb_2}^2\bar{c_1}^2N;(\n\fb_2)c_1\m)\cdots
S_{\fq_r}(\bar{\n\fb_2}^2\bar{c_r}^2N;(\n\fb_2)c_r\m).
$$
In particular, we plainly have 
$(\n\fb_2)c_i\m\in (\widehat{\gfq_i})^n$, for $1\leq i\leq r$. 

We are now aligned for an application of Lemma \ref{lem:prime}.
For each $i\in \{1,\dots,r\}$, this yields
$$
S_{\fq_i}(\bar{\n\fb_2}^2\bar{c_i}^2N;(\n\fb_2)c_i\m)
\ll
(\n\fq_i)^{\theta_{\fq_i}+(3n-m)/2},  
$$
where
$$
\theta_{\fq_i}=
\begin{cases}
1 & \text{ if $\fq_i\mid N$  and $\ord_{\fq_i}(G(\m))\geq -1$,}\\
\frac{1}{2} & \text{ otherwise.}
\end{cases}
$$
Note that 
$$
(\n\fq_i)^{\theta_{\fq_i}+(3n-m)/2}=
  (\n\fq_i)^{\theta_{\fq_i}-(n-m)/2}  
(\n \gfq_i)^{n-m/2},
$$
since $\n\gfq_i=(\n\fq_i)^2$. 
Thus
$$
S_{\fb_1}  (\bar{\n\fb_2}^2N;(\n\fb_2)\m) \ll
(\n\fb_1)^{\frac{1}{2}-(n-m)/2+\ve}  
(\n \gfb_1)^{n-m/2}
\prod_{
\substack{
\fp\mid (\fb_1,N)\\ \ord_\fp(G(\m))\geq -1}
}(\n\fp)^{\frac{1}{2}}.
$$
Combining these estimates in \eqref{eq:breath}, we conclude that 
\begin{align*}
S_\fb(N;\m)
&\ll 
  (\n\fb)^{\frac{1}{2}-(n-m)/2+\ve} 
(\n \gfb)^{n-m/2} (\n\fb_2)^{\frac{1}{2}}
\prod_{
\substack{
\fp\mid (\fb_1,N)\\ \ord_\fp(G(\m))\geq -1}
}(\n\fp)^{\frac{1}{2}},
\end{align*}
since 
$(\n \fb_1)(\n \fb_2)=\n\fb$ and
$(\n \gfb_1)(\n \gfb_2)=\n \gfb$. 
The statement of the lemma is now clear.
\end{proof}

Our next task is to analyse the oscillatory integral $K(u,P\m)$
when $F$ is given by 
\eqref{eq:special}, 
based on 
\eqref{eq:K(um)}. 
To the fixed automorphism  $\tau\in \Gal(K/\QQ)$ 
in \eqref{eq:special}, 
we can associated a unique 
integer $l_\tau\in \{1,\dots,d\}$, as in \eqref{eq:l_tau}.
We therefore have 
$$
F^{(l)}(\underline\x)=
\sum_{i=1}^n a_i^{(l)} (x_i^{(l)})^2
+ \sum_{i=1}^m b_i^{(l)} (x_i^{(l_{\tau^{-1}})})^2,
$$
for $1\leq l\leq d$. 
Let  $\mathbf{A}_l=\diag(a_1^{(l)},\dots, a_n^{(l)})$
and   $\mathbf{B}_l=\diag(b_1^{(l)},\dots,b_m^{(l)}, 0,\dots,0)$. 
Then it follows that the quadratic form \eqref{eq:real-Q}
has an  underlying matrix which is the block diagonal
matrix
$$
\begin{pmatrix}
u_1 \mathbf{A}_1+u_{1_\tau} \mathbf{B}_{1_\tau}  & \mathbf{0} & \cdots & \mathbf{0} \\
\mathbf{0} & u_2 \mathbf{A}_2+u_{2_\tau} \mathbf{B}_{2_\tau}  &  \cdots & \mathbf{0} \\
\vdots & \vdots &\ddots & \vdots\\
\mathbf{0} & \mathbf{0}& \cdots &
u_d \mathbf{A}_d+u_{d_{\tau}} \mathbf{B}_{d_\tau}
\end{pmatrix}.
$$
If  $\m$ is given coordinates 
$
\underline\m=(\m^{(1)},\dots,\m^{(d)})
$
on $V^n$, then we have 
\begin{align*}
K(u,P\m)
&=
\prod_{l=1}^d
\int_{ \RR^{n}}W(\x^{(l)})
e\left(G^{(l)}(\x^{(l)})-P\m^{(l)}.\x^{(l)}\right)\d\x^{(l)},
\end{align*}
where $G^{(l)}$ has underlying matrix  $u_l \mathbf{A}_l +u_{l_\tau}\mathbf{B}_l$.
Since this matrix is diagonal, 
on assuming that the weight $W$ is chosen suitably, 
 we may further factorise to obtain
$$
K(u,P\m)
=
\prod_{l=1}^d 
H_1^{(l)}\cdots H_m^{(l)} I_{m+1}^{(l)}\cdots I_{n}^{(l)},
$$
where we write
$$
H_i^{(l)}=
\int_{ \RR}W(x)
e\left( (a_i^{(l)}u_l +b_i^{(l_\tau)}u_{l_\tau})x^2
 -Pm_i^{(l)}x\right)\d x 
 $$
 for $i\leq m$,
 and 
 $$
 I_i^{(l)}=
\int_{ \RR}W(x)
e\left( a_i^{(l)}u_l x^2
 -Pm_i^{(l)}x\right)\d x 
$$
for $i>m$.

\begin{lemma}\label{lem:tornado}
Let 
$$
L_i(u)=a_iu+\tau^{-1}(b_iu),
$$
for $1\leq i\leq m$. 
Then, for any $\ve>0$,  $K(u,P\m)$ is essentially supported on the set of $u$ and $\m$ for which 
\begin{equation}\label{eq:mot}
|m_i^{(l)}| \ll 
\begin{cases} 
P^{-1+\ve} |\rho_l(L_i(u))| &\text{ if $i\leq m$,}\\
P^{-1+\ve} |u_l| &\text{ if $i> m$,}
\end{cases}
\end{equation}
for $1\leq l\leq d$. Moreover, we have 
$$
K(u,P\m)\ll   \frac{1}{\sqrt{\mathfrak{H}(L_1(u))\cdots 
 \mathfrak{H}(L_m(u))
 \mathfrak{H}(u)^{n-m}}}.
$$
\end{lemma}

\begin{proof}
Clearly we get exponential decay in $K(u,P\m)$ unless $P|\m|\ll |u| P^\ve$, as we now assume. 
However, on examining each of the factors in $K(u,P\m)$ separately, the essential support of 
$K(u,P\m)$ is rendered clear. Next, for each $i\leq m$ and  $1\leq l\leq d$, 
we have 
$
\rho_l(L_i(u))=a_i^{(l)}u_l+b_i^{(l_\tau)}u_{l_\tau}.
$
The  second derivative bound for exponential integrals yields
$$
H_i^{(l)}\ll \min\left(1,|\rho_l(L_i(u))|^{-1/2}\right),
$$
for $i\leq m$, and 
$$
I_i^{(l)}\ll \min\left(1,|u_l|^{-1/2}\right) ,
$$
for $i> m$. The statement is now clear. 
\end{proof}

We now piece everything together in  
our expression 
\eqref{eq:ENPXY} for 
$E(N;P;X,Y)$.
We shall continue to adhere to the convention that the value of $\ve>0$ is allowed to change at each appearance, and that all
implied constants are allowed to depend on $\ve$.

Recall
the definition of  $\mathcal{U}$ from  \eqref{eq:breeze}.
Combining \eqref{eq:ENPXY} and Lemma 
\ref{lem:tornado}, we obtain
$$
E(N;P;X,Y)\ll_A P^\ve Y^{-n}
\sum_{\substack{\fb\in \mathcal{B}(X,Y)}}
\sum_{\substack{0\neq \m\in \widehat{\gfb  }^n}}
|S_\mathfrak{b}  (N;\m)|
\int_{\mathcal{R}(\m)} 
f(u)\d u +P^{-A},
$$
where now
\begin{equation}\label{eq:def-f}
f(u)=\frac{1}
{\sqrt{\mathfrak{H}(L_1(u))\cdots 
 \mathfrak{H}(L_m(u))
 \mathfrak{H}(u)^{n-m}}}
\end{equation}
and 
$\mathcal{R}(\m)$ denotes  the set of 
$u\in \mathcal{U}$ such that  \eqref{eq:mot} holds.

We now make the exact same change of variables 
$\c=\delta \m$ that we made previously in \eqref{eq:change-var}.
Then, in particular,  we can assume that 
\eqref{eq:squall} holds.
Moreover, on dropping the information about $G(\m)$,  Lemma \ref{lem:b=p} yields
\begin{align*}
S_\mathfrak{b}  (N;\delta^{-1}\c)
&\ll 
  (\n\fb)^{\frac{1}{2}-(n-m)/2+\ve}  
(\n \gfb)^{n-m/2} \sqrt{g(\fb)}\\
&\ll 
X^{\frac{1}{2}-(n-m)/2+\ve}
Y^{n-m/2} \sqrt{g(\fb)},
\end{align*}
where 
$$
g(\fb)=
\prod_{
\substack{
\fp \mid (\fb,N)\\ 
\n\fp
 \| \n\fb}}\n\fp
\hspace{0.2cm}
\prod_{
\substack{p^k\| \n \fb \\ k\geq 2}}
p^{k}.
$$
In this notation we conclude that 
\begin{align*}
E(N;P;X,Y)
&\ll_A
  \frac{X^{\frac{1}{2}-(n-m)/2}P^\ve }{
Y^{m/2}} 
\sum_{\fb\in \mathcal{B}(X,Y)}
\sum_{\substack{0\neq \c\in \fo^n\\ \text{\eqref{eq:squall} holds}}}
\sqrt{g(\fb)}
\int_{\mathcal{R}(\delta^{-1}\c)} 
f(u)\d u +P^{-A}.
\end{align*}
Let 
$$
L(u)=
\sum_{\fb\in \mathcal{B}(X,Y)} 
\sum_{
\c\in \mathcal{C}(u,\fb)}
\sqrt{g(\fb)},
$$
where 
$\mathcal{C}(u,\fb)$ is the set of non-zero vectors $\c\in \fo^n$ for which \eqref{eq:squall} holds and
$$
|c_i^{(l)}| \ll 
\begin{cases} 
P^{-1+\ve} Y^{1/d} |\rho_l(L_i(u))| &\text{ if $i\leq m$,}\\
P^{-1+\ve}  Y^{1/d} |u_l| &\text{ if $i> m$,}
\end{cases}
$$
for $1\leq l\leq d$. Then  we may write 
\begin{equation}\label{eq:waypoint9}
E(N;P;X,Y)
\ll 
  \frac{X^{\frac{1}{2}-(n-m)/2}P^\ve }{
Y^{m/2}} 
\int_{\mathcal{U}}
f(u) L(u)
 \d u,
\end{equation}
rendering our next task to estimate $L(u)$. The following result is an analogue of Lemma~\ref{lem:wish}.

\begin{lemma}\label{lem:wish2}
Let $u\in V$ be such that $\mathfrak{H}(u)\leq P^{d+\ve}/X$.
If $L(u)\neq 0$ then 
\begin{equation}\label{eq:this}
\frac{P^{d-\ve}}{X}\ll 
\mathfrak{H}(u)\ll \frac{P^{d+\ve}}{X}
\quad \text{ or }\quad 
P^{-d+\ve} Y \max_{1\leq i\leq m} 
\mathfrak{H}(L_i(u))\gg 1.
\end{equation}
Moreover, we have 
$
L(u)\ll P^\ve X J(u),
$
where
$$
 J(u)=
\prod_{i=1}^{m}
\max\left\{1,   
P^{-d} Y  \mathfrak{H}(L_i(u))\right\}.
$$
\end{lemma}

\begin{proof}
Let  us write $\fb=\fb_1\fb_2$, where $\n\fb_1$ is square-free and 
$\n\fb_2$ is square-full, with $\gcd(\n\fb_1,\n\fb_2)=1$. 
Then 
$
g(\fb)=
\n\fb_2 \n\mathfrak{h},
$
where 
$\mathfrak{h}$ is the greatest common ideal divisor of $\fb_1$ and $N$.
In summary, we may now write
\begin{equation}\label{eq:Lu}
L(u)\leq 
\sum_{\substack{\n\fb_2\ll X\\ \text{$\n\fb_2$ square-full}}}\sqrt{\n\fb_2}
\sum_{\substack{\n\fb_1\ll X/(\n \fb_2)\\ 
\n\gfb_1 \ll Y/\n\gfb_2\\
\gcd(\n \fb_1,\n \fb_2)=1}} \mu^2(\n\fb_1) 
\sum_{\mathfrak{h}\mid (\fb_1,N)} \sqrt{\n\mathfrak{h}}
\# \mathcal{C}(u,\fb_1\fb_2).
\end{equation}

In order to proceed, we assume without loss of generality that $u\in V$ satisfies
$$
\mathfrak{H}(L_1(u))\leq\cdots\leq \mathfrak{H}(L_m(u)).
$$
Let us write $\c=(\c',\c'')$, where $\c'=(c_1,\dots,c_m)$ and $\c''=(c_{m+1},\dots,c_n)$. 
Keeping in mind  \eqref{eq:squall}, we first fix  a choice of $\c''\in \left((\fb_1\fb_2)^{-1} \gfb_1 \gfb_2\right)^{n-m}$ satisfying
$$
|c_i^{(l)}| \ll 
P^{-1+\ve}  Y^{1/d} |u_l| ,
$$
for $m+1\leq i\leq n$ and $1\leq l\leq d$.  We claim that 
 there exists
 $\lambda\in K$ such that 
$$(\lambda)=(\fb_1\fb_2)^{-1} \gfb_1 \gfb_2  \fp_2^{-1},
$$ 
for a suitable prime ideal $\fp_2$ of norm $O(P^\ve)$. 
To begin with, it follows from part (ii) of Lemma \ref{lem:alg1} that
 there exists
 $\lambda_3\in \fo$ such that 
$(\lambda_3)=(\fb_1\fb_2)^{-1} \gfb_1 \gfb_2  \fp_3$, for a suitable prime ideal $\fp_3$ of norm $O(P^\ve)$. 
A second application of this result reveals that there exists $\lambda_2\in \fo$ and a prime ideal
$\fp_2$ of norm $O(P^\ve)$, such that $(\lambda_2)=\fp_3\fp_2$. The claim now follows with 
$\lambda=\lambda_3/\lambda_2$.

On multiplying by units we can further assume that 
$$
X^{-1/d}Y^{1/d}\ll |\lambda^{(l)}|\ll X^{-1/d}Y^{1/d+\ve},
$$
for $1\leq l\leq d$, on
recalling that  $\n\fb_1\n\fb_2 \asymp X$ and
 $\n\gfb_1\n\gfb_2 \asymp Y$.
Making the change of variables $\c''=\lambda \mathbf{d}''$,
we deduce that 
for $m+1\leq i\leq n$, we have 
$d_i\in \fo$ and
$$
|d_i^{(l)}| \ll 
P^{-1+\ve}  X^{1/d} |u_l| ,
$$
for  $1\leq l\leq d$.
In particular, 
if $\c''\neq \0$ then there exists $i\in \{m+1,\dots,n\}$ such that 
$$
1\leq |N_{K/\QQ}(d_i)| \ll 
P^{-d+\ve}  X |\nm (u)| .
$$
Recalling that 
 $\nm(u)\leq \mathfrak{H}(u)\ll P^{d+\ve}/X$, we deduce that 
\begin{equation}\label{eq:neq0-1}
\c''\neq \0 \Longrightarrow   
\frac{P^{d-\ve}}{X}\ll 
\mathfrak{H}(u)\ll \frac{P^{d+\ve}}{X}.
\end{equation}
Moreover, arguing as in Lemma \ref{lem:wish}, it readily  follows from a result in Lang
 \cite[Thm.~0 in \S V.1]{lang} that  
 the overall  number of  vectors $\mathbf{d}''$ is $O(P^\ve)$.
We must next address the number of $\c'\in \fo^m$, with 
$(\c',\c'')\neq \0$, which satisfy
$$
|c_i^{(l)}| \ll 
P^{-1+\ve} Y^{1/d} |\rho_l(L_i(u))|,
$$
for $1\leq i\leq m$ and $1\leq l\leq d$.
It is clear that 
\begin{equation}\label{eq:neq0-2}
\c'\neq \0 \Longrightarrow   
1\ll P^{-d+\ve} Y 
\mathfrak{H}(L_m(u)).
\end{equation}
Together,  \eqref{eq:neq0-1} and \eqref{eq:neq0-2} yield the first part of the lemma. 

Appealing once more to 
Lang
 \cite[Thm.~0 in \S V.1]{lang}, 
we deduce that 
  the number of $\c'$ is
  \begin{align*}
&\ll  \prod_{i=1}^{m}\left(1+  \prod_{l=1}^d
P^{-1+\ve} Y^{1/d} |\rho_l(L_i(u))| 
\right)
\ll  J(u),
\end{align*}
in the notation of the lemma. 
Returning to 
\eqref{eq:Lu}, 
we deduce that 
\begin{align*}
L(u)
&\ll P^\ve J(u)
\sum_{\substack{\n\fb_2\ll X\\ \text{$\n\fb_2$ square-full}}}\sqrt{\n\fb_2}
\sum_{\mathfrak{h}\mid N} \sqrt{\n\mathfrak{h}}
\sum_{\substack{\n\fb_1\ll X/(\n \fb_2)\\ \mathfrak{h}\mid \fb_1}} 1\\
&\ll P^\ve J(u)
\sum_{\substack{\n\fb_2\ll X\\ \text{$\n\fb_2$ square-full}}}\sqrt{\n\fb_2}
\sum_{\mathfrak{h}\mid N} \sqrt{\n\mathfrak{h}} \cdot \frac{X}{(\n \fb_2) (\n \mathfrak{h})}\\
&\ll P^\ve X J(u)
\sum_{\substack{\n\fb_2\ll X\\ \text{$\n\fb_2$ square-full}}}\frac{1}{\sqrt{\n\fb_2}},
\end{align*}
since there are $O(1)$ ideal divisors $\mathfrak{h}\mid N$ when $N\in \fo$ is non-zero.
Finally, the lemma follows on noting that  there are $O(\sqrt{X})$ integral ideals such that $\n\fb_2$ is a square-full integer of modulus at most $X$. 
\end{proof}

We may now
apply Lemma \ref{lem:wish2} in \eqref{eq:waypoint9}. 
Let 
$R$ denote the set of $u\in \mathcal{U}$ such that 
\eqref{eq:this} holds.
On recalling the definition \eqref{eq:def-f} of $f(u)$, we deduce that 
\begin{equation}\label{eq:resin}
E(N;P;X,Y)
\ll_A  P^{-A}+
  \frac{X^{\frac{3}{2}-(n-m)/2}P^\ve }{
Y^{m/2}} I(X,Y),
\end{equation}
where
\begin{align*}
I(X,Y)&=
\int_{R}
\frac{\prod_{i=1}^{m}
\max\left\{1,   
P^{-d} Y  \mathfrak{H}(L_i(u))\right\}}{
{\sqrt{\mathfrak{H}(L_1(u))\cdots 
 \mathfrak{H}(L_m(u))
 \mathfrak{H}(u)^{n-m}}}}
 \d u.
\end{align*}
The following result deals  with  this integral.

\begin{lemma}\label{lem:I-theta}
We have 
\begin{align*}
I(X,Y)\ll ~&
P^\ve (P^{-d} Y)^{m}
\left( \left(\frac{P^{d}}{X}\right)^{3m/2-n/2+1}+1\right)\\
&\quad +
c_YP^\ve \left( 
\frac{P^{d(m/2-n/2+1)}Y^{m-1}}{X^{3m/2-n/2}}+
 \frac{Y}{P^d}\right)+P^\ve \left(\frac{P^d}{X}\right)^{1-(n-m)/2}
 \hspace{-1cm},
 \end{align*}
where
$$
c_Y=\begin{cases}
1 &\text{ if $Y\leq P^{d}$,}\\
0 & \text{ otherwise}.
\end{cases}
$$
\end{lemma}

\begin{proof}
In the proof of this result we shall make  frequent use of the observation that 
\begin{align*}
\mathfrak{H}(L_i(u))
&=\prod_{l=1}^d \max\{1,|a_i^{(l)}u_l+b_i^{(l_\tau)}u_{l_\tau}|\}\ll \mathfrak{H}(u)^2,
\end{align*}
for any $i\in \{1,\dots,n\}$, which follows from \eqref{eq:carpet}.

We may assume without loss of generality  that  the range of integration is restricted to satisfy
\begin{equation}\label{eq:moose}
\mathfrak{H}(L_1(u))\leq \cdots\leq \mathfrak{H}(L_m(u)).
\end{equation}
We further break the range of integration into $m+1$ regions.
For $0\leq t\leq m$, let  $R_t$ denote the set of 
$u\in V$ with 
$\mathfrak{H}(u)\leq P^{d+\ve}/X$,
such that \eqref{eq:this} 
 and \eqref{eq:moose} hold, with 
$$
\mathfrak{H}(L_t(u))\ll P^{d-\ve}/Y,\quad 
\mathfrak{H}(L_{t+1}(u))\gg P^{d-\ve}/Y.
$$
(Note that the left inequality is vacuous when $t=0$, and similarly for the right hand inequality when $t=m$.) In particular, it is clear that 
$
R_m=\emptyset
$
when the second inequality in \eqref{eq:this} holds.
Moreover, 
when   $t\in \{1,\dots,m-1\}$, 
we observe that 
$$
R_t\neq  \emptyset \Longrightarrow
 Y\ll P^{d-\ve},
$$
since $\mathfrak{H}(L_t(u))\geq 1$.
We have 
\begin{align*}
I(X,Y)
&\ll \sum_{t=0}^{m} P^\ve
\int_{
 R_t}
\frac{\left(
(P^{-d} Y)^{m-t}  \mathfrak{H}(L_{t+1}(u))\cdots 
 \mathfrak{H}(L_{m}(u))\right)}{
{\sqrt{\mathfrak{H}(L_1(u))\cdots 
 \mathfrak{H}(L_m(u))
 \mathfrak{H}(u)^{n-m}}}}
 \d u.
 \end{align*}
 Thus
\begin{align*}
 I(X,Y)
&\ll
 \sum_{t=0}^{m}
 I^{(t)}(X,Y),
 \end{align*}
 where
$$
I^{(t)}(X,Y)=
(P^{-d} Y)^{m-t}P^\ve
\int_{
 R_t}
\frac{\left(
  \mathfrak{H}(L_{t+1}(u))\cdots 
 \mathfrak{H}(L_{m}(u))\right)^{\frac{1}{2}}}{
{\mathfrak{H}(u)^{(n-m)/2}}}
 \d u,
$$
on taking 
$\mathfrak{H}(L_1(u))\cdots \mathfrak{H}(L_t(u))\geq 1$.

We first deal with 
$ I^{(0)}(X,Y)$.
Recalling that 
 $\mathfrak{H}(L_i(u))\ll \mathfrak{H}(u)^2$, for $1\leq i\leq m$, it follows that 
$$
\frac{\left(
  \mathfrak{H}(L_{1}(u))\cdots 
 \mathfrak{H}(L_{m}(u))\right)^{\frac{1}{2}}}{ \mathfrak{H}(u)^{(n-m)/2}}
\ll \mathfrak{H}(u)^{3m/2-n/2}.
$$
If $
3m/2-n/2\geq -1$, then 
 \eqref{eq:BoundB} yields
$$
\int_{
 R_0}
 \mathfrak{H}(u)^{3m/2-n/2} \d u \ll 
 P^\ve \left(\frac{P^{d}}{X}\right)^{3m/2-n/2+1}.
 $$
 Alternatively, 
 if  $3m/2-n/2< -1$ then the left hand side is $O(1)$
 by 
  \eqref{eq:BoundC}.
Thus
\begin{align*}
 I^{(0)}(X,Y)
&\ll P^\ve (P^{-d} Y)^{m}
\left( \left(\frac{P^{d}}{X}\right)^{3m/2-n/2+1}+1\right),
\end{align*}
which is satisfactory.

Terms with $1\leq t\leq m-1$ only contribute when $Y\leq P^d$. 
Arguing as above, 
it follows from
 \eqref{eq:BoundC} and \eqref{eq:BoundB}  that 
\begin{align*}
 \sum_{t=1}^{m-1}
 I^{(t)}(X,Y)
&\ll
 \sum_{t=1}^{m-1}
(P^{-d} Y)^{m-t}
\int_{
 R_t}
 \mathfrak{H}(u)^{3m/2-t-n/2}
 \d u\\
&\ll
P^\ve \sum_{t=1}^{m-1}
(P^{-d} Y)^{m-t}
\left( \left(\frac{P^{d}}{X}\right)^{3m/2-t-n/2+1}+1\right)\\
&\ll
P^\ve \sum_{t=1}^{m-1}
\left(
(P^{-d} Y)^{m}
 \left(\frac{P^{d}}{X}\right)^{3m/2-n/2+1}
\left(\frac{X}{Y}\right)^t
+
(P^{-d} Y)^{m-t}\right)\\
&\ll
P^\ve \left\{ 
\frac{P^{d(m/2-n/2+1)}Y^{m-1}}{X^{3m/2-n/2}}+
 \frac{Y}{P^d}
\right\}.
\end{align*}
since $X\leq Y$. This is satisfactory for the lemma.

It remains to 
estimate  $I^{(m)}(X,Y)$. In this case we may assume that  $u$  satisfies the  first inequality in 
\eqref{eq:this}, since $R_m=\emptyset$ otherwise. 
Hence
 \eqref{eq:BoundA} yields
\begin{align*}
I^{(m)}(X,Y)
&=
\int_{\{u \in V : 
P^{d-\ve}/X\ll 
\mathfrak{H}(u)\ll P^{d+\ve}/X\}}
\frac{1}{
{\mathfrak{H}(u)^{(n-m)/2}}}
 \d u
 \ll P^\ve \left(\frac{P^d}{X}\right)^{1-(n-m)/2},
\end{align*}
which is satisfactory and so completes the proof of the lemma.
\end{proof}

It is now time to return to our goal 
of proving that \eqref{eq:end} holds for a suitable $\Delta>0$, for any $X,Y\geq 1$ satisfying \eqref{eq:XY}.
We wish to do so under the assumption that $n-m\geq 4$.
Applying 
Lemma \ref{lem:I-theta} in 
\eqref{eq:resin}, 
the overall contribution
to $E(N;P;X,Y)$
 from the final term
 is seen to be
 \begin{align*}
&\ll   \frac{X^{\frac{3}{2}-(n-m)/2}P^\ve }{
Y^{m/2}}\cdot 
\left(\frac{P^d}{X}\right)^{1-(n-m)/2}\\
&\ll   
 \frac{X^{1/2}}{Y^{m/2}}P^{-d(n-m-2)/2+\ve} \\
 &\ll   
 P^{-d(n-m-2)/2+\ve},
  \end{align*}
on taking $X\leq Y$ and $m\geq 1$. This is $O(P^{-d+\ve})$, 
if $n-m\geq 4$, which  is  satisfactory
for \eqref{eq:end}.
 Next, 
the  second term in Lemma \ref{lem:I-theta} makes the overall contribution
 \begin{align*}
&\ll   \frac{X^{\frac{3}{2}-(n-m)/2}P^\ve }{
Y^{m/2}}\cdot 
 c_Y \left( 
\frac{P^{d(m/2-n/2+1)}Y^{m-1}}{X^{3m/2-n/2}}+
 \frac{Y}{P^d}
\right)\\
&\ll   c_Y P^\ve  \left( 
\frac{Y^{m/2-1}}{X^{m-3/2} P^{d(n-m-2)/2}}+
\frac{X^{3/2-(n-m)/2} }{
Y^{m/2-1}P^d} 
\right).
 \end{align*}
If  $m\geq 2$, we take $Y\leq X^2$ in the first term, and $X,Y\geq 1$ in the second term. 
Assuming that $n-m\geq 4$, this yields
\begin{align*}
&\ll   
 P^{-d(n-m-2)/2+\ve}+
P^{-d+\ve}\ll P^{-d+\ve},
\end{align*}
which is  satisfactory
for \eqref{eq:end}.
If  $m=1$, on the other hand, then we  get the contribution
\begin{align*}
&\ll   c_Y P^\ve  \left( 
\frac{X^{1/2}}{Y^{1/2}P^{d(n-3)/2}}+
\frac{ Y^{1/2}}{X^{(n-4)/2} P^d} 
\right)\ll   P^{-d/2+\ve},
\end{align*}
by \eqref{eq:XY} and the assumption $n\geq m+4=5$, together with the fact that $Y\leq P^d$ when $c_Y\neq 0$.

 Turning to the contribution 
to \eqref{eq:resin}
  from the first term in Lemma \ref{lem:I-theta}, 
 we see that  this is 
\begin{align*}
& \ll 
  \frac{X^{\frac{3}{2}-(n-m)/2}P^\ve }{
Y^{m/2}} \cdot 
 (P^{-d} Y)^{m}
\left( \left(\frac{P^{d}}{X}\right)^{3m/2-n/2+1}+1\right)\\
 &=  \frac{Y^{m/2}P^{d(m-n+2)/2+\ve}}{X^{m-1/2}}
 +
 \frac{X^{\frac{3}{2}-(n-m)/2}Y^{m/2} P^\ve}{
P^{dm} } \\
&\leq  X^{1/2}P^{d(m-n+2)/2+\ve}
+X^{\frac{3}{2}-n/2+3m/2} P^{-dm+\ve}.
\end{align*}
Taking $X\ll P^d$, the first term is 
$O(P^{-d(n-m-3)/2+\ve})$, which is $O(P^{-d/2+\ve})$, if $n-m\geq 4$.
The second term is plainly $P^{-dm+\ve}$ if $\frac{3}{2}-n/2+3m/2\leq 0$, and it is 
$$
\ll P^{d(\frac{3}{2}-n/2+m/2)+\ve}=P^{-d(n-m-3)/2+\ve}
$$
otherwise, on taking $X\ll P^d$. This is $O(P^{-d/2+\ve})$ if 
$n-m\geq 4$.
All of our estimates are satisfactory for \eqref{eq:end}, which  therefore concludes  the proof of Theorem \ref{t:diag}.

\end{document}